\documentclass{amsart}

\usepackage{amssymb}
\usepackage[latin1]{inputenc}
\usepackage{graphicx}

\newtheorem{theorem}{Theorem}[section]
\newtheorem{lemma}[theorem]{Lemma}

\newtheorem{corollary}[theorem]{Corollary}
\newtheorem{proposition}[theorem]{Proposition}

\theoremstyle{definition}

\theoremstyle{remark}
\newtheorem{remark}[theorem]{Remark}

\numberwithin{equation}{section}

\begin{document}

\title{Analytic invariants for the $1:-1$ resonance}

\author{José Pedro Gaivão}
\address{Cemapre, Rua do Quelhas 6 1200-781 Lisboa Portugal}

\email{jpgaivao@iseg.utl.pt}
\thanks{}



\date{}

\dedicatory{}


\begin{abstract}
Associated to analytic Hamiltonian vector fields in $\mathbb{C}^4$ having an equilibrium point satisfying a non semisimple $1:-1$ resonance, we construct two universal constants that are invariant with respect to local analytic symplectic changes of coordinates. These invariants vanish when the Hamiltonian is integrable. We also prove that one of these invariants does not vanish on an open and dense set.  
\end{abstract}

\maketitle
 \tableofcontents 
\section{Introduction}
Let $X_H:(\mathbb{C}^4,0)\rightarrow (\mathbb{C}^4,0)$ be an analytic Hamiltonian vector field, i.e. there exists an analytic function $H:(\mathbb{C}^4,0)\rightarrow(\mathbb{C},0)$ called the Hamiltonian such that $\Omega(X_H,\mathbf{v})=\mathrm{d}H(\mathbf{v})$ for every $\mathbf{v}\in\mathbb{C}^4$ where $\Omega$ is a symplectic form in $\mathbb{C}^4$. For definiteness we assume that $\Omega$ is the standard symplectic form,
\begin{equation}\label{1:Omegadef}
\Omega(\mathbf{x},\mathbf{y})=\mathbf{x}^{T}J\mathbf{y},\quad \mathbf{x},\mathbf{y}\in\mathbb{C}^4,\quad\text{where}\quad J=\begin{pmatrix}0&\mathrm{Id}\\-\mathrm{Id}&0\end{pmatrix}.
\end{equation}
The matrix $J$ is known as the standard symplectic matrix. In this setting, the Hamiltonian vector field $X_H$ written in coordinates reads,
$$
X_H(\mathbf{q},\mathbf{p})=\left(\frac{\partial H}{\partial \mathbf{p}}(\mathbf{q},\mathbf{p}),-\frac{\partial H}{\partial \mathbf{q}}(\mathbf{q},\mathbf{p})\right),\quad (\mathbf{q},\mathbf{p})\in\mathbb{C}^2\times\mathbb{C}^2.
$$
In this paper we study a Hamiltonian vector field $X_{H}$ with an equilibrium point $X_H(0)=0$ in a $1:-1$ resonance, i.e. the matrix $DX_H(0)$ is not diagonalizable and has a pair of double imaginary eigenvalues $\pm i\alpha$, $\alpha>0$. 

Our study is motivated by the problem of estimating the size of the chaotic zone near a Hamiltonian-Hopf bifurcation \cite{GG:11,LLLB:04,PMRM:03}. This is a codimension one bifurcation of an equilibrium point in a two degrees of freedom Hamiltonian system in $\mathbb{R}^4$. 
More precisely, let $H_\epsilon$ be a real analytic family of Hamiltonian functions defined in a neighbourhood of the origin in $\mathbb{R}^4$. Suppose that the origin is an equilibrium point of $X_{H_\epsilon}$, i.e., $X_{H_\epsilon}(0)=0$ for every $\epsilon$, and that as $\epsilon\rightarrow0^{+}$ the equilibrium goes through a Hamiltonian-Hopf bifurcation: for $\epsilon>0$ the matrix $DX_{H_\epsilon}(0)$ has two pairs of complex conjugate eigenvalues $\pm\beta_\epsilon\pm i\alpha_\epsilon$, $\alpha_\epsilon,\beta_\epsilon>0$ that approach the imaginary axis as $\epsilon\rightarrow0^{+}$ yielding a pair of double imaginary eigenvalues $\pm i\alpha_0$, $\alpha_0>0$ for $DX_{H_0}(0)$. At the critical value $\epsilon=0$ the equilibrium is at a $1:-1$ resonance. This bifurcation has been extensively studied \cite{MJ:85} and it is known that there are two main bifurcation scenarios. In one of these scenarios, for $\epsilon>0$ there are two dimensional stable $W_\epsilon^s$ and unstable $W_\epsilon^u$ manifolds that live inside the three dimensional energy level set $\left\{H_\epsilon=H_\epsilon(0)\right\}$ and shrink to the equilibrium as the bifurcation parameter approaches the critical value. Points in the manifold $W_\epsilon^s$ (resp. $W_\epsilon^u$) converge to the equilibrium forward (resp. backward) in time under the action of the flow. The intersection $W_\epsilon^s\cap W^u_\epsilon$ if not empty consists of homoclinic orbits, thus is at least one-dimensional. It is well known that the existence of a transverse homoclinic orbit is a route to the onset of chaotic dynamics in a neighborhood of the equilibrium point \cite{RD:76,Le:00}. 

In \cite{GG:11} a quantity $\omega$ known as \textit{homoclinic invariant} was introduced to measure the size of the splitting of stable and unstable manifolds. Roughly speaking, it is defined to be the symplectic area formed by a pair of normalized tangent vectors at a homoclinic point. Let us show how to define it precisely. In a neighborhood of the equilibrium, the unstable manifold $W^{u}_\epsilon$ can be locally parameterized by a $C^1$ function,
$$
\mathbf{\Gamma}^u:\{(\varphi,z):\varphi\in \mathbb{T},z<z_0\}\to\mathbb R^4
$$
for some $z_0\in\mathbb{R}$ where $\mathbb{T}=\mathbb{R}/2\pi\mathbb{Z}$. Moreover, $\mathbf{\Gamma}^u$ is a solution of the nonlinear PDE,
\begin{equation}\label{2:PDE}
\alpha_\epsilon\partial_\varphi\mathbf{\Gamma}^u+\beta_\epsilon\partial_z\mathbf{\Gamma}^u=X_{H_\epsilon}(\mathbf{\Gamma}^u)\,,
\end{equation}
with the following asymptotic condition,
$$
\lim_{z\to-\infty}\mathbf{\Gamma}^u(\varphi,z)=0\,.
$$
Such parameterization is said to be a \textsl{natural parameterization} of $W^u_{\epsilon}$. Since it satisfies the PDE \eqref{2:PDE}, $\mathbf{\Gamma}^u$ conjugates the motion on the unstable manifold in a neighborhood of the equilibrium to the linear motion on the cylinder $\mathbb{T}\times (-\infty,z_0)$. That is,
\begin{equation}\label{2:EqGammau}
\mathbf{\Gamma}^u(\varphi+\alpha_\epsilon t,z+\beta_\epsilon t)=\Phi^t_{H_\epsilon}\circ \mathbf{\Gamma}^u(\varphi,z)\,,
\end{equation}
where $\Phi^t_{H_\epsilon}$ is the Hamiltonian flow. The derivatives $\partial_z\mathbf{\Gamma}^u$ and $\partial_\varphi\mathbf{\Gamma}^u$ define a basis of tangent vectors at each point of $W^u_\epsilon$. To obtain a natural parametrization for the stable manifold we can reverse the time and repeat the same reasoning, or equivalently consider $-H_\epsilon$. For simplicity, suppose that $X_{H_\epsilon}$ is time-reversible, i.e., $\mathcal{S}_{*}X_{H_\epsilon}=-X_{H_\epsilon}$, where $\mathcal{S}\neq\pm\mathrm{Id}$ is some linear involution. In the reversible setting it is convenient to define a local parameterization for the stable manifold as
\begin{equation*}
\mathbf{\Gamma}^s(\varphi,z):=\mathcal{S}\circ\mathbf{\Gamma}^u(-\varphi,-z),
\end{equation*}
which satisfies the same PDE \eqref{2:PDE}. The freedom in the definition of the parameterizations is reduced to translations in their arguments. Let $\mathrm{Fix}(\mathcal{S})$ denote the set of fixed points of the involution. Given an orbit $\gamma$ of the vector field $X_{H_\epsilon}$ we call it symmetric if $\gamma\cap\mathrm{Fix}(\mathcal{S})\neq\emptyset$. In \cite{GL:95} the existence of two primary symmetric homoclinic orbits is proved. Roughly, they correspond to the ``first intersection'' of both $W^{s,u}_\epsilon$ with $\mathrm{Fix}(\mathcal{S})$. Let $\gamma_h$ denote one these homoclinic orbits. Due to the freedom in the definition of the parameterizations we can suppose that $\gamma_h(t_0)=\mathbf{\Gamma}^{u}(\varphi_0,z_0)=\mathbf{\Gamma}^{s}(\varphi_0,z_0)$ for some $t_0\in\mathbb{R}$ and $(\varphi_0,z_0)\in\mathbb{T}\times\mathbb{R}$. The homoclinic invariant of $\gamma_h$ is defined in the following way,
$$
\omega=\Omega(\partial_\varphi\mathbf{\Gamma}^s(\varphi_0,z_0),\partial_\varphi\mathbf{\Gamma}^u(\varphi_0,z_0)). 
$$
Clearly, $\omega$ takes the same value along the homoclinic orbit $\gamma_h$. Moreover, if $\omega\neq0$ then $\gamma_h$ is a transverse homoclinic orbit. Thus, $\omega$ measures the splitting of the stable and unstable manifolds along the homoclinic orbit $\gamma_h$. Based on analytical and numerical evidence, in \cite{GG:11} it is conjectured that the homoclinic invariant has the following asymptotic expansion,
\begin{equation}\label{2:asymptoticformulaomega}
\omega \sim \pm e^{-\frac{\pi\alpha_\epsilon}{2\beta_\epsilon}}\sum_{k\geq0}\omega_k\epsilon^{k}\quad\text{as}\quad\epsilon\rightarrow 0^{+}\,.
\end{equation} 
The symbol $\sim$ in \eqref{2:asymptoticformulaomega} means that if we truncate the series then the error in the approximation is of the order of the first missing term. Recall that $\beta_\epsilon$ is the absolute value of the real part of the eigenvalues and that $\beta_\epsilon\rightarrow 0$  as $\epsilon\rightarrow 0^{+}$. Thus \eqref{2:asymptoticformulaomega} implies that $\omega$ is exponentially small with respect to $\epsilon$. The leading term $\omega_0$ in the asymptotic expansion is called the \textsl{splitting constant} since $\omega_0\neq0$ implies that $\omega\neq0$ for $\epsilon$ sufficiently small. The splitting constant is defined at the moment of bifurcation, i.e., it only depends on the Hamiltonian with a $1:-1$ resonance. Moreover, $\omega_0=2\sqrt{\left|\mathcal{K}\right|}$ where $\mathcal{K}$ is one of the invariants studied in the present paper. 

Proving \eqref{2:asymptoticformulaomega} is a highly non-trivial problem comparable to the problem of the splitting of the separatrices of the standard map that started with the work of V. Lazutkin \cite{VL:84} and ended with a complete proof given by V. Gelfreich in \cite{VG:99}. Based on the results of \cite{Ga:10} and on the results of the present paper the author has an unpublished proof of \eqref{2:asymptoticformulaomega} that will send for publication as a separate paper. 

Also related to this work is the study of the so-called inner equation \cite{BS:08,MSS:11,OSS:03}. In most problems of exponentially small splitting of separatrices, the leading constant of an asymptotic formula that measures the splitting comes from the study of an inner equation which, roughly speaking, contains the most singular behavior of the problem \cite{VG:97}. 

The study of exponentially small splitting of invariant manifolds in Hamiltonian systems of higher dimensions can be found in \cite{LMS:03,Sa:01}. In these works, the authors have devised a geometrical method to study the splitting of stable and unstable manifolds of a partially hyperbolic invariant torus (known as ``whiskered torus'') in near-integrable Hamiltonian systems. 

The combination of geometrical and analytical methods to study the exponentially small splitting of separatrices has proved fruitful and still today, it follows closely the original ideas introduced by V. Lazutkin in \cite{VL:84}.  

Finally, let us mention that the invariants found in this paper have a parallel to the analytic invariants found in \cite{GN:08} which are defined for diffeomorphisms in $\mathbb{C}^2$ with a parabolic fixed point. One of these invariants also plays a role in the splitting of separatrices near a saddle-center bifurcation \cite{Ge:00}. In particular, for the Hénon map the same study was carried out in \cite{GS:01} where a connection with the resurgent theory of J. Écalle was established. For a more recent treatment on the connection between resurgence and splitting of separatrices the reader is referred to \cite{Sa:06}. See also \cite{E:92,MR:82,So:96} for related studies in analytic classification of germs of vector fields.

To conclude the introduction let us outline the structure of the rest of paper. In Section \ref{2:preliminaries} we setup the problem and recall some well known facts about normal forms. The main results of this paper are presented in Section \ref{2:mainresults}. In Section \ref{2:sectionformalseries} we construct formal solutions of certain differential equations. Section \ref{2:sectionlinearop} develops a theory to invert a type of linear operators. In Section \ref{2:sectionVariationalEq} we study a variational equation and Sections \ref{2:sectionunstableseparatrix} and \ref{2:theoremasympdelta} contain the proofs of our main results. 

\section{Preliminaries}\label{2:preliminaries}
Let $X_H$ be defined as in the introduction. The well known normal form theory for quadratic Hamiltonians \cite{NBRC:74} provides a symplectic linear change of variables that transforms the quadratic part of the Hamiltonian $H$ into the following normal form,  
\begin{equation*}
H(\mathbf{q},\mathbf{p})=-\alpha\left(q_2 p_1 - q_1 p_2\right) + \frac{\iota}{2}\left(q_1^2 + q_2^2\right)+ \mathrm{high}\ \mathrm{order}\ \mathrm{terms},
\end{equation*}
where $\mathbf{q}=(q_1,q_2)$, $\mathbf{p}=(p_1,p_2)$, $\iota^2 = 1$ and $\alpha>0$. Without lost of generality we can assume that $\alpha = 1$ and $\iota=1$. Indeed, by a re-parametrization of time or equivalently by scaling the Hamiltonian $H$ by $\iota\alpha^{-1}$ and performing the symplectic linear change of variables,
\begin{equation*}
(q_1,q_2,p_1,p_2)\mapsto\left(\iota\frac{\alpha}{\sqrt{\alpha}}q_1,\sqrt{\alpha}q_2,\iota\frac{\sqrt{\alpha}}{\alpha}p_1,\frac{1}{\sqrt{\alpha}}p_2\right),
\end{equation*}
we obtain the desired normalization of $\alpha$ and $\iota$. It is also possible to normalize the higher order terms of $H$. The normal form of $H$ is attributed to Sokol'ski{\u\i} who derive it when studying the formal stability of $H$. \begin{theorem}[Sokol'ski{\u\i} \cite{So:74}]\label{1:TheoremNormalform11}
There is a formal near identity symplectic change of coordinates $\Phi$ such that,
\begin{equation*}
H^{\sharp}=H\circ\Phi=-I_1+I_2+ \sum_{l+k\geq 2}a_{l,k}I_1^lI_3^k,
\end{equation*}
where
\begin{equation}\label{2:defI1I2I3}
I_1=q_2p_1-q_1p_2,\quad I_2=\frac{q_1^2+q_2^2}{2},\quad I_3=\frac{p_1^2+p_2^2}{2}\,.
\end{equation}
The normal form coefficients $a_{l,k}\in\mathbb{C}$ are uniquely defined, forming an infinite set of invariants for the Hamiltonian $H$.
\end{theorem}
The normal form $H^{\sharp}$ is obtained inductively by constructing a near identity symplectic changes of variables that normalizes each order of $H$ at a time without affecting the previous orders. Moreover, it is constructed in such a way that has an additional $\mathbb{S}^1$ symmetry induced by the integral of motion $I_1$, i.e. $\Omega(X_{H^{\sharp}},X_{I_1})=0$. There is a convenient way of rewriting the normal form that takes into account the different contributions of the higher order monomials. More precisely, we define a new order for a monomial in $\mathbb{C}[q_1,q_2,p_1,p_2]$: for $i=1,2$ we let $q_i$ have order $2$ while $p_i$ have order $1$. For example, using this new ordering we say that the monomial $p_1p_2$ has order $2$ while $q_1p_2$ has order $3$. Reordering the terms of $H^{\sharp}$ according to this new order we get,
\begin{equation*}
H^{\sharp}=H\circ\Phi=-I_1+I_2+\eta I_3^2+ \sum_{3l+2k\geq 5}a_{l,k}I_1^lI_3^k,
\end{equation*}
where the coefficient $\eta$ is equal to $a_{2,0}$. In general, the limit of the normal form procedure produces a normal form transformation $\Phi$ that is divergent. However the normal form is rather useful and can be used to approximate at any order the original $H$ by an integrable one. Thus we can assume that $H$ is in the general form,
\begin{equation}\label{2:H}
H=-I_1 + I_2+\eta I_3^2 + F,
\end{equation}
where $\eta\in\mathbb{C}$ and $F:\mathcal{U}\rightarrow\mathbb{C}$ is a bounded analytic function defined in an open neighborhood $\mathcal{U}$ of the origin in $\mathbb{C}^4$ and containing monomials of order greater or equal than $5$. 

In the real analytic setting, the normal form coefficients are real and $\eta$ determines the stability type of the equilibrium of $X_H$. According to \cite{Le:09}, when $\eta>0$ the equilibrium is Lyapunov stable and it becomes unstable when $\eta<0$. The degenerate case corresponds to $\eta=0$. 

Throughout this paper we will only consider the case of a non-degenerate elliptic equilibrium,
\begin{equation}\label{2:nondegeneratecondition}
\eta\neq0.
\end{equation}
This is a generic condition. In the degenerate case, one has to include in the leading order \eqref{2:H} the next term $a_{0,k}I_3^{k}$ of the normal form for which $a_{0,k}\neq0$.  

Although the equilibrium point of $X_H$ is elliptic, we will show that it has a stable (resp. unstable) immersed complex manifold by constructing a stable (resp. unstable) parametrization $\mathbf{\Gamma}^{+}(\varphi,\tau)$ (resp. $\mathbf{\Gamma}^{-}(\varphi,\tau)$) defined in certain regions of $\mathbb{C}^2$, with some prescribed asymptotics at infinity and satisfying the nonlinear PDE:
\begin{equation}\label{2:DH}
\mathcal{D}\mathbf{\Gamma}^{\pm}=X_{H}(\mathbf{\Gamma}^{\pm}),\qquad\mathrm{where}\quad \mathcal{D}=\partial_\varphi+\partial_\tau.
\end{equation}
In a common domain of intersection, the stable and unstable parameterizations are described by a single asymptotic expansion, implying that their difference is beyond all algebraic orders. We will obtain a refined estimate for the difference of parameterizations and prove that it has an asymptotic expansion with an exponentially small prefactor. Moreover, in the four dimensional space $\mathbb{C}^4$ the difference of the parameterizations can be locally described by four constants that can be used to define two local analytic invariants for the Hamiltonian $H$. 


Let us precisely state our results.

\section{Main results}\label{2:mainresults}
\subsection{Parameterizations}
First we will study formal solutions of equation \eqref{2:DH}. Denote by $\mathsf{T}$ the space of trigonometric polynomials with complex coefficients, i.e., the space of functions of the form,
\begin{equation*}
a_0+\sum_{k=1}^na_k\cos(k\varphi)+\sum_{k=1}^nb_k\sin(k\varphi),\quad a_k,b_k\in\mathbb{C},\, n\in\mathbb{N}_0.
\end{equation*}
We solve equation \eqref{2:DH} in the space of formal power series $\mathsf{T}^4[[\tau^{-1}]]$, i.e., we substitute a series into the equation, collect coefficients at each order of $\tau^{-1}$ in both sides and then solve an infinite system of equations in $\mathsf{T}$. Then we obtain the following result.
\begin{theorem}[Formal Separatrix]\label{2:Tformalseparatrix}
Equation \eqref{2:DH} has a non-zero formal solution $\hat{\mathbf{\Gamma}}$ having the form,
\begin{equation}\label{2:formhatgamma}
\hat{\mathbf{\Gamma}}(\varphi,\tau) =\begin{pmatrix}\tau^{-2}\hat{\Gamma}_1(\varphi,\tau)\\\tau^{-2}\hat{\Gamma}_2(\varphi,\tau)\\\tau^{-1}\hat{\Gamma}_3(\varphi,\tau)\\\tau^{-1}\hat{\Gamma}_4(\varphi,\tau)\end{pmatrix},\quad\text{where}\quad \hat{\Gamma}_i\in\mathsf{T}[[\tau^{-1}]],\quad i=1,\ldots,4,
\end{equation}
with the leading orders,
\begin{equation*}\begin{split}
\hat{\Gamma}_1(\varphi,\tau)=\kappa\cos\varphi +\frac{\kappa a_{1,1}}{\eta}\tau^{-1}+ \cdots&\quad \hat{\Gamma}_2(\varphi,\tau)=\kappa\sin\varphi -\frac{\kappa a_{1,1}}{\eta}\tau^{-1}+ \cdots\\
\hat{\Gamma}_3(\varphi,\tau)=\kappa\cos\varphi +\frac{\kappa a_{1,1}}{\eta}\tau^{-1}+ \cdots&\quad \hat{\Gamma}_4(\varphi,\tau)=\kappa\sin\varphi -\frac{\kappa a_{1,1}}{\eta}\tau^{-1}+ \cdots
\end{split}
\end{equation*}
where $\kappa^2=-\frac{2}{\eta}$ and the ellipsis mean higher order terms in $\tau^{-1}$.
Moreover, for any other non-zero formal solution $\hat{\tilde{\mathbf{\Gamma}}}$ of \eqref{2:DH} having the same form \eqref{2:formhatgamma} there exist $(\varphi_0,\tau_0) \in \mathbb{C}^2$ such that $\hat{\tilde{\mathbf{\Gamma}}}(\varphi,\tau)= \hat{\mathbf{\Gamma}}(\varphi+\varphi_0,\tau+\tau_0)$. 
\end{theorem}
This theorem is proved in Section \ref{2:sectionformalseries}. We call $\hat{\mathbf{\Gamma}}$ a \textit{formal separatrix}. In general, these formal series do not converge (see Corollary \ref{2:corollaryintegrability}). According to the previous theorem, the freedom in the choice of formal solutions is given by translations in the $(\varphi,z)$-plane. We can eliminate this freedom by fixing the first two coefficients of the formal series $\hat{\Gamma}_i$. This freedom can not be eliminated in a coordinate independent way, unless the Hamiltonian vector field has some extra properties, such as being time-reversible (see Remark \ref{2:remarkseparatrixreversible}).
 
In the following we construct analytic solutions of equation \eqref{2:DH} with prescribed asymptotics $\hat{\mathbf{\Gamma}}$ in certain regions of $\mathbb{C}^2$. Fix $h>0$ and let
$$
\mathbb{T}_h = \left\{\varphi \in \mathbb{C}/2\pi\mathbb{Z}\:\colon\: \left|\mathrm{Im}\,\varphi\right|<h \right\}.
$$ 
In order to state our results we need to introduce the notion of asymptotic expansion. Let $X$ be a subset of $\mathbb{C}$ that contains a limit point $a$, possibly the point at infinity. A sequence of functions $\left\{\xi_n\right\}_{n\in\mathbb{N}}$ defined in $X$ and taking values in $\mathbb{C}$ is called as \textit{asymptotic sequence} as $\tau\rightarrow a$ if none of the functions $\xi_n$ vanish in a neighborhood of $a$ (except the point $a$) and if for every $n\in\mathbb{N}$ we have,
$$
\lim_{x\rightarrow a}\frac{\xi_{n+1}(\tau)}{\xi_n(\tau)}=0.
$$
For example, $\left\{\tau^{-n}\right\}_{n\in\mathbb{N}}$ is an asymptotic sequence as $\tau\rightarrow \infty$. Given two functions $f,g:\mathbb{T}_h\times X\rightarrow\mathbb{C}$ we shall frequently use the big-O notation $f=O(g)$ meaning that there exists a constant $C>0$ such that $\left|f(\varphi,\tau)\right|\leq C \left|g(\varphi,\tau)\right|$ for all $(\varphi,\tau)\in \mathbb{T}_h\times X$ or we write $f=O(g)$ as $(\tau\rightarrow a)$ meaning that there exists a constant $C>0$ and a neighborhood $U$ of $a$ such that $\left|f(\varphi,\tau)\right|\leq C \left|g(\varphi,\tau)\right|$ for all $(\varphi,\tau)\in \mathbb{T}_h\times (X\cap U)$. Finally, given a function $f:\mathbb{T}_h\times X\rightarrow \mathbb{C}$ we say that it has an \textit{asymptotic expansion} with respect to the asymptotic sequence $\left\{\xi_n\right\}$ and write,
$$
f(\varphi,\tau)\sim \sum_{n\geq1}c_n(\varphi)\xi_n(\tau),
$$
if for every $N\in\mathbb{N}$ the following holds,
$$
f(\varphi,\tau)-\sum_{n=1}^N c_n(\varphi)\xi_n(\tau)=O(\xi_{N+1}(\tau))\quad\text{as}\quad (\tau\rightarrow a).
$$
It is easy to see that the asymptotic expansion of $f$ is unique. Moreover, the definition of the big-O notation and of asymptotic expansion easily extends to functions taking values in $\mathbb{C}^k$ for any $k\in\mathbb{N}$.

Given $r>0$ and $0<\theta<\frac{\pi}{4}$ consider the following sector,
\begin{equation}\label{2:defDrminus}
D^{-}_r = \left\{\tau \in \mathbb{C}\:\colon\: \left| \arg{(\tau+r)} \right| > \theta\right\},
\end{equation}
which can be visualized in Figure \ref{1:Dr}.
\begin{figure}[t]
  \begin{center}
    \includegraphics[width=4in]{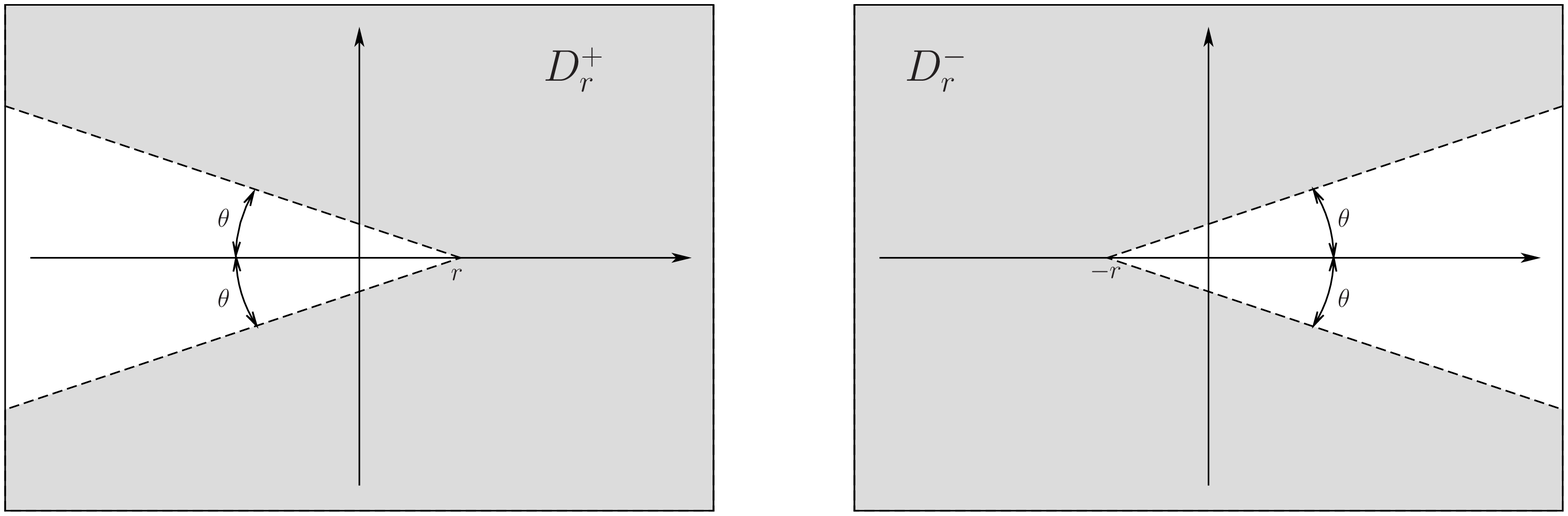}
  \end{center}
  \caption{Domains $D_r^{\pm}$.} 
  \label{1:Dr}
\end{figure}
We shall leave the parameters $\theta$ and $h$ fixed throughout this paper. The next theorem gives the existence of an analytic solution of equation \eqref{2:DH} having the formal separatrix as an asymptotic expansion in the sector $D^{-}_r$. The proof of the theorem can be found in Section \ref{2:sectionunstableseparatrix}.
\begin{theorem}[Unstable Parameterization]\label{2:theoremunstableseparatrix}
Given a formal separatrix $\hat{\mathbf{\Gamma}}$ there exist $r_{-}>0$ and a unique analytic function $\mathbf{\Gamma}^{-}:\mathbb{T}_h\times D^{-}_{r_{-}}\rightarrow\mathbb{C}^4$ solving equation \eqref{2:DH} such that $\mathbf{\Gamma}^{-}(\varphi,\tau)\sim\hat{\mathbf{\Gamma}}(\varphi,\tau)$ as $\tau\rightarrow\infty$ in $D^{-}_{r_{-}}$.
\end{theorem}
It follows from the asymptotics of $\mathbf{\Gamma}^{-}$ that for $r>0$ sufficiently large the set $\mathbf{\Gamma}^{-}(\mathbb{T}_h\times D^{-}_{r})$ is a two dimensional immersed complex manifold. Points in this manifold converge to the equilibrium under the flow, i.e. $\Phi_H^t(\mathbf{\Gamma}^{-}(\varphi,\tau))\rightarrow 0$ as $\mathrm{Re}\,t\rightarrow -\infty$. Thus $\mathbf{\Gamma}^{-}$ is an analytic parameterization of a local unstable manifold of the equilibrium of $X_H$. An analogous result is valid for the stable manifold. More precisely, for $r>0$ let $D_r^{+}$ be the symmetric sector,
$$
D^{+}_r=\left\{\tau\in\mathbb{C}\:|\: -\tau\in D^{-}_r\right\}.
$$ 
By properly modifying the arguments in the proof of Theorem \ref{2:theoremunstableseparatrix}, we can prove that given a formal separatrix $\hat{\mathbf{\Gamma}}$ there exist $r_{+}>0$ and an analytic function $\mathbf{\Gamma}^{+}:\mathbb{T}_h\times D^{+}_{r_{+}}\rightarrow\mathbb{C}^4$ solving the same equation \eqref{2:DH} such that $\mathbf{\Gamma}^{+}(\varphi,\tau)\sim\hat{\mathbf{\Gamma}}(\varphi,\tau)$ as $\tau\rightarrow\infty$ in $D^{+}_{r_{+}}$. 
\subsection{The difference $\mathbf{\Gamma}^{+}-\mathbf{\Gamma}^{-}$}
Therefore, equation \eqref{2:DH} has two analytic solutions $\mathbf{\Gamma}^{\pm}$ both defined in symmetric sectors $D_r^{\pm}$ for $r=\max\left\{r_{-},r_{+}\right\}$ whose intersection in the $\tau$-plane consists of two connected components (see Figure \ref{2:Drwedge}). Since both functions have the same asymptotic expansion $\hat{\mathbf{\Gamma}}$ then, 
$$
\mathbf{\Gamma}^{+}(\varphi,\tau)-\mathbf{\Gamma}^{-}(\varphi,\tau)\sim 0\quad\text{as}\quad\tau\rightarrow\infty\quad\text{in}\quad D^{+}_r\cap D^{-}_r.
$$ 
\begin{figure}[t]
  \begin{center}
    \includegraphics[width=2in]{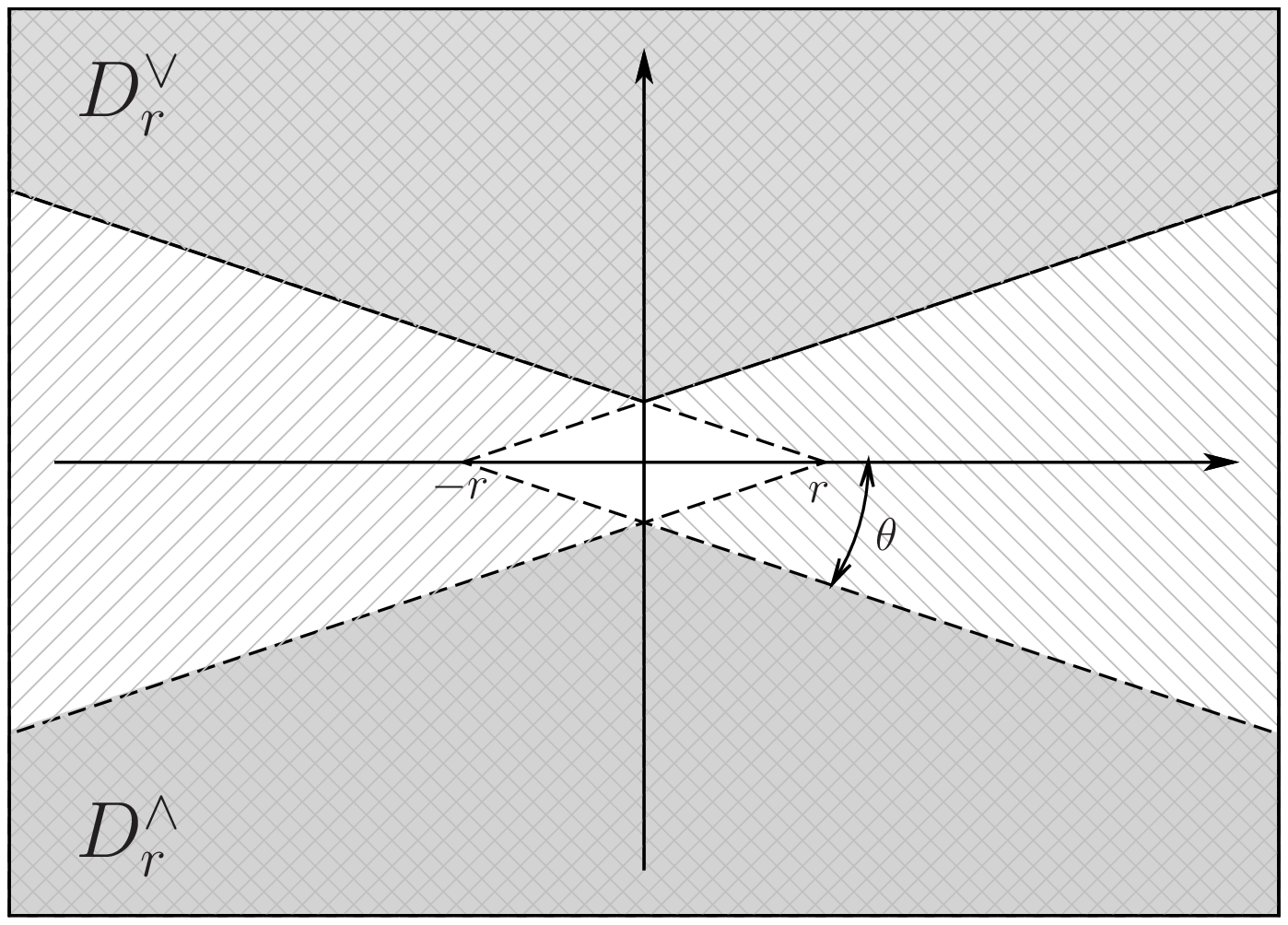}
  \end{center}
  \caption[Domains $D_r^{\wedge}$ and $D_r^{\vee}$.]{The intersection of the domains $D_{r}^{\pm}$.} 
  \label{2:Drwedge}
\end{figure}
Thus, their difference is said to be \textsl{beyond all algebraic orders}. We shall obtain a more precise estimate for the difference of the parameterizations on the lower component of the set $D^{+}_r\cap D^{-}_r$ which we denote by $D_{r}^{\wedge}$. Similar considerations work for the upper connected component $D_{r}^{\vee}$. In order to obtain such estimate we will use the fact that $\mathbf{\Gamma}^{+}-\mathbf{\Gamma}^{-}$ is approximately a solution of the variational equation of $X_{H}$ along the unstable solution $\mathbf{\Gamma}^{-}$. Therefore, we study the analytic solutions of the variational equation,
\begin{equation}\label{2:VariationaleqGammaminus}
\mathcal{D}\mathbf{u}=DX_H(\mathbf{\Gamma}^{-}(\varphi,\tau))\mathbf{u}.
\end{equation}
Since both $\partial_\varphi\mathbf{\Gamma}^{-}$ and $\partial_\tau\mathbf{\Gamma}^{-}$ solve equation \eqref{2:VariationaleqGammaminus} we shall construct a matrix solution $\mathbf{U}$ of equation \eqref{2:VariationaleqGammaminus} satisfying the following properties:
\begin{enumerate}
	\item The matrix-valued function $\mathbf{U}:\mathbb{T}_h\times D_r^{-}\rightarrow\mathbb{C}^{4\times 4}$ is analytic and continuous on the closure of its domain.
	\item The third and fourth columns of $\mathbf{U}$ are the known solutions $\partial_\varphi\mathbf{\Gamma}^{-}$ and $\partial_\tau\mathbf{\Gamma}^{-}$ respectively.
	\item $\mathbf{U}$ is symplectic, i.e. $\mathbf{U}^TJ\mathbf{U}=J$ where $J$ is the standard symplectic matrix \eqref{1:Omegadef}. 
\end{enumerate}
A matrix $\mathbf{U}$ satisfying the above conditions is said to be a \textsl{normalized fundamental solution} of equation \eqref{2:VariationaleqGammaminus}. We will also construct asymptotic expansions for these fundamental solutions as formal solutions of the formal variational equation,
\begin{equation}\label{2:formalVariationaleqGammaminus}
\mathcal{D}\mathbf{u}=DX_H(\hat{\mathbf{\Gamma}}(\varphi,\tau))\mathbf{u},
\end{equation}
where $\hat{\mathbf{\Gamma}}$ is a formal separatrix. The existence of such formal solutions is provided by the next proposition whose proof can be found in Section \ref{2:sectionformalseries}. 
\begin{proposition}\label{2:Thformalnormalizedfundmatrix}
Given a formal separatrix $\hat{\mathbf{\Gamma}}$, the corresponding formal variational equation \eqref{2:formalVariationaleqGammaminus} has a formal fundamental solution $\hat{\mathbf{U}}$ of the following form,
\begin{equation*}\hat{\mathbf{U}}=\begin{pmatrix}
\tau^{1}\hat{u}_{1,1}&\tau^{2}\hat{u}_{1,2}&\tau^{-2}\hat{u}_{1,3}&\tau^{-3}\hat{u}_{1,4}\\
\tau^{1}\hat{u}_{2,1}&\tau^{2}\hat{u}_{2,2}&\tau^{-2}\hat{u}_{2,3}&\tau^{-3}\hat{u}_{2,4}\\
\tau^{2}\hat{u}_{3,1}&\tau^{3}\hat{u}_{3,2}&\tau^{-1}\hat{u}_{3,3}&\tau^{-2}\hat{u}_{3,4}\\
\tau^{2}\hat{u}_{4,1}&\tau^{3}\hat{u}_{4,2}&\tau^{-1}\hat{u}_{4,3}&\tau^{-2}\hat{u}_{4,4}
\end{pmatrix},
\end{equation*}
where $\hat{u}_{i,j}\in\mathsf{T}[[\tau^{-1}]]$, for $i,j=1,\ldots,4$ such that the third and fourth columns of $\hat{\mathbf{U}}$ are $\partial_\varphi\hat{\mathbf{\Gamma}}$ and $\partial_\tau\hat{\mathbf{\Gamma}}$ respectively and $\hat{\mathbf{U}}^TJ\hat{\mathbf{U}}=J$. Moreover for any other formal fundamental solution $\hat{\tilde{\mathbf{U}}}$ of the same form of $\hat{\mathbf{U}}$ there exists $C\in \mathbb{C}^{2\times 2}$ symmetric matrix ($C^T=C$) such that $\hat{\tilde{\mathbf{U}}}=\hat{\mathbf{U}}E_C$ where,
\begin{equation}\label{2:matrixE}
E_C=\begin{pmatrix}
\mathrm{Id}&0\\
C&\mathrm{Id}\\
\end{pmatrix}.
\end{equation}  
\end{proposition}
The existence of a normalized fundamental solution of equation \eqref{2:VariationaleqGammaminus} with asymptotic expansion $\hat{\mathbf{U}}$ is given by the following proposition whose proof is placed in Section \ref{2:sectionVariationalEq}.
\begin{proposition}\label{2:normalizedFundamentalMatrix}
Given an unstable parameterization $\mathbf{\Gamma}^{-}\sim\hat{\mathbf{\Gamma}}$ and a formal fundamental solution $\hat{\mathbf{U}}$ there exists $r>0$ such that the variational equation \eqref{2:VariationaleqGammaminus} has an unique normalized fundamental solution $\mathbf{U}:\mathbb{T}_h\times D_{r}^{-}\rightarrow\mathbb{C}^{4\times 4}$ such that $\mathbf{U}\sim\hat{\mathbf{U}}$ as $\tau\rightarrow\infty$ in $D^{-}_{r}$.
\end{proposition}
Using these fundamental solutions for the variational equation \eqref{2:VariationaleqGammaminus} we obtain an exponentially small estimate for the difference of stable and unstable parameterizations.
\begin{theorem}\label{2:theoremdeltaasymptoticformula}
Given $\epsilon>0$ and a normalized fundamental solution $\mathbf{U}$ there exists a vector $\mathbf{\Theta}^{-} \in \mathbb{C}^4$ such that the following asymptotic formula holds,
\begin{equation}\label{2:asymptoticformuladeltatheorem}
\mathbf{\Gamma}^{+}(\varphi,\tau)-\mathbf{\Gamma}^{-}(\varphi,\tau)= e^{-i(\tau-\varphi)}\mathbf{U}(\varphi,\tau)\mathbf{\Theta}^{-} + O(e^{-(2-\epsilon)i(\tau-\varphi)}),
\end{equation}
as $\tau\rightarrow\infty$ in $D_r^{\wedge}$.
\end{theorem}
We prove this theorem is Section \ref{2:theoremasympdelta}. As an immediate corollary of Theorem \ref{2:theoremdeltaasymptoticformula} and taking into account the asymptotic expansion of $\mathbf{U}$ we obtain the following asymptotic expansion for the difference,
\begin{equation*}
e^{i(\tau-\varphi)}\left(\mathbf{\Gamma}^{+}(\varphi,\tau)-\mathbf{\Gamma}^{-}(\varphi,\tau)\right)\sim \hat{\mathbf{U}}(\varphi,\tau)\Theta^{-}\quad\text{as}\quad\tau\rightarrow\infty\quad\text{in}\quad D_{r}^{\wedge}.
\end{equation*}
Using the leading orders of $\hat{\mathbf{U}}$ (see Proposition \ref{2:Thformalnormalizedfundmatrix}) it is possible to obtain an exponentially small upper bound for the difference of stable and unstable parameterizations in the lower connected component $D_r^{\wedge}$. Indeed, since for every $\tau\in D_r^{\wedge}$ and $\sigma>0$ the vertical segment $[\tau,\tau-i\sigma]$ is contained in $D_r^{\wedge}$ then there exists $C>0$ such that for every $\sigma>0$,
\begin{equation*}
\left|\mathbf{\Gamma}^{+}(\varphi,\tau)-\mathbf{\Gamma}^{-}(\varphi,\tau)\right|\leq C \sigma^3 e^{-\sigma},
\end{equation*}
for all $\varphi\in \mathbb{T}_h$ and $\tau\in D_r^{\wedge}$ with $\mathrm{Im}\tau<-\sigma$.

As mentioned before, it is possible to use the previous arguments \textsl{mutatis mutandis} to study the difference $\mathbf{\Gamma}^{+}-\mathbf{\Gamma}^{-}$ in the upper connected component $D^{\vee}_r$. Similar to Theorem \ref{2:theoremdeltaasymptoticformula} one can prove the existence of $\Theta^+\in\mathbb{C}^4$ such that,
\begin{equation*}
e^{-i(\tau-\varphi)}\left(\mathbf{\Gamma}^{+}(\varphi,\tau)-\mathbf{\Gamma}^{-}(\varphi,\tau)\right)\sim \hat{\mathbf{U}}(\varphi,\tau)\Theta^{+}\quad\text{as}\quad\tau\rightarrow\infty\quad\text{in}\quad D_{r}^{\vee}.
\end{equation*}
\subsection{Analytic invariants}
In this section we use the asymptotic formula of Theorem \ref{2:theoremdeltaasymptoticformula} to construct two analytic invariants for the Hamiltonian $H$. On of these invariants measures the splitting distance of the complex manifolds parametrized by $\mathbf{\Gamma}^{\pm}$. This invariant is also related to the Stokes phenomenon which is observed in solutions of certain differential equations where the same solution possesses different asymptotic expansions at infinity in different sectors of the complex plane \cite{CKA:98}. 

In order to define these invariants, let $\mathbf{\Gamma}^{\pm}\sim \hat{\mathbf{\Gamma}}$ be a stable and unstable parameterization and $\mathbf{U}\sim\hat{\mathbf{U}}$ a normalized fundamental solution of the variational equation around $\mathbf{\Gamma}^{-}$. Moreover, let $$\Delta(\varphi,\tau)=\mathbf{\Gamma}^{+}(\varphi,\tau)-\mathbf{\Gamma}^{-}(\varphi,\tau)\,.$$ According to Theorem \ref{2:theoremdeltaasymptoticformula} we have the following asymptotics,
\begin{equation}\label{2:asymptoticexpdelta}
e^{\mp i(\tau-\varphi)}\Delta(\varphi,\tau)\sim \hat{\mathbf{U}}(\varphi,\tau)\Theta^{\pm}\quad\text{as}\quad\mathrm{Im}\tau\rightarrow\pm \infty.
\end{equation}
We call the first two components of $\Theta^{\pm}=\left(\Theta^{\pm}_1,\Theta^{\pm}_2,\Theta^{\pm}_3,\Theta^{\pm}_4\right)$ the \textsl{normal components} and the last two the \textit{tangent components}. The following limit provides a way to compute the components of the vectors $\Theta^{\pm}$,
\begin{equation}\label{2:limitTheta}
\Omega(\mathbf{\Theta}^{\pm},\mathbf{v})=\lim_{\mathrm{Im}\tau\rightarrow\pm\infty}\Omega(\Delta(\varphi,\tau),\mathbf{U}(\varphi,\tau)\mathbf{v})e^{\mp i(\tau-\varphi)},\quad \mathbf{v}\in\mathbb{C}^4,
\end{equation}
where $\Omega$ is the standard symplectic form and the convergence of the limit is uniform with respect to $\varphi\in \mathbb{T}_h$. The proof of \eqref{2:limitTheta} is straightforward. Indeed, it follows from the asymptotics \eqref{2:asymptoticexpdelta} and the fact that $\hat{\mathbf{U}}^{T}J\hat{\mathbf{U}}=J$. Moreover, the previous formula is useful from the computational point of view, since to compute the normal components of $\Theta^{\pm}$ it only requires knowing the stable and unstable parameterizations $\mathbf{\Gamma}^{\pm}$. In fact $\Theta_1^{\pm}=\Omega(\Theta^{\pm},e_3)$ where $e_3=(0,0,1,0)$. Since $\mathbf{U}e_3=\partial_\varphi\mathbf{\Gamma}^{-}$ we conclude that,
\begin{equation}\label{2:limitdefTheta1}
\Theta_1^{\pm}=\lim_{\mathrm{Im}\tau\rightarrow\pm\infty}\Omega(\Delta(\varphi,\tau),\partial_\varphi\mathbf{\Gamma}^{-}(\varphi,\tau))e^{\mp i(\tau-\varphi)}.
\end{equation}
A similar formula is valid for the normal component $\Theta^{\pm}_2$, where the tangent vector field $\partial_\varphi\mathbf{\Gamma}^{-}$ is replaced by $\partial_\tau\mathbf{\Gamma}^{-}$. The components of the vector $\Theta^{\pm}$ are not independent and due to the freedom in the choice of the parameterizations they are not uniquely defined. 

\begin{lemma}\label{2:lemmatranslationTheta}
Given any stable (resp. unstable) parameterizations $\mathbf{\Gamma}^\pm\sim\hat{\mathbf{\Gamma}}$ and normalized fundamental solution $\mathbf{U}\sim\hat{\mathbf{U}}$, the following holds:
\begin{enumerate}
	\item\label{2:lemmaitem1} $\Theta^{\pm}_1+\Theta^{\pm}_2=0$.
	\item\label{2:lemmaitem2} If $\tilde{\mathbf{\Gamma}}^\pm\sim\hat{\tilde{\mathbf{\Gamma}}}$ is another stable (resp. unstable) parameterization with normalized fundamental solution $\tilde{\mathbf{U}}\sim\hat{\tilde{\mathbf{U}}}$ then there exist $(\varphi_0,\tau_0)\in\mathbb{C}^2$ and a symmetric matrix $C\in\mathbb{C}^{2\times 2}$ such that $$\tilde{\Theta}^{\pm}=E_C\Theta^{\pm}e^{\pm i(\tau_0-\varphi_0)}$$
\end{enumerate}
\end{lemma}

\begin{proof}
To prove item \eqref{2:lemmaitem1} it is enough to show the equality for the $-$ case, since the $+$ case is completely analogous. 

Note that \eqref{2:asymptoticexpdelta} implies,
\begin{equation*}
H(\mathbf{\Gamma}^{+}(\varphi,\tau))=H(\mathbf{\Gamma}^{-}(\varphi,\tau))+\nabla H(\mathbf{\Gamma}^{-}(\varphi,\tau))\Delta(\varphi,\tau)+O(e^{-(2-\epsilon)i(\tau-\varphi)}),
\end{equation*}
as $\mathrm{Im}\tau\rightarrow -\infty$ for some $\epsilon>0$ arbitrarily small. Due to the conservation of energy we have that $H(\mathbf{\Gamma}^{\pm}(\varphi,\tau))=0$. Thus,
\begin{equation}\label{2:limitnablaHgammaminus}
\lim_{\mathrm{Im}\tau\rightarrow-\infty}\nabla H(\mathbf{\Gamma}^{-}(\varphi,\tau))\Delta(\varphi,\tau)e^{i(\tau-\varphi)}=0.
\end{equation}
Moreover,
\begin{equation*}\begin{split}
\nabla H(\mathbf{\Gamma}^{-})\Delta&=\Omega(X_H(\mathbf{\Gamma}^{-}),\Delta)\\
&=\Omega(\mathcal{D}\mathbf{\Gamma}^{-},\Delta)\\
&=-\Omega(\Delta,\partial_\varphi\mathbf{\Gamma}^{-})-\Omega(\Delta,\partial_\tau\mathbf{\Gamma}^{-}).
\end{split}
\end{equation*}
Thus, \eqref{2:limitnablaHgammaminus} implies that,
\begin{equation*}
\lim_{\mathrm{Im}\tau\rightarrow-\infty}\left(\Omega(\Delta(\varphi,\tau),\partial_\varphi\mathbf{\Gamma}^{-}(\varphi,\tau))+\Omega(\Delta(\varphi,\tau),\partial_\tau\mathbf{\Gamma}^{-}(\varphi,\tau))\right)e^{i(\tau-\varphi)}=0
\end{equation*}
which proves the desired equality. 

To prove item \eqref{2:lemmaitem2}, let $\tilde{\Delta}=\tilde{\mathbf{\Gamma}}^{+}-\tilde{\mathbf{\Gamma}}^{-}$. Similar to \eqref{2:asymptoticexpdelta} there exists $\tilde{\Theta}^{\pm}\in\mathbb{C}^4$ such that,
\begin{equation}\label{2:aympformuladeltatilde}
e^{\mp i(\tau-\varphi)}\tilde{\Delta}(\varphi,\tau)\sim \hat{\tilde{\mathbf{U}}}(\varphi,\tau)\tilde{\Theta}^{\pm}\quad\text{as}\quad\mathrm{Im}\tau\rightarrow\pm \infty.
\end{equation}
According to Theorem \ref{2:Tformalseparatrix} there exists $(\varphi_0,\tau_0)\in\mathbb{C}^2$ such that $\hat{\tilde{\mathbf{\Gamma}}}(\varphi,\tau)=\hat{\mathbf{\Gamma}}(\varphi+\varphi_0,\tau+\tau_0)$. Thus, the uniqueness of solutions in Theorem \ref{2:theoremunstableseparatrix} implies that $\tilde{\mathbf{\Gamma}}(\varphi,\tau)=\mathbf{\Gamma}(\varphi+\varphi_0,\tau+\tau_0)$. Moreover, since $\hat{\mathbf{U}}(\varphi+\varphi_0,\tau+\tau_0)$ is a formal normalized fundamental solution of the formal variational equation around $\hat{\tilde{\mathbf{\Gamma}}}$, then by Proposition \ref{2:Thformalnormalizedfundmatrix} there exists a $2\times 2$ symmetric matrix $C$ such that $\hat{\tilde{\mathbf{U}}}(\varphi,\tau)=\hat{\mathbf{U}}(\varphi+\varphi_0,\tau+\tau_0)E_C$. Again, by uniqueness of solutions in Proposition \ref{2:normalizedFundamentalMatrix} we conclude that $\tilde{\mathbf{U}}(\varphi,\tau)=\mathbf{U}(\varphi+\varphi_0,\tau+\tau_0)E_C$. Thus, we can rewrite \eqref{2:aympformuladeltatilde} as follows,
\begin{equation*}
e^{\mp i(\tau-\varphi)}\Delta(\varphi+\varphi_0,\tau+\tau_0)\sim \hat{\mathbf{U}}(\varphi+\varphi_0,\tau+\tau_0)E_C\tilde{\Theta}^{\pm}\quad\text{as}\quad\tau\rightarrow\pm i\infty,
\end{equation*}
which is equivalent to,
\begin{equation*}
e^{\mp i(\tau+\tau_0-(\varphi+\varphi_0))}\Delta(\varphi+\varphi_0,\tau+\tau_0)\sim \hat{\mathbf{U}}(\varphi+\varphi_0,\tau+\tau_0)E_C\tilde{\Theta}^{\pm}e^{\mp i (\tau_0-\varphi_0)},
\end{equation*}
as $\tau\rightarrow\pm i\infty$. On the other hand, taking into account \eqref{2:asymptoticexpdelta} we have that,
\begin{equation*}
e^{\mp i(\tau+\tau_0-(\varphi+\varphi_0))}\Delta(\varphi+\varphi_0,\tau+\tau_0)\sim \hat{\mathbf{U}}(\varphi+\varphi_0,\tau+\tau_0)\Theta^{\pm},
\end{equation*}
as $\tau\rightarrow\pm i\infty$. Finally, the uniqueness of the asymptotic expansions implies that $\Theta^{\pm}=E_C\tilde{\Theta}^{\pm}e^{\mp i (\tau_0-\varphi_0)}$. Rearranging terms and noting that $E_C^{-1}=E_{-C}$ we conclude the proof of the lemma. 
\end{proof}

Using this result and the definition of the constants $\Theta^{\pm}$ we construct the following analytic invariants.  

\begin{theorem}[Analytic Invariants]\label{2:TheoremStokesconstant}
The following numbers,
\begin{equation*} \mathcal{K}=\Theta_1^{+}\Theta_1^{-}\quad\text{and}\quad\mathcal{J}=\Omega(\Theta^{+},\Theta^{-}),\end{equation*}
do not depend on the choice of parameterizations and are invariant under symplectic changes of coordinates fixing the origin. Moreover, if $H$ is real analytic then
\begin{equation*}
\mathcal{K}=-\mathrm{sgn}(\eta)\left|\Theta^{-}_1\right|^2\in\mathbb{R}\quad\text{and}\quad\mathcal{J}=-\mathrm{sgn}(\eta)\Omega(\overline{\Theta^{-}},\Theta^{-})\in i\mathbb{R}.
\end{equation*}
\end{theorem}

\begin{proof}
First, we prove that $\mathcal{K}$ and $\mathcal{J}$ do not depend on the choice of the parameterizations. Given two parameterizations $\mathbf{\Gamma}^{\pm}$ and $\tilde{\mathbf{\Gamma}}^{\pm}$ we know by Lemma \ref{2:lemmatranslationTheta} that there exist $(\varphi_0\,\tau_0)\in\mathbb{C}^2$ and $C\in\mathbb{C}^{4\times 4}$ ($C^T=C$) such that $\tilde{\Theta}^{\pm}=E_C\Theta^{\pm}e^{\pm i(\tau_0-\varphi_0)}$. Thus,
$$
\tilde{\mathcal{K}}=\tilde{\Theta}^{+}_1\tilde{\Theta}^{-}_1=\Theta^{+}_1e^{i(\tau_0-\varphi_0)}\Theta^{-}_1e^{-i(\tau_0-\varphi_0)}=\Theta^{+}_1\Theta^{-}_1=\mathcal{K},
$$
and
\begin{equation*}\begin{split}
\tilde{\mathcal{J}}=\Omega(\tilde{\Theta}^{+},\tilde{\Theta}^{-})=\Omega(E_C&\Theta^{+}e^{i(\tau_0-\varphi_0)},E_C\Theta^{-}e^{- i(\tau_0-\varphi_0)})\\&=\Omega(E_C\Theta^{+},E_C\Theta^{-})=\Omega(\Theta^{+},\Theta^{-})=\mathcal{J}.
\end{split}
\end{equation*}
Next we prove that $\mathcal{K}$ and $\mathcal{J}$ are invariant under symplectic changes of coordinates fixing the origin. Let $\Psi:(\mathbb{C}^4,0)\rightarrow (\mathbb{C}^4,0)$ be an analytic symplectic map. Define
\begin{equation*}
\tilde{\mathbf{\Gamma}}^{\pm}(\varphi,\tau):=\Psi(\mathbf{\Gamma}^{\pm}(\varphi,\tau))\quad\text{and}\quad\tilde{\mathbf{U}}(\varphi,\tau):=D\Psi(\mathbf{\Gamma}^{-}(\varphi,\tau))\mathbf{U}(\varphi,\tau).
\end{equation*}
It is enough to prove that $\Omega(\tilde{\Theta}^{\pm},\mathbf{v})=\Omega(\Theta^{\pm},\mathbf{v})$ for all $\mathbf{v}\in\mathbb{C}^4$. Taking into account \eqref{2:asymptoticexpdelta} we can write $\tilde{\Delta}:=\tilde{\mathbf{\Gamma}}^{+}-\tilde{\mathbf{\Gamma}}^{-}$ as follows,
\begin{equation*}
\tilde{\Delta}(\varphi,\tau)=D\Psi(\mathbf{\Gamma}^{-}(\varphi,\tau))\Delta(\varphi,\tau)+\mathbf{g}(\varphi,\tau),
\end{equation*}
where $\mathbf{g}$ is analytic in $\mathbb{T}_h\times (D_r^{+}\cap D_r^{-})$ such that,
\begin{equation}\label{2:limitg}
\lim_{\mathrm{Im}\tau\rightarrow \pm\infty}\mathbf{g}(\varphi,\tau)e^{\mp i(1+\mu)(\tau-\varphi)}=0,
\end{equation}
for any $\mu>0$ arbitrarily small. Moreover, for $\mathbf{v}\in\mathbb{C}^4$ we have that,
\begin{equation}\label{2:omegatildeequalities}\begin{split}
\Omega(\tilde{\Delta},\tilde{\mathbf{U}}\mathbf{v})&=\Omega(D\Psi(\mathbf{\Gamma}^{-})\Delta+\mathbf{g},D\Psi(\mathbf{\Gamma}^{-})\mathbf{U}\mathbf{v})\\
&=\Omega(\Delta,\mathbf{U}\mathbf{v})+\Omega(\mathbf{g},D\Psi(\mathbf{\Gamma}^{-})\mathbf{U}\mathbf{v}),
\end{split}
\end{equation}
where the last equality follows from the fact that $\Psi$ is symplectic. From the asymptotics of $\mathbf{\Gamma}^{-}$ and $\mathbf{U}$ we know that $\mathbf{\Gamma}^{-}(\varphi,\tau)=O(\tau^{-1})$ and $\mathbf{U}(\varphi,\tau)=O(\tau^3)$ as $\tau\rightarrow \infty$ in $D_r^{-}$. Thus, for every $\mu>0$,
\begin{equation*}
\lim_{\mathrm{Im}\tau\rightarrow\pm\infty}D\Psi(\mathbf{\Gamma}^{-}(\varphi,\tau))\mathbf{U}(\varphi,\tau)\mathbf{v}e^{\pm i\mu(\tau-\varphi)}=0,
\end{equation*}
and taking into account \eqref{2:limitg} we get that,
\begin{equation*}
\lim_{\mathrm{Im}\tau\rightarrow\pm\infty}\Omega(\mathbf{g}(\varphi,\tau),D\Psi(\mathbf{\Gamma}^{-}(\varphi,\tau))\mathbf{U}(\varphi,\tau)\mathbf{v})e^{\mp i(\tau-\varphi)}=0.
\end{equation*}
Finally, the previous limit and \eqref{2:omegatildeequalities} gives,
\begin{equation*}\begin{split}
\Omega(\tilde{\Theta}^{\pm},\mathbf{v})&=\lim_{\mathrm{Im}\tau\rightarrow\pm\infty}\Omega(\tilde{\Delta}(\varphi,\tau),\tilde{\mathbf{U}}(\varphi,\tau)\mathbf{v})e^{\mp i(\tau-\varphi)}\\
&=\lim_{\mathrm{Im}\tau\rightarrow\pm\infty}\Omega(\Delta(\varphi,\tau),\mathbf{U}(\varphi,\tau)\mathbf{v})e^{\mp i(\tau-\varphi)}=\Omega(\Theta^{\pm},\mathbf{v})\,.
\end{split}
\end{equation*}

To conclude the proof of the theorem suppose that $H$ is real analytic. It is sufficient to prove that for any $\mathbf{v}\in\mathbb{R}^4$ we have,
\begin{equation}\label{2:equalityomegatheta}
\overline{\Omega(\Theta^{-},\mathbf{v})}=-\mathrm{sgn}(\eta)\Omega(\Theta^{+},\mathbf{v}).
\end{equation}
Indeed, it follows from the previous equality that $\overline{\Theta^-}=-\mathrm{sgn}(\eta)\Theta^+$, from which we obtain
\begin{equation*}
\mathcal{K}=-\mathrm{sgn}(\eta)\left|\Theta^{-}_1\right|^2\quad\text{and}\quad\mathcal{J}=-\mathrm{sgn}(\eta)\Omega(\overline{\Theta^{-}},\Theta^{-}).
\end{equation*}

We prove \eqref{2:equalityomegatheta} considering $\eta>0$. The $\eta<0$ case is proved analogously. According to \eqref{2:limitTheta} we can take $\tau_n=-i\sigma_n$ where $\sigma_n\to+\infty$ is an increasing sequence of real numbers such that for every $\mathbf{v}\in\mathbb{R}^4$ we have,
\begin{equation}\label{2:equationcorollaryomega}
\Omega(\Theta^{-},\mathbf{v})=\lim_{n\rightarrow+\infty}\Omega(\Delta(0,-i\sigma_n),\mathbf{U}(0,-i\sigma_n)\mathbf{v})e^{\sigma_n}.
\end{equation}
Remarks \ref{2:remarkrealcase} and \ref{2:remarkrealcase2} imply that,
\begin{equation*}
\overline{\Delta(0,-i\sigma_n)}=\Delta(\pi,i\sigma_n)\quad\text{and}\quad\overline{\mathbf{U}(0,-i\sigma_n)}=\mathbf{U}(\pi,i\sigma_n).
\end{equation*}
Thus, taking complex conjugation in \eqref{2:equationcorollaryomega} we get,
\begin{equation*}\begin{split}
\overline{\Omega(\Theta^{-},\mathbf{v})}&=\lim_{n\rightarrow+\infty}\Omega(\overline{\Delta(0,-i\sigma_n)},\overline{\mathbf{U}(0,-i\sigma_n)})e^{\sigma_n}\\
&=\lim_{n\rightarrow+\infty}\Omega(\Delta(\pi,i\sigma_n),\mathbf{U}(\pi,i\sigma_n))e^{i(-i\sigma_n-\pi)}e^{-i\pi}\\
&=-\Omega(\Theta^{+},\mathbf{v}),
\end{split}
\end{equation*}
as we wanted to show.

\end{proof}
The invariant $\mathcal{K}$ is known as the \textsl{Stokes constant}. If $\mathcal{K}$ does not vanish then the asymptotic expansion \eqref{2:asymptoticexpdelta} provides an exponentially small lower bound for the splitting distance $\left|\mathbf{\Gamma}^{+}(\varphi,\tau)-\mathbf{\Gamma}^{-}(\varphi,\tau)\right|$, which implies that $H$ is non-integrable and the normal form transformation diverges \cite{Z:05}.
\begin{corollary}\label{2:corollaryintegrability}
If $\mathcal{K}\neq0$ then $H$ is non-integrable. 
\end{corollary}

\subsection{Parametrized families}
Let $\mathcal{U}\subseteq\mathbb{C}^4$ be an open neighborhood of the origin and denote by $\mathbb{D}_\delta\subseteq\mathbb{C}$ the open disc of radius $\delta$ centered at the origin. In this section we consider analytic one-parameter families of Hamiltonians $H_\nu$ with a generic $1:-1$ resonance. We say that $H_\nu$ is an \textsl{analytic family} if,
\begin{equation*}
H_\nu=-I_1 + I_2+\eta I_3^2 + F_\nu,
\end{equation*}
where $\nu\in\mathbb{D}_\delta$ and $F_\nu:\mathcal{U}\rightarrow\mathbb{C}$ is analytic. We also suppose that $F_\nu$ is analytic with respect to $\nu$ and for each $\nu\in \mathbb{D}_\delta$, $F_\nu$ contains only monomials of order greater or equal than $5$. Moreover, the elliptic equilibrium satisfies the non-degenerate condition $\eta\neq0$.

For each $\nu\in\mathbb{D}_\delta$ the Hamiltonian vector field $X_{H_\nu}$ satisfies the assumptions of the previous theorems. In particular the function $\nu\mapsto \mathcal{K}(\nu)$ is well defined, where $\mathcal{K}(\nu)$ is the Stokes constant of the Hamiltonian $H_\nu$. The next result shows that the Stokes constant varies analytically with $\nu$. 
\begin{theorem}\label{2:analyticdependenceofK0}There exist parameterizations $\mathbf{\Gamma}^{\pm}_\nu$ and a normalized fundamental solution $\mathbf{U}_\nu$ both analytic with respect to $\nu\in\mathbb{D}_\delta$ such that $\Theta^{\pm}:\mathbb{D}_\delta\rightarrow\mathbb{C}^4$ is analytic.
\end{theorem}
According to the definition of $\mathcal{K}$ (see Theorem \ref{2:TheoremStokesconstant}) we conclude that $\mathcal{K}:\mathbb{D}_\delta\rightarrow\mathbb{C}$ is analytic. 
\begin{proof}[Proof of Theorem \ref{2:analyticdependenceofK0}] Tracing the proofs of Theorems \ref{2:Tformalseparatrix} and \ref{2:Thformalnormalizedfundmatrix} we see that there exist formal series $\hat{\mathbf{\Gamma}}_\nu$ and $\hat{\mathbf{U}}_\nu$ such that the coefficients of the these formal series depend polynomially on a finite number of coefficients of $H_\nu$, which are assumed to be analytic with respect to $\nu$. Thus, the coefficients of both $\hat{\mathbf{\Gamma}}_\nu$ and $\hat{\mathbf{U}}_\nu$ are analytic with respect to $\nu$. Note that the theory developed in Section \ref{2:sectionlinearop} can be generalized to functions that are also analytic with respect to a parameter. Following the proofs of Theorems \ref{2:theoremunstableseparatrix} and \ref{2:normalizedFundamentalMatrix} and the fact that the fundamental matrix $\mathbf{U}_0$ defined in \eqref{2:U0} does not depend on $\nu$ we conclude that there exist a normalized fundamental solution $\mathbf{U}_\nu$ and analytic parameterizations $\mathbf{\Gamma}^{\pm}_\nu$, all of which are analytic with respect to $\nu$ such that $\mathbf{U}_\nu\sim\hat{\mathbf{U}}_\nu$ and $\mathbf{\Gamma}^{\pm}_\nu\sim\hat{\mathbf{\Gamma}}_\nu$. Let $\Delta_\nu=\mathbf{\Gamma}^{+}_\nu-\mathbf{\Gamma}^{-}_\nu$. A closer look at the proof of Theorem \ref{2:theoremdeltaasymptoticformula} reveals that,
\begin{equation*}
\Delta_\nu(\varphi,\tau)=\mathbf{U}_\nu(\varphi,\tau)\mathbf{c}_{\nu}(\tau-\varphi)+\mathbf{R}_\nu(\varphi,\tau),
\end{equation*}
where $\mathbf{c}_\nu$ is an analytic $2\pi$-periodic vector-valued function defined in a lower half complex plane, analytic with respect to $\nu$ and $\mathbf{c}_\nu(z)\rightarrow 0$ as $\mathrm{Im}\,z\rightarrow-\infty$. Moreover $\mathbf{R}_\nu(\varphi,\tau)=O(e^{-(2-\epsilon)i(\tau-\varphi)})$ where the upper bound is uniform with respect to $\nu$ and $\epsilon$ is an arbitrarily small positive real number. As in the proof of Theorem \ref{2:theoremdeltaasymptoticformula} we can write $\mathbf{c}_\nu$ in Fourier series:
\begin{equation*}
\mathbf{c}_\nu(z)=\Theta^{-}(\nu) e^{-iz}+O(e^{-2iz}),\quad\text{as}\quad \mathrm{Im}\,z\rightarrow -\infty,
\end{equation*}
where again the bound is uniform with respect to $\nu$. The first Fourier coefficient $\Theta^-(\nu)$ is given by the well known integral,
\begin{equation*}
\Theta^-(\nu)=\frac{1}{2\pi}\int_{-i\sigma}^{2\pi-i\sigma}\mathbf{c}_\nu(s)e^{is}ds,
\end{equation*}
for some $\sigma>0$. Clearly $\Theta^-(\nu)$ is analytic with respect to $\nu$. Arguing in a similar way one can also prove that $\Theta^+(\nu)$ is analytic.  
\end{proof}
\subsubsection{Example}\label{2:ExampleH}
We shall give an example of a Hamiltonian having non-zero Stokes constant. Consider the following analytic family $H_\nu$ of Hamiltonians,
\begin{equation*}
H_\nu=-I_1+I_2+\eta I_3^2+\nu q_2^5,
\end{equation*}
where $\eta\in\mathbb{C}^*$ and $\nu\in\mathbb{C}$. Notice that $H_0=-I_1+I_2+\eta I_3^2$ is integrable since $I_1$ is a first integral of $H_0$. 

According to Theorem \ref{2:theoremunstableseparatrix} there exist $r>0$ and analytic parameterizations $\mathbf{\Gamma}^{\pm}_\nu:\mathbb{T}_h\times D_r^{\pm}\rightarrow\mathbb{C}^4$. Following the arguments in the proof of Theorem \ref{2:analyticdependenceofK0} these parameterizations are also analytic with respect to $\nu$. Thus we can write them as follows,
\begin{equation}\label{2:formulaGammapm0}
\mathbf{\Gamma}^{\pm}_{\nu}=\mathbf{\Gamma}_0+\nu\xi_0^{\pm}+O(\nu^2),
\end{equation} 
where $\mathbf{\Gamma}_0$ is the parameterization of $H_0$ (see \eqref{2:Gamma0}) and $\xi_0^{\pm}$ satisfies the following equation,
\begin{equation}\label{2:eqxi0pm}
\mathcal{D}\xi_0^{\pm}=A_0(\varphi,\tau)\xi_0^{\pm}+X_{q_2^5}(\mathbf{\Gamma}_0),
\end{equation}
where $A_0(\varphi,\tau):=DX_{H_0}(\mathbf{\Gamma}_0(\varphi,\tau))$. For our convenience, let us write (see \eqref{2:Gamma0}) the expression for $\mathbf{\Gamma}_0$,
\begin{equation*}
\mathbf{\Gamma}_0(\varphi,\tau)=\left(\kappa\tau^{-2}\cos\varphi,\kappa\tau^{-2}\sin\varphi,\kappa\tau^{-1}\cos\varphi,\kappa\tau^{-1}\sin\varphi\right)^{T}.
\end{equation*}
The homogeneous equation in \eqref{2:eqxi0pm} has a fundamental solution $\mathbf{U}_0(\varphi,\tau)$ given by \eqref{2:U0} and having the following properties: it is symplectic, i.e., $\mathbf{U}_0^T J \mathbf{U}_0=J$ and its last two columns are $\partial_\varphi\mathbf{\Gamma}_0$ and $\partial_\tau\mathbf{\Gamma}_0$ respectively. Thus, by the method of variation of constants we can write some integral formulae for $\xi_0^{\pm}$,
\begin{equation}\label{2:formulaexi0pm}\begin{split}
\xi_0^{-}(\varphi,\tau)&=\mathbf{U}_0(\varphi,\tau)\int_{-\infty}^0\mathbf{U}_0^{-1}(\varphi+s,\tau+s)X_{q_2^5}(\mathbf{\Gamma}_0(\varphi+s,\tau+s))ds,\\
\xi_0^{+}(\varphi,\tau)&=-\mathbf{U}_0(\varphi,\tau)\int_{0}^{+\infty}\mathbf{U}_0^{-1}(\varphi+s,\tau+s)X_{q_2^5}(\mathbf{\Gamma}_0(\varphi+s,\tau+s))ds.
\end{split}
\end{equation}
The integrals above converge uniformly for $\tau\in D_r^{\pm}$. Indeed a simple computation shows that,
\begin{equation}\label{2:Xq25}
X_{q_2^5}(\mathbf{\Gamma}_0)=\left(0,0,0,-\frac{5\kappa^4\sin^4\varphi}{\tau^8}\right)^{T}\,.
\end{equation}
Taking into account the leading orders of $\mathbf{U}_0$ (see \eqref{2:U0}), we can bound from above the integral in the first formula of \eqref{2:formulaexi0pm} using the following integral,
\begin{equation*}
\int_{-\infty}^{0}\frac{1}{\left|\tau+s\right|^2}ds
\end{equation*}
which converges uniformly with respect to $\tau\in D^{-}_r,$. A similar estimate shows that the second integral in \eqref{2:formulaexi0pm} converges uniformly. 

Our goal is to compute the Stokes constant $\mathcal{K}(\nu)$ of $H_\nu$. According to the results of the previous section $\mathcal{K}(\nu)$ is analytic with respect to $\nu$ and by definition $\mathcal{K}(\nu)=\Theta_1^{-}(\nu)\Theta_1^{+}(\nu)$ where $\Theta_1^{\pm}(\nu)$ are defined by the limits,
\begin{equation*}
\Theta_1^{\pm}(\nu)=\lim_{\mathrm{Im}\tau\rightarrow\pm\infty}\Omega(\Delta_\nu(\varphi,\tau),\partial_\varphi\mathbf{\Gamma}^{-}_\nu(\varphi,\tau))e^{\mp i(\tau-\varphi)},
\end{equation*}
where $\Delta_\nu:=\mathbf{\Gamma}^{+}_\nu-\mathbf{\Gamma}^{-}_\nu$. 
Since $H_0$ is integrable we know that $\mathcal{K}(0)=0$. So in order to prove that $\mathcal{K}(\nu)$ is non-zero for a certain $\nu$ it is sufficient to prove that the derivative of $\Theta_1^{\pm}(\nu)$ at $\nu=0$ does not vanish. The following lemma provides a formula for computing this derivative,
\begin{lemma}\label{2:Lemmaformulatheta0minus}Let $\Delta_0=\xi_0^{+}-\xi_0^{-}$. Then,
\begin{equation*}
\frac{d\Theta_1^{\pm}}{d\nu}(0)=\lim_{\mathrm{Im}\,\tau\rightarrow\pm\infty}\Omega(\Delta_0(\varphi,\tau),\partial_\varphi\mathbf{\Gamma}_0(\varphi,\tau))e^{\mp i(\tau-\varphi)}.
\end{equation*}
\end{lemma}
Let us postpone the proof of this lemma. In order to use the formula of the previous lemma we have to compute the difference $\Delta_0=\xi_0^{+}-\xi_0^{-}$. It follows from \eqref{2:formulaexi0pm} that,
\begin{equation}\label{2:Delta0integral}
\begin{split}
\Delta_0(\varphi,\tau)&=\mathbf{U}_0(\varphi,\tau)\int_{-\infty}^{+\infty}\mathbf{F}_0(\varphi+s,\tau+s)ds,\\
\text{where}\quad \mathbf{F}_0&(\varphi,\tau):=-\mathbf{U}_0^{-1}(\varphi,\tau)X_{q_2^5}(\mathbf{\Gamma}_0(\varphi,\tau)).
\end{split}
\end{equation}
Again, taking into account the expressions for $\mathbf{U}_0$ and \eqref{2:Xq25} a simple computation shows that,
\begin{equation}\label{2:F0exp}
\mathbf{F}_0(\varphi,\tau)=\left(\frac{5\kappa^5\cos\varphi\sin^4\varphi}{\tau^{10}},-\frac{10\kappa^5\sin^5\varphi}{\tau^{11}},-\frac{10\kappa^3\cos\varphi\sin^4\varphi}{3\tau^7},\frac{3\kappa^3\sin^5\varphi}{\tau^6}\right)^{T}.
\end{equation}
Since $\mathbf{U}_0$ is symplectic,  \eqref{2:Delta0integral} and \eqref{2:F0exp} imply that,
\begin{equation*}\begin{split}
\Omega(\Delta_0(\varphi,\tau),\partial_\varphi\mathbf{\Gamma}_0(\varphi,\tau))&=\Omega\left(\mathbf{U}_0(\varphi,\tau)\int_{-\infty}^{+\infty}\mathbf{F}_0(\varphi+s,\tau+s)ds,\partial_\varphi\mathbf{\Gamma}_0(\varphi,\tau)\right)\\
&=\int_{-\infty}^{+\infty}\frac{5\kappa^5\cos(\varphi+s)\sin^4(\varphi+s)}{(\tau+s)^{10}}ds\,.
\end{split}
\end{equation*}
Let us denote the integral above by $I(\varphi,\tau)$. Using the calculus of residues to compute this integral we obtain,
\begin{equation}\label{2:formulac03}
I(\varphi,\tau)=\delta\left(\frac{5\kappa^5\pi}{2^39!}e^{\delta i(\tau-\varphi)}-\frac{3^{10}5\kappa^5\pi}{2^49!}e^{3\delta i(\tau-\varphi)}+\frac{5^{10}\kappa^5\pi}{2^49!}e^{ 5\delta i(\tau-\varphi)}\right),
\end{equation}
where $\delta=\mathrm{sgn}(\mathrm{Im}\,\tau)$.
Finally, Lemma \ref{2:Lemmaformulatheta0minus} and \eqref{2:formulac03} give,
\begin{equation*}
\frac{d\Theta_1^{\pm}}{d\nu}(0)=\lim_{\mathrm{Im}\,\tau\rightarrow\pm\infty}\Omega(\Delta_0(\varphi,\tau),\partial_\varphi\mathbf{\Gamma}_0(\varphi,\tau))e^{\mp i(\tau-\varphi)}=\pm\frac{5\kappa^5\pi}{2^39!}.
\end{equation*}
Recall that $\kappa^2=-\frac{2}{\eta}$. Since $\eta\neq0$, the previous equality implies that $\frac{d\Theta_1^{\pm}}{d\nu}(0)\neq0$. Consequently $\mathcal{K}(\nu)$ is non-zero for $\left|\nu\right|\neq0$ sufficiently small.
\begin{proof}[Proof of Lemma \ref{2:Lemmaformulatheta0minus}]
We prove the lemma for the $-$ case, omitting the $+$ case as it is completely analogous.
According to the definition of $\Theta_1^{-}(\nu)$ we have that,
\begin{equation}\label{2:eqtheta0minuslemma}
\Theta_1^{-}(\nu)=\lim_{\mathrm{Im}\,\tau\rightarrow-\infty}\Omega(\Delta_\nu(\varphi,\tau),\partial_\varphi\mathbf{\Gamma}^{-}_\nu(\varphi,\tau))e^{i(\tau-\varphi)},
\end{equation}
where $\Delta_\nu=\mathbf{\Gamma}^{+}_\nu-\mathbf{\Gamma}^{-}_\nu$. Moreover, it follows from \eqref{2:formulac03} that,
\begin{equation}\label{2:F0star}
F_0:=\lim_{\mathrm{Im}\,\tau\rightarrow-\infty}\Omega(\Delta_0(\varphi,\tau),\partial_\varphi\mathbf{\Gamma}_0(\varphi,\tau))e^{i(\tau-\varphi)}<\infty.
\end{equation}
Define the following auxiliary function,
\begin{equation*}
R(\varphi,\tau,\nu)=\left\{\Omega(\Delta_\nu(\varphi,\tau),\partial_\varphi\mathbf{\Gamma}^{-}_\nu(\varphi,\tau))-\Omega(\Delta_0(\varphi,\tau),\partial_\varphi\mathbf{\Gamma}_0(\varphi,\tau))\nu\right\}e^{i(\tau-\varphi)}.
\end{equation*}
Note that $R$ is analytic in $\mathbb{T}_h\times D_r^{1}\times \mathbb{C}$ and $\frac{dR}{d\nu}(\varphi,\tau,0)=0$. Moreover, it follows from \eqref{2:eqtheta0minuslemma} and \eqref{2:F0star} that,
\begin{equation*}
\lim_{\mathrm{Im}\,\tau\rightarrow-\infty}R(\varphi,\tau,\nu)=\Theta_0^{-}(\nu)-F_0\nu.
\end{equation*}
Due to the uniform convergence of the limit we get at once,
\begin{equation*}
0=\left.\frac{d}{d\nu}\lim_{\mathrm{Im}\,\tau\rightarrow-\infty}R(\varphi,\tau,\nu)\right|_{\nu=0}=\frac{d\Theta_0^{-}}{d\nu}(0)-F_0.
\end{equation*}   
\end{proof}
\begin{corollary}
Let $G_\nu$ by an analytic family. For every $\epsilon>0$ there exists an $\epsilon$-close analytic family $H_\nu$, i.e.,
  \begin{equation*}
  \sup_{\nu\in\mathbb{D}_\delta}\left\|H_\nu-G_\nu\right\|<\epsilon,
  \end{equation*}
such that the Stokes constant of $H_\nu$ does not vanish on an open and dense subset of $\mathbb{D}_\delta$.  
\end{corollary}
\begin{proof}
By assumption $G_\nu$ is in the general form,
\begin{equation*}
G_\nu=-I_1 + I_2+\eta I_3^2 + F_\nu,
\end{equation*}
where $F_\nu$ is analytic and contains only monomials of order greater or equal than $5$. 
According to Example \ref{2:ExampleH} there exists $\nu_*\in\mathbb{D}_\delta$ such that the Stokes constant of the Hamiltonian $H_*=-I_1+I_2+\eta I_3^2 + \nu_* q_2^5$ is non-zero. Let,
\begin{equation*}
H_{\nu,\lambda}= G_\nu+\lambda(H_*-G_{\nu_*}),\quad \lambda\in\mathbb{C}\,.
\end{equation*}
Denote by $K(\nu,\lambda)$ the Stokes constant of $H_{\nu,\lambda}$. It follows from Theorem \ref{2:analyticdependenceofK0} that $\mathcal{K}(\nu_*,\lambda)$ is analytic with respect to $\lambda$. Moreover, since $H_{\nu_*,1}=H_*$ then $\mathcal{K}(\nu_*,1)\neq0$. Thus, for any $\epsilon>0$ we can choose,
\begin{equation*}
\gamma<\left\|H_*-G_{\nu_*}\right\|^{-1}\epsilon,
\end{equation*}
such that there exists $\lambda_{*}\in\mathbb{C}$ with $\left|\lambda_*\right|<\gamma$ and $\mathcal{K}(\nu_*,\lambda_*)\neq0$. Thus, $H_{\nu,\lambda_*}$ is the desired family.
\end{proof}

\section{Asymptotic series}\label{2:sectionformalseries}

In this section we prove Theorem \ref{2:Tformalseparatrix} and Proposition \ref{2:Thformalnormalizedfundmatrix}. These results deal with formal series, therefore we do not care about the convergence of the power series involved. 

We will look for formal solutions of equation \eqref{2:DH} in the class of formal power series in the variable $\tau^{-1}$ with coefficients in $\mathsf{T}$. To that end, it is convenient to transform $H$ into its normal form and compute a formal solution in the normal form coordinates. Then using the normal form transformation we pullback the formal solution to the original coordinates. 

According to Theorem \ref{1:TheoremNormalform11} there is a formal near identity symplectic change of variables $\mathbf{x}=\Phi(\mathbf{z})$ that transforms the Hamiltonian $H$ into its normal form, 
\begin{equation}\label{2:HNF}
H^{\sharp}=H\circ\Phi=-I_1+ I_2+\eta I_3^2+\sum_{3l+2k\geq 5}a_{l,k}I_1^lI_3^k,
\end{equation}
where $I_1$, $I_2$ and $I_3$ are given in \eqref{2:defI1I2I3} and $a_{l,k}\in\mathbb{C}$. Note that the normal form \eqref{2:HNF} is rotationally symmetric, i.e., it commutes with the one parameter group of rotations $\mathcal{R}_\varphi$,
\begin{equation}\label{2:defR}
\mathcal{R}_\varphi=\begin{pmatrix}R_\varphi&0\\0&R_\varphi\end{pmatrix}\quad\text{where}\quad R_\varphi=\begin{pmatrix}\cos\varphi&-\sin\varphi\\\sin\varphi&\cos\varphi\end{pmatrix}.
\end{equation}
In the following we look for formal solutions of,
\begin{equation}\label{2:eqDHNF}
\mathcal{D} \mathbf{z}=X_{H^{\sharp}}(\mathbf{z}),
\end{equation}
in the class of formal power series $\tau^{-1}\mathsf{T}^4[[\tau^{-1}]]$ (which is to be understood as the formal power series in $\mathsf{T}^4[[\tau^{-1}]]$ without the constant term). 
\begin{proposition}\label{2:lemmaformalseparatrix}
Equation \eqref{2:eqDHNF} has a formal solution $\hat{\mathbf{Z}}\in\tau^{-1}\mathsf{T}^4[[\tau^{-1}]]$ having the form $\hat{\mathbf{Z}}(\varphi,\tau)=\mathcal{R}_\varphi\xi(\tau)$ where $\xi\in\tau^{-1}\mathbb{C}^4[[\tau^{-1}]]$. The components of $\xi$ satisfy,
\begin{align*}
\xi_1(\tau)&=-\partial_\tau r(\tau)\cos \theta(\tau),&\xi_3(\tau)&=r(\tau)\cos \theta(\tau),\\
\xi_2(\tau)&=-\partial_\tau r(\tau)\sin \theta(\tau),&\xi_4(\tau)&=r(\tau)\sin \theta(\tau).
\end{align*}
where $\theta,r\in\mathbb{C}[[\tau^{-1}]]$ are odd formal power series having the leading orders,
\begin{equation*}
r(\tau)=\kappa\tau^{-1}+\cdots,\quad\theta=-\frac{a_{1,1}}{\eta}\tau^{-1}+\cdots,
\end{equation*}
where $\kappa^2=\frac{2}{\eta}$. The formal solution $\hat{\mathbf{Z}}$ is unique up to a rotation $\mathcal{R}_\pi$, i.e. $\hat{\mathbf{Z}}$ and $\mathcal{R}_\pi\hat{\mathbf{Z}}$ are the only formal solutions satisfying the properties stated above. Moreover, for any other formal solution $\hat{\mathbf{Y}}\in\tau^{-1}\mathsf{T}^4[[\tau^{-1}]]$ there exist $(\varphi_0,\tau_0)\in\mathbb{C}^2$ such that $\hat{\mathbf{Y}}(\varphi,\tau)=\hat{\mathbf{Z}}(\varphi+\varphi_0,\tau+\tau_0)$.
\end{proposition}
\begin{proof}
Setting $\mathbf{z}(\varphi,\tau)=\mathcal{R}_\varphi\xi(\tau)$ and taking into account that $X_{H^{\sharp}}$ commutes with $\mathcal{R}_\varphi$ (which has infinitesimal generator $-X_{I_1}$) then equation \eqref{2:eqDHNF} reduces to,
\begin{equation}\label{2:eqxi}
\partial_\tau\xi=X_{H^{\sharp}+I_1}(\xi).
\end{equation}
It is convenient to change to polar coordinates given by,
\begin{equation}\label{2:changeofvariablestheoremformalsep}
\begin{array}{cc}
\xi_1=R \cos\theta-\frac{\Theta}{r}\sin\theta ,& \xi_3=r \cos\theta, \\&\\
\xi_2=R \sin\theta+\frac{\Theta}{r}\cos\theta ,& \xi_4=r \sin\theta,
\end{array}
\end{equation}  
where $\xi=(\xi_1,\xi_2,\xi_3,\xi_4)$. Note that $I_1=\Theta$. In these new variables equation \eqref{2:eqxi} takes the form,
\begin{align}\label{2:XHdeltaNF}
\partial_\tau\theta=-\frac{\Theta}{r^2}-\sum_{3i+2j\geq5}\frac{i a_{i,j}}{2^{j}}\Theta^{i-1}r^{2j},\quad\partial_\tau r=- R,\quad\partial_\tau \Theta &=0,\\\label{2:XHdeltaNFb}
\partial_\tau R=\left(-\frac{\Theta^2}{r^3}+\eta r^3\right) + \sum_{3i+2j\geq5}\frac{2ja_{i,j}}{2^j}\Theta^ir^{2j-1}.
\end{align}
We solve these equations formally in $\mathbb{C}[[\tau^{-1}]]$. Let us start with the third equation of \eqref{2:XHdeltaNF}. Taking $\Theta\in\mathbb{C}[[\tau^{-1}]]$ and substitute into the equation we get immediately that $\Theta(\tau)=\Theta_0$ with $\Theta_0\in\mathbb{C}$. Since $\Theta(\tau)=\xi_2(\tau)\xi_3(\tau)-\xi_1(\tau)\xi_4(\tau)$ and each $\xi_i$ must be in $\tau^{-1}\mathbb{C}[[\tau^{-1}]]$ we conclude that $\Theta\in\tau^{-2}\mathbb{C}[[\tau^{-1}]]$. Thus $\Theta(\tau)=0$. 

We consider now the second equation of \eqref{2:XHdeltaNF} and equation \eqref{2:XHdeltaNFb}. Setting $\Theta=0$, these two equations are equivalent to the following single equation,
\begin{equation}\label{2:eqr1}
\partial^2_\tau r =-\eta r^3-\sum_{j\geq2}\frac{2(j+1)a_{0,j+1}}{2^{j+1}}r^{2j+1}.
\end{equation}
\begin{lemma}\label{2:claimr}
Equation \eqref{2:eqr1} has a non-zero formal solution $r$ having only odd powers of $\tau^{-1}$. Moreover,
\begin{equation}\label{2:leadingordersr}
r(\tau)=\kappa\tau^{-1}-\frac{1}{8}a_{0,3}\kappa^5\tau^{-3}+\cdots,
\end{equation}
where $\kappa^2=-\frac{2}{\eta}$. The solution is unique if we fix one of the two values for $\kappa$. Moreover, for any other non-zero formal solution $\tilde{r}\in\tau^{-1}\mathbb{C}[[\tau^{-1}]]$ of equation \eqref{2:eqr1} there exists $\tau_0$ such that $\tilde{r}(\tau)=\pm r(\tau+\tau_0)$.
\end{lemma}
\begin{proof}
Let us take a formal series $r(\tau)=\sum_{k\geq 1}r_k\tau^{-k}$ and substitute into equation \eqref{2:eqr1}. After collecting terms of the same order in $\tau^{-k-2}$ we obtain an equation which we can solve for the coefficient $r_k$. Let us present the details. At order $\tau^{-3}$ we get the following equation for $r_1$,
\begin{equation}\label{2:relationr1}
2r_{1}=-\eta r_1^3,
\end{equation}
which implies that $r_1^2=-\frac{2}{\eta}$ (the other solution is trivially $r_1$ which leads to the zero formal solution $r=0$). Hence we let $r_1=\kappa$ where $\kappa^2=-\frac{2}{\eta}$. Note that $\kappa$ can take two distinct values. We choose one value for $\kappa$ and move to the next order. At order $\tau^{-4}$ we obtain,
\begin{equation*}
6r_2=-3\eta r_1^2r_2.
\end{equation*}
Note that this equation is linear with respect to $r_2$. Taking into account that $r_1=\kappa$ we can simplify the previous equation and conclude that it holds for every $r_2\in\mathbb{C}$. Hence $r_2$ is a free coefficient. Since we are considering only odd powers of $r$ we set this coefficient to zero. 

At this stage, we have determined $r_1=\kappa$ and $r_2=0$. Now we proceed by induction on $k$. First let us determine $r_3$. It is not difficult to write the equation for $r_3$ which reads,
\begin{equation*}
6r_3=-\frac{6}{8}a_{0,3}r_1^5.
\end{equation*}
Thus $r_3=-\frac{1}{8}a_{0,3}\kappa^5$. Now suppose that all coefficients $r_l$, $3\leq l \leq k$ have been defined uniquely such that for $l$ even we have $r_l=0$ and for $l$ odd we have $r_l=p(\kappa)$ where $p\in\mathbb{C}[\kappa]$ and contains only odd powers in $\kappa$. Due to the induction hypothesis, at the order $\tau^{-k-3}$ we have the following equation for $r_{k+1}$,
\begin{equation*}
((k+1)(k+2)-6)r_{k+1}=f_{k+1}(r_1,\ldots,r_{k})
\end{equation*}
where $f_{k+1}$ is a polynomial depending on a finite number of coefficients $a_{0,j+1}$ for $j\geq2$. Note that it is always possible to solve the previous equation with respect to $r_{k+1}$ for $k\geq 2$ since $(k+1)(k+2)-6=0$ only if $k=1$ or $k=-4$. Now we have to distinguish two cases. First consider the case when $k+1$ is even. Since the right hand side of equation \eqref{2:eqr1} has only odd powers of $r$ and according to the induction hypothesis $r_l=0$ for even $l$ then $f_{k+1}=0$. Thus $r_{k+1}=0$. On the other hand, when $k+1$ is odd then by the same reasoning as above it is not difficult to see that $f_{k+1}$ is a polynomial in $\mathbb{C}[\kappa]$, having only odd powers of $\kappa$, and $r_{k+1}$ is determined uniquely by the formula $r_{k+1}=((k+1)(k+2)-6)^{-1}f_{k+1}$. This completes the induction. 
Finally let $\tilde{r}\in\tau^{-1}\mathbb{C}[[\tau^{-1}]]$ be a non-zero formal solution of equation \eqref{2:eqr1}. We can write $\tilde{r}=\sum_{k\geq1}\tilde{r}_k\tau^{-k}$. As before, we conclude that $\tilde{r}_1^2=\kappa^2$ thus, $\tilde{r}_1=\pm\kappa$. Now for $\tau_0\in\mathbb{C}$ we have that,
\begin{equation*}
r(\tau+\tau_0)=\frac{\kappa}{\tau+\tau_0}+\cdots=\frac{\kappa}{\tau}-\frac{\tau_0\kappa}{\tau^2}+\cdots.
\end{equation*}
is also a formal solution of equation \eqref{2:eqr1}. Comparing the second order coefficient $-\tau_0\kappa$ with the coefficient $\tilde{r}_2$ we conclude by the uniqueness of $r$ that if $\tau_0=-\frac{\tilde{r}_2}{\kappa}$ then $\tilde{r}(\tau)=\pm r(\tau+\tau_0)$ and the claim is proved.
\end{proof}

Using the formal solutions $\Theta(\tau)$ and $r(\tau)$ we simplify the first equation of \eqref{2:XHdeltaNF} to obtain,
\begin{equation}\label{2:eqtheta}
\partial_\tau \theta =-\sum_{j\geq 1}\frac{a_{1,j}}{2^j}\left(\sum_{k\geq1}r_{k}\tau^{-k}\right)^{2j}.
\end{equation}
Note that $\left(\sum_{k\geq1}r_{k}\tau^{-k}\right)^{2j}\in\tau^{-2j}\mathbb{C}[[\tau^{-1}]]$ and contains only even powers in $\tau^{-1}$. 
Thus equation \eqref{2:eqtheta} can be further simplified,
\begin{equation*}
\partial_\tau \theta =\sum_{k\geq1}b_k\tau^{-2k},
\end{equation*}
where $b_k$ depend on a finite number of coefficients of $r(\tau)$ and $a_{1,j}$ for $j\geq1$. Thus,
\begin{equation}\label{2:definitiongeneraltheta}
\theta(\tau)=\theta_0+\sum_{k\geq 1}\frac{b_k}{-2k+1}\tau^{-2k+1},
\end{equation}
where $\theta_0\in\mathbb{C}$. We set $\theta_0=0$. To conclude the proof, we show how to come back to the variable $\xi$. First observe that,
\begin{equation*}
\cos \theta(\tau)=\sum_{i\geq0}\frac{(-1)^i}{(2i)!}\left(\sum_{k\geq 1}\frac{b_k}{-2k+1}\tau^{-2k+1}\right)^{2i},
\end{equation*}
and taking into account that the formal series inside the parenthesis of the right hand side of the previous formula is an even formal series in $\tau^{-1}$ starting with the term $\tau^{-2i}$ we conclude that,
\begin{equation}\label{2:eqcos}
\cos \theta(\tau)=\sum_{k\geq0}w_k\tau^{-2k},
\end{equation}
where $w_k$ depend on a finite number of coefficients of $\theta(\tau)$. A similar formula holds for the sine which reads,
\begin{equation}\label{2:eqsin}
\sin \theta(\tau)=\sum_{k\geq0}z_k\tau^{-2k+1},
\end{equation}
where $z_k$ depend on a finite number of coefficients of $\theta(\tau)$. Now according to the change of variables \eqref{2:changeofvariablestheoremformalsep} the formal power series $\hat{\mathbf{Z}}(\tau):=\mathcal{R}_\varphi\xi(\tau)$ is the desired formal solution of equation \eqref{2:eqDHNF} where the components of $\xi$ are given by,
\begin{align*}
\xi_1(\tau)&=-\partial_\tau r(\tau)\cos \theta(\tau),&\xi_3(\tau)&=r(\tau)\cos \theta(\tau),\\
\xi_2(\tau)&=-\partial_\tau r(\tau)\sin \theta(\tau),&\xi_4(\tau)&=r(\tau)\sin \theta(\tau).
\end{align*}
The expressions \eqref{2:eqcos} and \eqref{2:eqsin} imply that $\xi_i\in\tau^{-1}\mathbb{C}[[\tau^{-1}]]$ for $i=1,\ldots,4$, thus proving the first part of the proposition. Any other formal solution satisfying the same properties of $\hat{\mathbf{Z}}$ (as stated in the proposition) will have the form,
\begin{equation*}
\mathcal{R}_\varphi\mathcal{R}_{\theta(\tau)+\theta_0}(-\partial_\tau r(\tau+\tau_0),0,r(\tau+\tau_0),0)^{T},
\end{equation*}
for some $\tau_0,\theta_0\in\mathbb{C}$. Clearly for $\tau_0\neq 0$, $r(\tau+\tau_0)$ will be no longer an odd power series in $\tau^{-1}$. Thus $\tau_0$ must be zero. Moreover, equation \eqref{2:relationr1} implies that $\theta_0=0$ or $\theta_0=\pi$. Therefore, $\hat{\mathbf{Z}}$ is uniquely defined up to a rotation $\mathcal{R}_\pi$. Moreover, if $\hat{\mathbf{Y}}\in\tau^{-1}\mathsf{T}^4[[\tau^{-1}]]$ is another formal solution then there exists $\tilde{\xi}\in\tau^{-1}\mathbb{C}^4[[\tau^{-1}]]$ such that $\hat{\mathbf{Y}}(\varphi,\tau)=\mathcal{R}_\varphi\tilde{\xi}(\tau)$. Taking into account Lemma \ref{2:claimr} and equation \eqref{2:definitiongeneraltheta} we conclude that 
$$
\tilde{\xi}(\tau)=\mathcal{R}_{\theta(\tau)+\varphi_0}(-\partial_\tau r(\tau+\tau_0),0,r(\tau+\tau_0),0)^{T},
$$
for some $(\varphi_0,\tau_0)\in\mathbb{C}^2$. This completes the proof of the proposition. 
\end{proof}
\begin{remark}\label{2:RemarkrealanalyticH}
If the Hamiltonian $H$ is real analytic then its normal form $H^{\sharp}$ is a formal series with real coefficients, i.e. $\overline{H^{\sharp}(\mathbf{z})}=H^{\sharp}(\overline{\mathbf{z}})$. In particular, the normal form coefficient $\eta$ is real. Depending on the sign of $\eta$ we can say more about the structure of the formal solutions of \eqref{2:eqDHNF}. If $\eta<0$ then one can trace the proof of the previous proposition and conclude that the coefficients of $\hat{\mathbf{Z}}$ are real, i.e., $\hat{\mathbf{Z}}(\varphi,\tau)=\mathcal{R}_\varphi\xi(\tau)$ where $\xi\in \tau^{-1}\mathbb{R}^4[[\tau^{-1}]]$. Thus, $\hat{\mathbf{Z}}(\varphi,\tau)=\overline{\hat{\mathbf{Z}}(\overline{\varphi},\overline{\tau})}$ when $\eta<0$. On the other hand, when $\eta>0$ then the coefficients of $\hat{\mathbf{Z}}$ are imaginary numbers, i.e. $\hat{\mathbf{Z}}(\varphi,\tau)=i\mathcal{R}_\varphi\xi(\tau)$ where $\xi\in\tau^{-1}\mathbb{R}^4[[\tau^{-1}]]$. Thus, $\hat{\mathbf{Z}}(\varphi,\tau)=\overline{\hat{\mathbf{Z}}(\overline{\varphi}+\pi,\overline{\tau})}$ when $\eta>0$.
\end{remark}
\begin{remark}\label{2:RemarkreversibilityofHNF}
The normal form Hamiltonian vector field $X_{H^{\sharp}}$ is time-reversible with respect to the linear involution,
\begin{equation}\label{2:LinearinvolutionS}
\mathcal{S}(q_1,q_2,p_1,p_2)=(-q_1,q_2,p_1,-p_2).
\end{equation}
If the Hamiltonian $H$ is real analytic then the formal solution $\hat{\mathbf{Z}}$ satisfies,
\begin{equation*}
\hat{\mathbf{Z}}(\varphi,\tau)=\overline{\mathcal{S}(\hat{\mathbf{Z}}(-\overline{\varphi},-\overline{\tau}))}\,.
\end{equation*}
The formal solution $\hat{\mathbf{Z}}$ is said to be symmetric and this condition defines the solution uniquely (up to a rotation $\mathcal{R}_\pi$) in a coordinate independent way.
\end{remark}
\subsection{Proof of Theorem \ref{2:Tformalseparatrix}}

By the normal form theory there exists a (non-unique) near identity formal symplectic change of variables $\mathbf{x}=\Phi(\mathbf{z})$ that transforms the Hamiltonian $H$ into its normal form $H^{\sharp}=H\circ\Phi$. Let $\mathbf{z}=(\mathbf{q},\mathbf{p})\in\mathbb{C}^2\times\mathbb{C}^2$. For our purposes, we can suppose that $\Phi$ is in the general form,
\begin{equation}\label{2:Psinormalform}
(\mathbf{q},\mathbf{p})\mapsto \left(\mathbf{q}+\sum_{2\left|i\right|+\left|j\right|\geq 4}c_{i,j}\mathbf{q}^i\mathbf{p}^j,\mathbf{p}+\sum_{2\left|i\right|+\left|j\right|\geq 3}d_{i,j}\mathbf{q}^i\mathbf{p}^j\right)\,,
\end{equation}
written in multi-index notation, for some $c_{i,j},d_{i,j}\in\mathbb{C}^2$. According to Proposition \ref{2:lemmaformalseparatrix} there exists a formal series $\hat{\mathbf{Z}}\in\tau^{-1}\mathsf{T}^4[[\tau^{-1}]]$ such that $\mathcal{D}\hat{\mathbf{Z}}=X_{H^{\sharp}}(\hat{\mathbf{Z}})$. Thus,
\begin{equation*}
\hat{\mathbf{\Gamma}}(\varphi,\tau):=\Phi\circ\hat{\mathbf{Z}}(\varphi,\tau)),
\end{equation*}
is a formal solution of equation \eqref{2:DH}. Note that $\hat{\mathbf{Z}}$ starts with terms of order $\tau^{-1}$. Thus, $\Phi\circ\hat{\mathbf{Z}}$ belongs to the same class of $\hat{\mathbf{Z}}$ since its coefficients can be computed from a finite number of coefficients of $\hat{\mathbf{Z}}$ and $\Phi$. Moreover, we know that $\hat{\mathbf{Z}}(\varphi,\tau)=\mathcal{R}_\varphi\xi(\tau)$ where the components of $\xi$ have the leading orders,
\begin{equation*}\begin{split}
\xi_1(\tau)=\kappa\tau^{-2}+\cdots\,,&\quad \xi_2(\tau)=-\frac{\kappa a_{1,1}}{\eta}\tau^{-3}+\cdots\,,\\
\xi_3(\tau)=\kappa\tau^{-1}+\cdots\,,&\quad \xi_2(\tau)=-\frac{\kappa a_{1,1}}{\eta}\tau^{-2}+\cdots\,.
\end{split}
\end{equation*}
Taking into account \eqref{2:Psinormalform} we obtain the leading orders of $\hat{\mathbf{\Gamma}}$ as stated in the theorem. Moreover, if $\tilde{\hat{\mathbf{\Gamma}}}\in\tau^{-1}\mathsf{T}^4[[\tau^{-1}]]$ is another formal solution of \eqref{2:DH} then it is clear from Proposition \ref{2:lemmaformalseparatrix} that there exist $(\varphi_0,\tau_0)\in\mathbb{C}^2$ such that $\tilde{\hat{\mathbf{\Gamma}}}(\varphi,\tau)=\hat{\mathbf{\Gamma}}(\varphi+\varphi_0,\tau+\tau_0)$. 
\qed
%
%
\begin{remark}\label{2:remarkrealcase}
If the original Hamiltonian $H$ is real analytic then $\overline{\hat{\mathbf{\Gamma}}(\overline{\varphi},\overline{\tau})}$ is also a formal solution of equation \eqref{2:H}. Indeed,
\begin{equation*}
\mathcal{D}\overline{\hat{\mathbf{\Gamma}}(\bar{\varphi},\bar{\tau})}=\overline{\overline{\mathcal{D}}\hat{\mathbf{\Gamma}}(\bar{\varphi},\bar{\tau})}=\overline{X_{H}(\hat{\mathbf{\Gamma}}(\bar{\varphi},\bar{\tau}))}=X_{H}\left(\overline{\hat{\mathbf{\Gamma}}(\bar{\varphi},\bar{\tau})}\right),
\end{equation*}
where $\overline{\mathcal{D}}=\partial_{\bar{\varphi}}+\partial_{\bar{\tau}}$. Moreover, since in the real analytic case the normal form transformation $\Phi$ has real coefficients then Remark \ref{2:RemarkrealanalyticH} implies that,
\begin{align*}
\overline{\hat{\mathbf{\Gamma}}(\bar{\varphi},\bar{\tau})}&=\Phi(\overline{\hat{\mathbf{Z}}(\bar{\varphi},\bar{\tau})})=\Phi(\hat{\mathbf{Z}}(\varphi+\pi,\tau))=\hat{\mathbf{\Gamma}}(\varphi+\pi,\tau),&\quad\mathrm{for}\quad\eta>0,\\
\overline{\hat{\mathbf{\Gamma}}(\bar{\varphi},\bar{\tau})}&=\Phi(\overline{\hat{\mathbf{Z}}(\bar{\varphi},\bar{\tau})})=\Phi(\hat{\mathbf{Z}}(\varphi,\tau))=\hat{\mathbf{\Gamma}}(\varphi,\tau),&\quad\mathrm{for}\quad\eta<0.
\end{align*}
\end{remark}

\begin{remark}\label{2:remarkseparatrixreversible}
If the original Hamiltonian $H$ is real analytic and $X_H$ is reversible with respect to the involution \eqref{2:LinearinvolutionS} then the normal form preserves the reversibility. By Remark \ref{2:RemarkreversibilityofHNF} the formal solution $\hat{\mathbf{\Gamma}}$ uniquely defined (up to a translation $\hat{\mathbf{\Gamma}}(\varphi+\pi,\tau)$) by the following condition,
\begin{equation*}
\hat{\mathbf{\Gamma}}(\varphi,\tau)=\overline{\mathcal{S}(\hat{\mathbf{\Gamma}}(-\bar{\varphi},-\bar{\tau}))}.
\end{equation*}
\end{remark}

\begin{remark}\label{2:remarkGamman}
Let $n\geq 1$ and $\mathbf{\Gamma}_n$ be a partial sum of the formal series $\hat{\mathbf{\Gamma}}$ up to order $\tau^{-n-1}$ in the first two components and up to order $\tau^{-n}$ in the last two.
Then,
\begin{equation}\label{2:eqestimateGamman}
\mathcal{D}\mathbf{\Gamma}_n-X_H(\mathbf{\Gamma}_n)=\left(\tau^{-(n+2)}R_{1,n},\tau^{-(n+2)}R_{2,n},\tau^{-(n+1)}R_{3,n},\tau^{-(n+1)}R_{4,n}\right),
\end{equation}
for some $R_{i,n} \in \mathsf{T}^4[[\tau^{-1}]]$, $i=1,\ldots,4$. Indeed, for a formal series $\hat{\mathbf{\Gamma}}=\sum_{k\geq1}\Gamma_k\tau^{-k}$ to solve formally equation \eqref{2:DH}, then the coefficients $\Gamma_k$ must solve the infinite system of equations,
\begin{equation}\label{2:eqsystemGamman}
\partial_\varphi\Gamma_k-X_{-I_1+I_2}(\Gamma_{k})=(k-1)\Gamma_{k-1}+G_k(\Gamma_{1},\ldots,\Gamma_{k-2}),\quad k\in\mathbb{N},
\end{equation}
obtained from substituting the formal series into equation \eqref{2:DH} and collecting terms of the same order in $\tau^{-k}$. The $G_k$'s are polynomials in $k-2$ variables and can be defined in a recursive way.

Since the first $n$ coefficients of the sum $\mathbf{\Gamma}_n$ solve \eqref{2:eqsystemGamman} for $k=1,\ldots,n$ then in order to get \eqref{2:eqestimateGamman} we consider the equation \eqref{2:eqsystemGamman} for $k=n+1$.
Note that the left hand side of equation \eqref{2:eqsystemGamman} depends only on the $k$th coefficient of the formal series $\hat{\mathbf{\Gamma}}$. Moreover, due to the form of the vector field $X_{-I_1+I_2}$, the first two components of the expression in the left hand side of \eqref{2:eqsystemGamman} only depend on the first two components of $\Gamma_{k}$. These observations allow us to conclude \eqref{2:eqestimateGamman}.
\end{remark}

\subsection{Formal variational equation}
In this subsection we prove Proposition \ref{2:Thformalnormalizedfundmatrix}. Consider the formal variational equation of $X_H$ around the formal separatrix $\hat{\mathbf{\Gamma}}$,
\begin{equation}\label{2:eqdH}
\mathcal{D}\mathbf{u}=DX_{H}(\hat{\mathbf{\Gamma}})\mathbf{u}.
\end{equation}
Our goal is to construct a convenient basis for the space of formal solutions of equation \eqref{2:eqdH}. These formal solutions provide asymptotic expansions for certain analytic solutions of equation \eqref{2:VariationaleqGammaminus}. We know already two formal solutions of the previous equation: $\partial_\varphi\hat{\mathbf{\Gamma}}$ and $\partial_\tau\hat{\mathbf{\Gamma}}$. Note that these formal solutions are linearly independent as formal series in $\mathsf{T}^4[[\tau^{-1}]]$. Moreover,
\begin{equation}\label{2:Lagrangiancondition}
\Omega(\partial_\varphi\hat{\mathbf{\Gamma}},\partial_\tau\hat{\mathbf{\Gamma}})=0,
\end{equation}
where $\Omega$ is the standard symplectic form \eqref{1:Omegadef}. The previous equality follows from a more general fact: if $\mathbf{u}_1$ and $\mathbf{u}_2$ are two formal solutions of \eqref{2:eqdH}, then $\Omega(\mathbf{u}_1,\mathbf{u}_2)\in\mathbb{C}$. To prove this, note that
\begin{equation}\label{2:proofD0}\begin{split}
\mathcal{D}\Omega(\mathbf{u}_1,\mathbf{u}_2)&=\Omega(\mathcal{D}\mathbf{u}_1,\mathbf{u}_2)+\Omega(\mathbf{u}_1,\mathcal{D}\mathbf{u}_2)\\
&=\Omega(DX_{H}(\hat{\mathbf{\Gamma}})\mathbf{u}_1,\mathbf{u}_2)+\Omega(\mathbf{u}_1,DX_{H}(\hat{\mathbf{\Gamma}})\mathbf{u}_2)\\
&=0.
\end{split}
\end{equation}
In particular, $\mathcal{D}\Omega(\partial_\varphi\hat{\mathbf{\Gamma}},\partial_\tau\hat{\mathbf{\Gamma}})=0$. Now we apply the next Lemma to get the desired equality. 
\begin{lemma}\label{2:Lemmag0}
Let $g\in\tau^j\mathsf{T}^4[[\tau^{-1}]]$ for some $j\in\mathbb{Z}$ and suppose that $\mathcal{D}g=0$. Then $g=g_0\in\mathbb{C}$. In addition, if $j\leq -1$ then $g=0$.
\end{lemma}
\begin{proof}
Let $g=\sum_{k\leq j}g_k\tau^k$ where $g_k\in\mathsf{T}^4$. Substituting $g$ into the equation $\mathcal{D}g=0$ and collecting terms of the same order in $\tau^k$ we get the following system of equations,
\begin{equation}\label{2:syslemmag}\begin{split}
&\partial_\varphi g_j=0,\\
&\partial_\varphi g_k+(k+1)g_{k+1}=0,\quad k\leq j-1.
\end{split}
\end{equation}
The first equation of \eqref{2:syslemmag} implies that $g_j\in\mathbb{C}$. Now using the second equation we can solve for $g_k$. Taking into account that $g_k\in\mathsf{T}^4$ we conclude that $(k+1)g_{k+1}=0$ for all $k\leq j-1$. Note that when $k=-1$ we have no restriction on $g_0$ and the Lemma follows.
\end{proof}
\begin{proof}[Proof of Proposition \ref{2:Thformalnormalizedfundmatrix}]
In the proof of Theorem \ref{2:Tformalseparatrix} we have obtained the formal solution $\hat{\mathbf{\Gamma}}$ using the normal form Hamiltonian $H^{\sharp}$ by defining $\hat{\mathbf{\Gamma}}:=\Phi\circ\hat{\mathbf{Z}}$, where $\Phi$ is the normal form transformation and $\hat{\mathbf{Z}}$ is the formal solution of Proposition \ref{2:lemmaformalseparatrix}. Also from the same proposition we know that $\hat{\mathbf{Z}}=\mathcal{R}_\varphi\xi$ where $\mathcal{R}_\varphi$ is defined in \eqref{2:defR} and $\xi$ is a formal series having the form,
\begin{equation}\label{2:xipolarcoordinates}
\xi(\tau)=\left(-\partial_\tau r(\tau)\cos\theta(\tau),-\partial_\tau r(\tau)\sin\theta(\tau),r(\tau)\cos\theta(\tau),r(\tau)\sin\theta(\tau)\right)^{T},
\end{equation}
where $r$ and $\theta$ are the formal series \eqref{2:leadingordersr} and \eqref{2:definitiongeneraltheta} respectively. In the normal form coordinates equation \eqref{2:eqdH} reads,
\begin{equation}\label{2:eqdHb}
\mathcal{D}\mathbf{v}=DX_{H^{\sharp}}(\hat{\mathbf{Z}})\mathbf{v},
\end{equation}
where $\mathbf{u}=D\Phi(\hat{\mathbf{Z}})\mathbf{v}$. We seek for formal solutions of \eqref{2:eqdHb} in the form $\mathbf{v}=R_\varphi\zeta$ where $\zeta\in\tau^{j}\mathbb{C}^4[[\tau^{-1}]]$ for some $j\in\mathbb{Z}$. Similar to the proof of Pproposition \ref{2:lemmaformalseparatrix} the formal series $\zeta$ must satisfy the equation,
\begin{equation*}
\partial_\tau\zeta=DX_{H^{\sharp}+I_1}(\xi)\zeta.
\end{equation*}
Bearing in mind \eqref{2:xipolarcoordinates}, we now rewrite the previous equation in polar coordinates,
\begin{equation}\label{2:systemomega}\begin{split}
\partial_\tau w_1 &= -\sum_{l \geq 1} \frac{la_{1,l}}{2^{l-1}}r^{2l-1}w_2-\left(\frac{1}{r^2}+\sum_{l\geq0}\frac{a_{2,l}}{2^{l-1}}r^{2l}\right)w_3,\quad\partial_\tau w_2 = - w_4,\\
\partial_\tau w_3 &= 0,\quad\partial_\tau w_4 = \left(3\eta r^2+\sum_{l \geq 3}\frac{l(2l-1)a_{0,l}}{2^{l-1}}r^{2l-2}\right)w_2+\sum_{l \geq 1} \frac{la_{1,l}}{2^{l-1}}r^{2l-1}w_3,
\end{split}
\end{equation}
where $\hat{\mathbf{w}}=(w_i)$, $\zeta=D\Lambda(\theta,r,0,-\partial_\tau r)\hat{\mathbf{w}}$ and $\Lambda$ denotes the change of variables \eqref{2:changeofvariablestheoremformalsep}. Note that $\Lambda$ is symplectic with multiplier $-1$, i.e. $(D\Lambda)^T J D\Lambda =- J$. We know already two formal solutions of equation \eqref{2:systemomega}:
\begin{equation}\label{2:formalsolutionw1w4}
\hat{\mathbf{w}}_3=\left(1,0,0,0\right)^{T}\quad\mathrm{and}\quad\hat{\mathbf{w}}_4=\left(\partial_\tau\theta,\partial_\tau r,0,-\partial^2_\tau r\right)^{T}.
\end{equation}
In the original coordinates, these formal solutions correspond to $\partial_\varphi\hat{\mathbf{\Gamma}}$ and $\partial_\tau\hat{\mathbf{\Gamma}}$ respectively.
We now construct other two formal solutions that are formally independent of \eqref{2:formalsolutionw1w4} and belong to the class of formal series $\tau^{j}\mathbb{C}[[\tau^{-1}]]$ for some $j\in\mathbb{Z}$. Let us consider the second and fourth equations of \eqref{2:systemomega}. They are equivalent to the single equation,
\begin{equation}\label{2:eqforomega2}
\partial^2_\tau w_2 = -\left(3\eta r^2+\sum_{l \geq 3}\frac{l(2l-1)a_{0,l}}{2^{l-1}}r^{2l-2}\right)w_2-\sum_{l \geq 1} \frac{la_{1,l}}{2^{l-1}}r^{2l-1}w_3.
\end{equation}
In order to solve the previous equation, we first study the formal solutions of the homogeneous equation. 
\begin{lemma}\label{2:sublemmaw}
The linear homogeneous equation,
\begin{equation}\label{2:eqforomega2homogeneous}
\partial^2_\tau w_2 = -\left(3\eta r^2+\sum_{l \geq 3}\frac{l(2l-1)a_{0,l}}{2^{l-1}}r^{2l-2}\right)w_2,
\end{equation}
has two linearly independent formal solutions,
\begin{equation*}
w_{2,1}\in\tau^{-2}\mathbb{C}[[\tau^{-1}]]\quad\mathrm{and}\quad w_{2,2}\in\tau^3\mathbb{C}[[\tau^{-1}]]
\end{equation*}
such that $w_{2,1}$ is an even formal series and $w_{2,2}$ an odd formal series. Moreover $w_{2,1}=\partial_\tau r$, $w_{2,2}=\frac{\tau^3}{5\kappa}+\frac{7}{40}a_{0,3}\kappa^3\tau+\cdots$ and,
\begin{equation}\label{2:determinantcondition}
w_{2,2}\partial_\tau w_{2,1}-w_{2,1}\partial_\tau w_{2,2}=1.
\end{equation}
\end{lemma}
\begin{proof}
That $\partial_\tau r$ is a formal solution of the homogeneous equation is obvious. Moreover its properties follow from the properties of the formal series $r$. Now let us determine the second formal solution. It follows from the fact that the formal series $r\in\tau^{-1}\mathbb{C}[[\tau^{-1}]]$ is odd that the right hand side of the homogeneous equation \eqref{2:eqforomega2homogeneous} is a formal series of the form $b=\sum_{k\leq -1}b_k\tau^{2k}$ where $b_k$ depend on a finite number of coefficients of $r$ and $a_{0,l}$ for $l\geq3$. Moreover, according to \eqref{2:leadingordersr} we have,
\begin{equation*}
r(\tau)=\kappa\tau^{-1}-\frac{1}{8}a_{0,3}\kappa^5\tau^{-3}+\cdots,
\end{equation*}
where $\kappa^2=-\frac{2}{\eta}$. Using the leading orders of $r$, we compute the first few orders of the formal series $b$ for further reference,
\begin{equation}\label{2:leadingorderb}
b_{-1}=6\quad\mathrm{and}\quad b_{-2}=-\frac{21a_{0,3}}{\eta^2}.
\end{equation}
Now we are ready to solve equation \eqref{2:eqforomega2homogeneous} in the class of formal series. Thus, substituting the formal series $w_{2,2}=\sum_{k\leq 1}w_{2,2,k}\tau^{2k+1}$ into equation \eqref{2:eqforomega2homogeneous} and collecting terms of the same order in $\tau^k$ we obtain the following infinite system of linear equations,
\begin{equation*}
\left(2k(2k+1)-6\right)w_{2,2,k}=\sum_{j=k-2}^{-2}w_{2,2,k-j-1}b_j,\quad k=1,0,-1,\ldots
\end{equation*}
For $k=1$ we get no condition on the first coefficient, thus $w_{2,2,1}\in\mathbb{C}$. For $k=0$ we obtain $w_{2,2,0}=-\frac{1}{6}w_{2,2,1}b_{-2}$. When $k\leq-1$, a simple induction argument shows that we can determine the coefficients $w_{2,2,k}$ (which depend linearly on the coefficient $w_{2,2,1}$) in a recursive way by using the previous formula since $\left(2k(2k+1)-6\right)=0$ only if $k=1$ or $k=-\frac{3}{2}$.
Finally let us derive the equality \eqref{2:determinantcondition}. Since,
\begin{equation*}
\partial_\tau\left(w_{2,2}\partial_\tau w_{2,1}-w_{2,1}\partial_\tau w_{2,2}\right)=0,
\end{equation*}
due to the fact that both $w_{2,1}$ and $w_{2,2}$ solve the homogeneous equation \eqref{2:eqforomega2homogeneous}
we have that $w_{2,2}\partial_\tau w_{2,1}-w_{2,1}\partial_\tau w_{2,2}$ is equal to some constant. Taking into account the leading orders of the formal solutions $w_{2,1}$ and $w_{2,2}$ we conclude that $w_{2,2}\partial_\tau w_{2,1}-w_{2,1}\partial_\tau w_{2,2}=5\kappa w_{2,2,1}$. As $w_{2,2,1}$ is a free coefficient we can define $w_{2,2,1}:=\frac{1}{5\kappa}$ and obtain the desired equality. 
\end{proof}

Returning to the non-homogeneous equation \eqref{2:eqforomega2}, we see that the last term of the right hand side of the equation depends on $w_3$ from which we know that $\partial_\tau w_3=0$. Thus $w_3=w_{3,0}\in\mathbb{C}$ is a constant. Now, taking into account that $r$ is an odd formal power series we conclude that,
$$
g(\tau):=\sum_{l \geq 1} \frac{la_{1,l}}{2^{l-1}}r^{2l-1}\in\tau^{-1}\mathbb{C}[[\tau^{-1}]],
$$ 
is an odd formal series whose coefficients depend on a finite number of coefficients of $r$ and $a_{1,l}$ for $l\geq1$. Using the well known method of variation of constants we can write the general formal solution of \eqref{2:eqforomega2} as follows,
\begin{equation}\label{2:generalformalsolutionw2}
w_2=c_1w_{2,1}+c_2w_{2,2}+w_{2,2}\int^{\tau}w_{2,1}gw_{3,0}-w_{2,1}\int^{\tau}w_{2,2}gw_{3,0},
\end{equation}
where $w_{3,0},c_1,c_2\in\mathbb{C}$. Note that the integration in the previous formula is well defined in the class of formal series $\mathbb{C}[[\tau^{-1}]][[\tau]]$. Indeed, it can be easily checked that $w_{2,1}g\in\tau^{-3}\mathbb{C}[[\tau^{-1}]]$ is an odd formal series and $w_{2,2}g\in\tau^{2}\mathbb{C}[[\tau^{-1}]]$ is an even formal series. Hence both integrands do not contain the term $\tau^{-1}$. Next we define two particular formal solutions of \eqref{2:eqforomega2},
\begin{equation}\label{2:eqdefw20} 
w_2^0:=w_{2,2}\quad\mathrm{and}\quad w_2^{1}:=w_{2,2}\int^{\tau}w_{2,1}g-w_{2,1}\int^{\tau}w_{2,2}g.
\end{equation}
The first formal solution corresponds to setting $c_1=w_{3,0}=0$ and $c_2=1$ in the general solution \eqref{2:generalformalsolutionw2} and the second corresponds to $c_1=c_2=0$ and $w_{3,0}=1$. Note that $w_2^0\in\tau^{3}\mathbb{C}[[\tau^{-1}]]$ is an odd formal series and $w_2^{1}\in\tau\mathbb{C}[[\tau^{-1}]]$ is also odd formal series. 

Now coming back to the first equation of \eqref{2:systemomega}, we can rewrite it as follows,
\begin{equation*}
\partial_\tau w_1=-g w_2 + f w_{3,0},
\end{equation*}
where,
\begin{equation*}
f=-\frac{1}{r^2}-\sum_{l\geq0}\frac{a_{2,l}}{2^{l-1}}r^{2l}.
\end{equation*}
It is not difficult to see that $f\in\tau^2\mathbb{C}[[\tau^{-1}]]$ is an even formal series. Moreover both $gw_2^0\in\tau^2\mathbb{C}[[\tau^{-1}]]$ and $gw_2^{-1}\in\mathbb{C}[[\tau^{-1}]]$ are even formal series. These observations allow us to conclude that the following are formal solutions of \eqref{2:systemomega},
\begin{equation}\label{2:eqdefw10}
w_1^0=-\int^\tau gw_2^0\quad\mathrm{and}\quad w_1^{1}=-\int^\tau gw_2^{1}+\int^\tau f,
\end{equation}
which are well defined in the class of formal series $\mathbb{C}[[\tau^{-1}]][[\tau]]$ and moreover $w_1^0,w_1^{1}\in\tau^3\mathbb{C}[[\tau^{-1}]]$ are both odd formal series. Thus we obtain two formal solutions of \eqref{2:systemomega} defined as follows,
\begin{equation*}
\hat{\mathbf{w}}_1:=\left(w_1^{1},w_2^{1},1,-\partial_\tau w_2^{1}\right)^{T}\quad\mathrm{and}\quad\hat{\mathbf{w}}_2:=\left(w_1^0,w_2^0,0,-\partial_\tau w_2^0\right)^{T}.
\end{equation*}
Note that $\left\{\hat{\mathbf{w}}_i\right\}_{i=1,\ldots,4}$ is a set of linearly independent formal solutions of equation \eqref{2:systemomega} and that,
\begin{equation}\label{2:symplecticbasiscondition}\begin{split}
\Omega(\hat{\mathbf{w}}_1,\hat{\mathbf{w}}_2)=0,\quad\Omega(\hat{\mathbf{w}}_2,\hat{\mathbf{w}}_4)=-1,\quad\Omega(\hat{\mathbf{w}}_1,\hat{\mathbf{w}}_4)=0,\\
\Omega(\hat{\mathbf{w}}_2,\hat{\mathbf{w}}_3)=0,\quad\Omega(\hat{\mathbf{w}}_1,\hat{\mathbf{w}}_3)=-1,\quad\Omega(\hat{\mathbf{w}}_3,\hat{\mathbf{w}}_4)=0.
\end{split}
\end{equation}
where $\Omega$ is the canonical symplectic form in the polar coordinates, i.e., $\Omega=d\theta\wedge\Theta+dr\wedge dR$.
The bottom identities of \eqref{2:symplecticbasiscondition} are straightforward to prove using the definition of $\hat{\mathbf{w}}_i$. The ones on the top are harder to prove and so we handle them bellow. First note that similar arguments as in \eqref{2:proofD0} show that $\partial_\tau\Omega(\hat{\mathbf{w}}_i,\hat{\mathbf{w}}_j)=0$ for $i,j=1,\ldots,4$. Secondly, it follows from Lemma \ref{2:sublemmaw} and from \eqref{2:leadingordersr} that,
\begin{equation}\label{2:leadingorderswr}
w_{2,2}(\tau)=\frac{\tau^3}{5\kappa}+\frac{7}{40}a_{0,3}\kappa^3\tau+\cdots\quad\mathrm{and}\quad r(\tau)=\kappa\tau^{-1}-\frac{1}{8}a_{0,3}\kappa^5\tau^{-3}+\cdots.
\end{equation}
Now we compute $\Omega(\hat{\mathbf{w}}_1,\hat{\mathbf{w}}_2)$. Using the definition of both $\hat{\mathbf{w}}_1$ and $\hat{\mathbf{w}}_2$ we get
\begin{equation*}
\Omega(\hat{\mathbf{w}}_1,\hat{\mathbf{w}}_2)=-w^1_0-w_2^0\partial_\tau w_2^{1}+\partial_\tau w_0^2 w_2^{1}.
\end{equation*}
Bearing in mind \eqref{2:eqdefw20} and \eqref{2:eqdefw10} we can simplify the previous expression and rewrite it as follows,
\begin{equation*}
\Omega(\hat{\mathbf{w}}_1,\hat{\mathbf{w}}_2)=\left(1-w_{2,2}\partial_\tau^2 r+\partial_\tau w_{2,2}\partial_\tau r\right)\int^\tau gw_{2,2}.
\end{equation*}
Now using the leading orders \eqref{2:leadingorderswr} we conclude that the expression inside the parenthesis in the previous formula belongs to $\tau^{-4}\mathbb{C}[[\tau^{-1}]]$. Moreover $\int^\tau gw_{2,2}\in\tau^3\mathbb{C}[[\tau^{-1}]]$ and consequently $\Omega(\hat{\mathbf{w}}_1,\hat{\mathbf{w}}_2)\in\tau^{-1}\mathbb{C}[[\tau^{-1}]]$. Applying Lemma \ref{2:Lemmag0} we get $\Omega(\hat{\mathbf{w}}_1,\hat{\mathbf{w}}_2)=0$ as we wanted to show. 

Now we handle $\Omega(\hat{\mathbf{w}}_2,\hat{\mathbf{w}}_4)$. Again, it follows from the definitions \eqref{2:eqdefw20} that,
\begin{equation*}
\Omega(\hat{\mathbf{w}}_2,\hat{\mathbf{w}}_4)=w_{2,1}\partial_\tau w_{2,2}-w_{2,2}\partial_\tau w_{2,1}. 
\end{equation*}
The identity now follows from \eqref{2:determinantcondition}. 

At last, let us compute $\Omega(\hat{\mathbf{w}}_1,\hat{\mathbf{w}}_4)$. Again using the definitions of the power series $\hat{\mathbf{w}}_1$ and $\hat{\mathbf{w}}_4$ we get,
\begin{equation*}
\Omega(\hat{\mathbf{w}}_1,\hat{\mathbf{w}}_4)=\partial_\tau\theta+w_2^{-1}\partial_\tau R+\partial_\tau r\partial_\tau w_2^{-1}.
\end{equation*}
This last expression belongs to $\tau^{-2}\mathbb{C}[[\tau^{-2}]]$ and applying Lemma \ref{2:Lemmag0} we obtain the desired result.

Coming back to the coordinates of equation \eqref{2:eqdHb} we define,
\begin{equation*}
\hat{\mathbf{v}}_i(\varphi,\tau):=\mathcal{R}_\varphi D\Lambda(\theta(\tau),r(\tau),0,-\partial_\tau r(\tau))\hat{\mathbf{w}}_i(\tau).
\end{equation*}
Clearly the matrix $\hat{\mathbf{V}}=\left(\hat{\mathbf{v}}_i\right)_{i=1,\ldots,4}$ consists of linearly independent formal solutions of equation \eqref{2:eqdHb} such that $\hat{\mathbf{v}}_3=\partial_\varphi\hat{\mathbf{Z}}$ and $\hat{\mathbf{v}}_4=\partial_\tau\hat{\mathbf{Z}}$. Moreover, a simple computation shows that,
\begin{equation*}
D\Lambda(\theta,r,0,-\partial_\tau r)=\begin{pmatrix}\tau^{-3}\Lambda_{1}&0&\Lambda_{2}&\Lambda_{3}\\
\tau^{-2}\Lambda_{4}&0&\tau\Lambda_{5}&\tau^{-1}\Lambda_{6}\\
\tau^{-2}\Lambda_{7}&\Lambda_{8}&0&0\\
\tau^{-1}\Lambda_{9}&\tau^{-1}\Lambda_{10}&0&0
\end{pmatrix},
\end{equation*}
where $\Lambda_i\in\mathbb{C}[[\tau^{-1}]]$ for $i=1,\ldots,10$. Thus, taking into account the definition of $\hat{\mathbf{w}}_1$ and $\hat{\mathbf{w}}_2$ we conclude that,
\begin{equation*}\begin{split}
\hat{\mathbf{v}}_1&=\left(\tau\hat{v}_{1,1},\tau\hat{v}_{2,1},\tau^2\hat{v}_{3,1},\tau^2\hat{v}_{4,1}\right)^{T},\\
\hat{\mathbf{v}}_2&=\left(\tau^{2}\hat{v}_{1,2},\tau^2\hat{v}_{2,2},\tau^{3}\hat{v}_{3,2},\tau^{3}\hat{v}_{4,2}\right)^{T},
\end{split}
\end{equation*} 
for some $\hat{v}_{i,1},\hat{v}_{i,2}\in\mathsf{T}[[\tau^{-1}]]$, $i=1,\ldots,4$. Since $\Lambda$ is symplectic with multiplier $-1$ and taking into account the identities \eqref{2:symplecticbasiscondition} we get,
\begin{equation}\label{2:symplecticbasiscondition2}
\hat{\mathbf{V}}^T J \hat{\mathbf{V}}= J.
\end{equation}
Finally, pulling back the formal solutions $\hat{\mathbf{v}}_i$ by the normal form transformation $\Phi$ we obtain the desired formal fundamental solution $\hat{\mathbf{U}}:=D\Phi(\hat{\mathbf{Z}})\hat{\mathbf{V}}$. Similar to the proof of Theorem \ref{2:Tformalseparatrix}, $\hat{\mathbf{U}}$ belongs to the same class of formal series as $\hat{\mathbf{V}}$. Moreover, \eqref{2:symplecticbasiscondition2} implies that $\hat{\mathbf{U}}^T J \hat{\mathbf{U}}= J$. In order to conclude the proof of the proposition, note that by the method of variation of constants a general formal solution of equation \eqref{2:eqdH} is of the form $\hat{\mathbf{U}}\mathbf{c}$ where $\mathbf{c}$ is any formal series in $\tau^j\mathsf{T}^4[[\tau^{-1}]]$ for some $j\in\mathbb{Z}$, such that $\mathcal{D}\mathbf{c}=0$. It follows from Lemma \ref{2:Lemmag0} that $\mathbf{c}\in\mathbb{C}^4$. Thus, if $\hat{\tilde{\mathbf{U}}}$ is another formal fundamental solution of \eqref{2:eqdH} then there exists a matrix $E\in\mathbb{C}^{4\times 4}$ such that $\hat{\tilde{\mathbf{U}}}=\hat{\mathbf{U}}E$. 
Since $\hat{\tilde{\mathbf{U}}}$ and $\hat{\mathbf{U}}$ are symplectic it also follows that $E$ must be symplectic. Moreover, as the third and fourth columns of $\hat{\tilde{\mathbf{U}}}$ have to be the derivatives of $\hat{\mathbf{\Gamma}}$ then a simple computation shows that one can reduce the number of entries of $E$ to obtain \eqref{2:matrixE}. This concludes the proof of the proposition.
\end{proof}
\begin{remark}\label{2:remarkrealcase2}
Similar to Remark \ref{2:remarkrealcase}, one can trace the proof of the previous proposition and conclude that when $H$ is real analytic then,
\begin{equation*}
\overline{\hat{\mathbf{U}}(\overline{\varphi},\overline{\tau})}=\begin{cases}\hat{\mathbf{U}}(\varphi,\tau)&\text{if}\quad\eta<0\\\hat{\mathbf{U}}(\varphi+\pi,\tau)&\text{if}\quad\eta>0\end{cases}
\end{equation*}
\end{remark}
\begin{remark}\label{2:remarkUn}
For $n\geq1$ let $\mathbf{U}_n$ be a partial sum of the formal series $\hat{\mathbf{U}}$ up to order $\tau^{-n-1}$ in the first two components (of each column) and up to order $\tau^{-n}$ in the last two components. Similar to Remark \ref{2:remarkGamman} we have that,
\begin{equation*}
\mathcal{D}\mathbf{U}_n-DX_{H}(\mathbf{\Gamma}_{n+3})\mathbf{U}_n=
\begin{pmatrix}
\tau^{-n-2}\xi_{1,1}&\tau^{-n-2}\xi_{1,2}&\tau^{-n-2}\xi_{1,3}&\tau^{-n-2}\xi_{1,4}\\
\tau^{-n-2}\xi_{2,1}&\tau^{-n-2}\xi_{2,2}&\tau^{-n-2}\xi_{2,3}&\tau^{-n-2}\xi_{2,4}\\
\tau^{-n-1}\xi_{3,1}&\tau^{-n-1}\xi_{3,2}&\tau^{-n-1}\xi_{3,3}&\tau^{-n-1}\xi_{3,4}\\
\tau^{-n-1}\xi_{4,1}&\tau^{-n-1}\xi_{4,2}&\tau^{-n-1}\xi_{4,3}&\tau^{-n-1}\xi_{4,4}
\end{pmatrix}
\end{equation*}
for some $\xi_{i,j}\in\mathsf{T}[[\tau^{-1}]]$. 
\end{remark}

\section{Linear operators}\label{2:sectionlinearop}

In this section we define certain complex Banach spaces and study some linear operators acting on them. The linear operators and motivated by the study of the solutions of the nonlinear PDE \eqref{2:DH}. These technical results are at the core of the proofs of the main theorems of this paper. 
\subsection{Solutions of $\mathcal{D} u=f$}

Fix $h>0$ and let $\mathbb{T}_h = \left\{\varphi \in \mathbb{C}/2\pi\mathbb{Z}\,\colon\, \left|\mathrm{Im}\,\varphi\right|<h \right\}$. We consider the problem of solving the linear PDE,
\begin{equation}\label{2:Dxf}
\mathcal{D} u=f,
\end{equation}
where $\mathcal{D}=\partial_\varphi+\partial_\tau$ is a first order linear differential operator and $u$ and $f$ are analytic complex-valued functions defined in $\mathbb{T}_h\times B$ where $B$ is some domain of $\mathbb{C}$. 

The simplest case is when $f=0$. As one would expect, by using the method of characteristics, a solution of the homogeneous equation $\mathcal{D}u=0$ must be a function which is constant along the characteristics 
$$\dot{\varphi}=1\quad\text{and}\quad\dot{\tau}=1\,.
$$ 
Thus, $u$ is a function depending on a single variable, say $\tau-\varphi$. The next result determines such function and its domain of analyticity.
\begin{lemma}\label{2:LemmaDx0} Let $u:\mathbb{T}_h\times B\rightarrow\mathbb{C}$ be analytic and suppose that $\mathcal{D}u=~0$. Then there exists a unique analytic function $c:\bigcup_{\tau\in B}\tau+\mathbb{T}_h\rightarrow\mathbb{C}$ such that $u(\varphi,\tau)=c(\tau-\varphi)$.
\end{lemma}
\begin{proof}
Given $\tau_0\in B$ let
\begin{equation*}
\Omega_{\tau_0}=\left\{(\varphi,\tau)\in \mathbb{T}_h\times B\:\colon\:\varphi-\tau+\tau_0\in \mathbb{T}_h\right\}.
\end{equation*}
Note that $\Omega_{\tau_0}$ is an open and connected set of $\mathbb{C}^2$. The initial value problem,
\begin{equation}\label{2:systemxiDx0}
\begin{cases}
\mathcal{D}\xi=0\\
\xi(\varphi,\tau_0)=u(\varphi,\tau_0)
\end{cases}\,,
\end{equation}
has a solution $\xi(\varphi,\tau)=u(\varphi-\tau+\tau_0,\tau_0)$. Hence $\xi$ is an analytic function of a single variable $\tau-\varphi$ and is defined in the translated horizontal strip $\tau_0+\mathbb{T}_h$. By the main local existence and uniqueness theorem for analytic partial differential equations (see \cite{Fo:95} for instance) we conclude that $u=\xi$ on $\Omega_{\tau_0}$. Thus $u(\varphi,\tau)=u(\varphi-\tau+\tau_0,\tau_0)$. 
Taking into account that $\mathbb{T}_h\times B=\bigcup_{\tau_0\in B}\Omega_{\tau_0}$ and the uniqueness of analytic continuation we get the desired result. 
\end{proof}
When $f$ is non-zero and defined in $\mathbb{T}_h\times D_r^{\pm}$, where the sets $D_r^{\pm}$ are depicted in Figure \ref{1:Dr}, then equation \eqref{2:Dxf} has two solutions,
\begin{equation*}
u^{-}(\varphi,\tau)=\int_{-\infty}^0f(\varphi+s,\tau+s)ds\quad\mathrm{and}\quad u^{+}(\varphi,\tau)=-\int_{0}^{+\infty}f(\varphi+s,\tau+s)ds,
\end{equation*}
provided the integrand in both functions is well defined in the domain of $f$ and the corresponding integral converges.
\begin{proposition}\label{2:propositionDxf}
Let $r>1$ and $f:\mathbb{T}_h\times D_r^{-}\rightarrow\mathbb{C}$ be analytic and continuous in the closure of its domain. Moreover, suppose that $\left|f(\varphi,\tau)\right|\leq \frac{K_f}{\left|\tau\right|^p}$ for some $K_f>0$ and $p\geq2$. Then,
\begin{equation*}
u^{-}(\varphi,\tau)=\int_{-\infty}^{0}f(\varphi+s,\tau+s)ds,
\end{equation*}
defines an analytic function in $\mathbb{T}_h\times D_r^{-}$, continuous in the closure of its domain. Moreover,
\begin{equation}\label{2:propositionDxfestimates}
\left|u^{-}(\varphi,\tau)\right|\leq \frac{K_{p-1}K_f}{\left|\tau\right|^{p-1}},
\end{equation}
for some $K_p>0$ independent of $r$. 
\end{proposition}
In order to prove this proposition we need the following estimate,
\begin{lemma}\label{2:desint}
Let $p\geq1$, $\tau \in D^{+}_r$. Then there exists a constant $K_{p}>0$ such that,
\begin{equation}\label{2:E1}
\int_{-\infty}^{0} \frac{1}{\left|\tau+s\right|^{p+1}} ds \leq \frac{K_{p}}{\left|\tau\right|^p}.
\end{equation}
\end{lemma}

\begin{figure}[t]
  \begin{center}
    \includegraphics[width=2.1in]{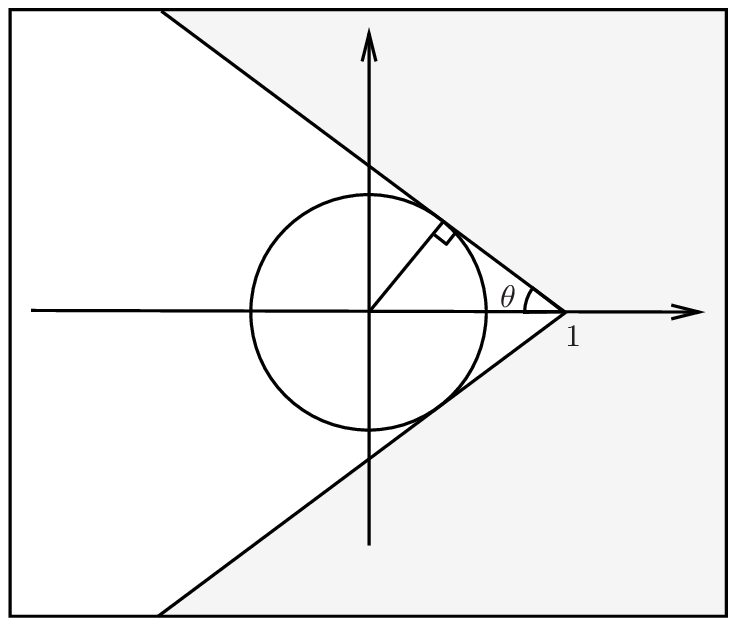}
  \end{center}
  \caption{The set $\left\{1+e^{-i \arg(\tau)}t\:\colon\: t\leq0\quad\text{and}\quad\tau\in D_r^-\right\}$.} 
  \label{2:figlemma53}
\end{figure}

\begin{proof}
The proof of this lemma follows from simple estimates. First, using a suitable change of variables we can write,
\begin{equation*}
\int_{-\infty}^{0}\frac{ds}{\left|\tau+s\right|^{p+1}}\underbrace{=}_{t=\frac{s}{\left|\tau\right|}}\frac{1}{\left|\tau\right|^p}\int_{-\infty}^{0}\frac{dt}{\left|1+e^{-i \arg(\tau)}t\right|^{p+1}}.
\end{equation*}
Now we show that the integral in the right-hand-side of the previous equation is bounded by a constant which only depends on $p$ and $\theta$ (see the definition of $D_r^+$ in \eqref{2:defDrminus}). To that end we split the integral,
\begin{equation*}
\int_{-\infty}^{0}\frac{dt}{\left|1+e^{-i \arg(\tau)}t\right|^{p+1}}=\int_{-1}^{0}\frac{dt}{\left|1+e^{-i \arg(\tau)}t\right|^{p+1}}+\int_{-\infty}^{-1}\frac{dt}{\left|1+e^{-i \arg(\tau)}t\right|^{p+1}}\,,
\end{equation*}
and estimate each term separately. Clearly $\left|1+e^{-i \arg(\tau)}t\right|\geq\sin\theta$ for all $t\leq0$ and $\tau \in D^{-}_r$ (see Figure \ref{2:figlemma53}). Thus
\begin{equation*}
\begin{split}
\int_{-1}^{0}\frac{dt}{\left|1+e^{-i \arg(\tau)}t\right|^{p+1}}&\leq\sup_{t\in \left[-1,0 \right]}\frac{1}{\left|1+e^{-i \arg(\tau)}t\right|^{p+1}}\\
&\leq \frac{1}{(\sin\theta)^{p+1}}\,.
\end{split}
\end{equation*}
On the other hand,
\begin{equation*}
\begin{split}
\left|1+e^{-i \arg(\tau)}t\right|^2&=1+2t\cos(\arg(\tau))+t^2\\
&\geq\cos^2(\arg(\tau))+2t\cos(\arg(\tau))+t^2\\
&=(\cos\arg(\tau)+t)^2\,.
\end{split}
\end{equation*}
Thus
\begin{equation*}
\left|1+e^{-i \arg(\tau)}t\right|\geq \left|t+\cos(\arg(\tau))\right|, \ \forall t\in\mathbb{R} \ \ \forall \tau \in D^{-}_r\,,
\end{equation*}
which implies that,
\begin{equation*}
\int_{-\infty}^{-1}\frac{dt}{\left|1+e^{-i \arg(\tau)}t\right|^{p+1}}\leq\frac{1}{p\left(1-\cos(\arg(\tau))\right)^p}\leq\frac{1}{p\left(1-\cos\theta\right)^p}\,,
\end{equation*} 
and the result follows.
\end{proof}

\begin{proof}[Proof of Proposition \ref{2:propositionDxf}]
Let $f:\mathbb{T}_h\times D_r^{-}\rightarrow\mathbb{C}$ be an analytic function as defined in the statement of the proposition. Moreover we know that $\left|f(\varphi,\tau)\right|\leq \frac{K_f}{\left|\tau\right|^p}$ for some $K_f>0$ and $p\geq2$. For $N\geq0$ we have $(\varphi-N,\tau-N)\in \mathbb{T}_h\times D^{-}_r$. Thus,
\begin{equation}\label{2:intunifconv}\begin{split}
\int_{-\infty}^{-N}\left|f(\varphi+s,\tau+s)\right|ds&\leq\int_{-\infty}^{0}\left|f(\varphi-N+s,\tau-N+s)\right|ds\\
&\leq\int_{-\infty}^0\frac{K_f}{\left|\tau-N+s\right|^p}ds\\
&\leq \frac{K_{p-1}K_f}{\left|\tau-N\right|^{p-1}},
\end{split}
\end{equation}
by Lemma \ref{2:desint}. Hence, the integral $\int_{-N}^{0}f(\varphi+s,\tau+s)ds$ converges uniformly in $\mathbb{T}_h\times D_r^{-}$ as $N\rightarrow + \infty$. We can apply a classical result of analysis (see for instance \cite{JD:69} on pag. 236) to deduce that, 
\begin{equation*}
u^{-}(\varphi,\tau)=\int_{-\infty}^{0}f(\varphi+s,\tau+s)ds,
\end{equation*}
is an analytic function in $\mathbb{T}_h\times D_r^{-}$. The continuity in the closure of its domain also follows from the continuity of $f$ and the uniform convergence of the integral \eqref{2:intunifconv}. The upper bound for $u^{-}$ follows from \eqref{2:intunifconv} with $N=~0$. 
\end{proof}
\begin{remark}
A similar proposition holds for the function,
\begin{equation*}
u^{+}(\varphi,\tau)=-\int_{0}^{+\infty}f(\varphi+s,\tau+s)ds,
\end{equation*}
which is defined in $\mathbb{T}_h\times D_r^{+}$.
\end{remark}

Now we consider the problem of solving equation \eqref{2:Dxf} but for functions defined in $\mathbb{T}_h\times D_r^{1}$ where,
\begin{equation*}
D^{1}_r =D^{+}_r \cap D^{-}_r \cap \left\{\tau \in \mathbb{C}\:\colon\: \mathrm{Im}\,\tau <-r\right\}.
\end{equation*}
Regarding this new domain $D_r^{1}$ we can not repeat the same arguments of Proposition \ref{2:propositionDxf} since $D_r^{1}$ does not contain an infinite horizontal segment. In order to overcome this difficulty, we construct an analytic solution of \eqref{2:Dxf} using a technique similar to the partition of unity, originally developed by V. F. Lazutkin in \cite{La:92}. This technique relies on a version of the Cauchy integral formula for analytic functions which we now describe in detail. 

Let $\mathcal{L}(\partial D_r^1)$ denote the set of bounded complex-valued Lipschitz functions $\chi:\partial D_r^1\to \mathbb{C}$ with the norm,
$$
\left\|\chi\right\|=\sup_x\left|\chi(x)\right|+\sup_{x\neq y}\left|\frac{\chi(x)-\chi(y)}{x-y}\right|\,.
$$
\begin{lemma}[Cauchy integral]\label{2:CauchyIntegral}
Let $\chi\in\mathcal{L}(\partial D_r^1)$ and $f\colon\mathbb{T}_h\times D_r^1\to \mathbb{C}$ be a bounded analytic function having a continuous extension to the closure of its domain. Moreover, suppose that
$$
J_f=\frac{1}{2\pi}\int_{\partial D_r^1}\left|f(\varphi,\tau)\right|\left|d\tau\right|<\infty\,,\quad\forall\varphi\in\mathbb{T}_h\,.
$$
Then
$$
h(\varphi,\tau)=\frac{1}{2\pi i}\int_{\partial D_r^1}\frac{\chi(\xi)f(\varphi,\tau)}{\xi-\tau}d\xi\,
$$
defines two analytic functions $h_{int}$ and $h_{ext}$ defined in $\mathbb{T}_h\times D_r^1$ and $\mathbb{T}_h\times\mathbb{C}\setminus \overline{D_r^1}$ respectively. Moreover, both functions extend continuously to the closure of its domains and 
$$\left|h_{int,ext}(\varphi,\tau)\right|\leq \left\|\chi\right\|(J_f+\sup\left|f\right|)\,.$$
\end{lemma}
\begin{proof}
This lemma is a parameterized version of Lemma 9.2 in \cite{VG:99}. Its proof is completely analogous and we shall omit the details. 
\end{proof}
\begin{remark}
If $\operatorname{supp}(\chi)\subsetneq\partial D_r^1$ then $h_{int}=h_{ext}$ on $\mathbb{C}\setminus\operatorname{supp}(\chi)$
\end{remark}

\begin{proposition}\label{2:propositionDxf1}
Let $\epsilon\geq0$, $p\geq4$ and $r>\max\left\{2,\frac{2\tan\theta}{1-\tan\theta}\right\}$. Suppose that $f\colon\mathbb{T}_h\times D_r^1\to\mathbb{C}$ is analytic, continuous on the closure of its domain and there exists $K_f>0$ such that
$$
\left|f(\varphi,\tau)\right|\leq \frac{K_f}{\left|\tau^{p}e^{i\epsilon(\tau-\varphi)}\right|}\,,\quad\forall (\varphi,\tau)\in\mathbb{T}_h\times D_r^1\,.
$$
Then equation $\mathcal{D}u=f$ has an analytic solution $u\colon\mathbb{T}_h\times D_r^1\to\mathbb{C}$, continuous on the closure of its domain, such that
\begin{equation*}
\left|u(\varphi,\tau)\right|\leq \frac{4 K_f K_{p-3}}{r}\frac{1}{\left|\tau^{p-3}e^{i\epsilon(\tau-\varphi)}\right|}
\end{equation*}
\end{proposition}
\begin{proof}
Following the ideas of \cite{VG:99} we define the domains,
\begin{equation*}
\begin{split}
\tilde{D}^{-}_r &= \left\{\tau \in \mathbb{C}\:\colon\: \left| \arg{(\tau+r)} \right| > \theta\quad\mathrm{and}\quad \mathrm{Im}\,\tau<-r\right\},\\
\tilde{D}^{+}_r &= \left\{\tau \in \mathbb{C}\:\colon\: -\tau \in \tilde{D}^{-}_r\right\}.\\
\end{split}
\end{equation*}
Note that $D^1_r=\tilde{D}^{+}_r\cap\tilde{D}^{-}_r$. Let $\mu(\varphi,\tau)=\tau^{p-2}e^{i\epsilon(\tau-\varphi)}$ and $\tilde{f}(\varphi,\tau)=\mu(\varphi,\tau)f(\varphi,\tau)$. We use the previous lemma on the Cauchy integral to write the function $\tilde{f}$ as a sum of two functions $\tilde{f}^{\pm}$ analytic in $\mathbb{T}_h\times \tilde{D}_r^{\pm}$ respectively. To that end, we define a partition of unity for the set $\partial D^1_r$ as follows. Let $\chi\colon\mathbb{R}\to[0,1]$ be a smooth function such that $\chi(t)=0$ for $t\leq-1$, $\chi(t)=1$ for $t\geq1$ and $\left|\chi'(t)\right|\leq1$ for all $t\in \mathbb{R}$. Define two functions $\chi^{\pm}\colon\partial D^1_r \to [0,1]$ by,
\begin{equation*}
\chi^{+}(\tau)=\chi\left(\mathrm{Re}(\tau)\right)\quad\text{and}\quad\chi^{-}(\tau)=1-\chi^{+}(\tau)\,.
\end{equation*}

\begin{figure}[t]
  \begin{center}
    \includegraphics[width=2.1in]{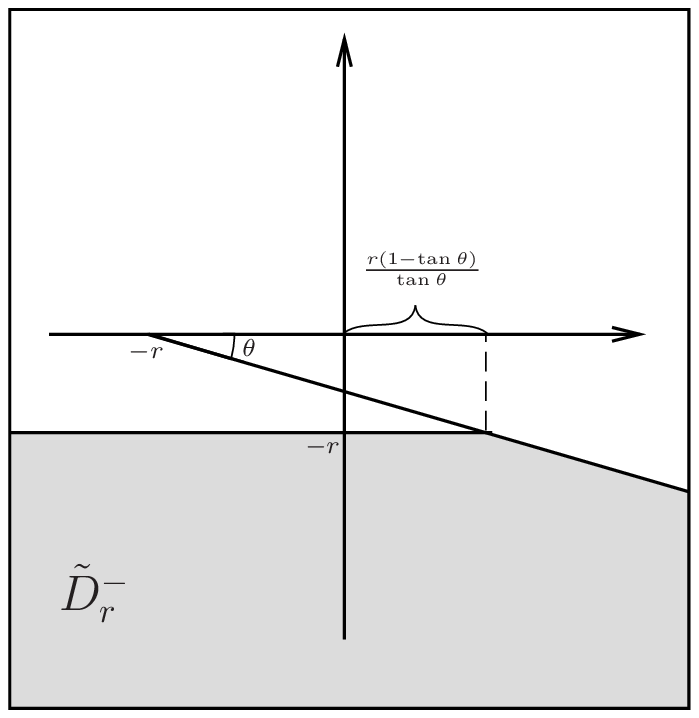}
  \end{center}
  \caption{The domain $\tilde{D}_r^-$} 
  \label{2:figprop57}
\end{figure}

Clearly $\chi^\pm\in\mathcal{L}(\partial D_r^1)$ and $\left\|\chi^\pm\right\|\leq2$. Since $r>\frac{2\tan\theta}{1-\tan\theta}$ (see Figure \ref{2:figprop57}), $\tilde{f}^\pm\colon \mathbb{T}_h\times \tilde{D}_r^\pm\to\mathbb{C}$ defined by
$$
\tilde{f}^\pm(\varphi,\tau)=\int_{\partial D_r^1}\frac{\chi^\pm(\xi)\tilde{f}(\varphi,\xi)}{\xi-\tau}d\xi
$$
is analytic, continuous on the closure of its domain. Moreover
\begin{equation*}
f(\varphi,\tau)=\frac{1}{\mu(\varphi,\tau)}\left(\tilde{f}^{-}(\varphi,\tau)+\tilde{f}^{+}(\varphi,\tau)\right).
\end{equation*}  
Hence,
\begin{equation}\label{2:E12}
u(\varphi,\tau)=\int_{-\infty}^0\frac{\tilde{f}^{-}(\varphi+s,\tau+s)}{\mu(\varphi+s,\tau+s)}ds-\int_{0}^{+\infty}\frac{\tilde{f}^{+}(\varphi+s,\tau+s)}{\mu(\varphi+s,\tau+s)}ds
\end{equation}
is a solution of equation $\mathcal{D}u=f$ provided the integrals in \eqref{2:E12} converge uniformly. Let us show that the first integral defines an analytic function in $\mathbb{T}_h\times D_r^1$. The second integral can be handled analogously. 

Applying Lemma \ref{2:desint} and the upper bound from Lemma \ref{2:CauchyIntegral} to the first term of \eqref{2:E12} we get,
\begin{equation*}
\begin{split}
\int_{-\infty}^{0}\left|\frac{\tilde{f}^{-}(\varphi+s,\tau+s)}{\mu(\varphi+s,\tau+s)}\right|ds&\leq 
\frac{\left\|\chi^{-}\right\|(J_{\tilde{f}}+\sup|\tilde{f}|)}{\left|e^{i\epsilon(\tau-\varphi)}\right|}\int_{-\infty}^0 \frac{1}{\left|\tau+s\right|^{p-2}}ds\\
&\leq \frac{\left\|\chi^{-}\right\|(J_{\tilde{f}}+\sup|\tilde{f}|)K_{p-3}}{\left|\tau^{p-3}e^{i\epsilon(\tau-\varphi)}\right|}\,.
\end{split}
\end{equation*}
Clearly $\left\|\chi^{-}\right\|\leq2$, $\sup|\tilde{f}|\leq K_f/r^2$ and $J_{\tilde{f}}\leq\frac{K_f}{2\pi r}$. Since $r>2$ we get,
$$
\left\|\chi^{-}\right\|(J_{\tilde{f}}+\sup|\tilde{f}|)\leq \frac{2 K_f}{r}\,,
$$
which implies that,
\begin{equation}\label{2:E15}
\int_{-\infty}^{0}\left|\frac{\tilde{f}^{-}(\varphi+s,\tau+s)}{\mu(\varphi+s,\tau+s)}\right|ds\leq \frac{2 K_f K_{p-3}}{r}\frac{1}{\left|\tau^{p-3}e^{i\epsilon(\tau-\varphi)}\right|}\,.
\end{equation}
Similar to the proof of Proposition \ref{2:propositionDxf}, for $p\geq4$ the integral converges uniformly in $\mathbb{T}_h\times\tilde{D}^{-}_r$. Hence, it defines an analytic function in $\mathbb{T}_h\times\tilde{D}^{-}_r$. The continuity on the closure of $\mathbb{T}_h\times\tilde{D}^{-}_r$ also follows from uniform convergence and continuity of $\tilde{f}^{-}$. In an analogous way we conclude that $$
(\varphi,\tau)\mapsto\int_{0}^{+\infty}\frac{\tilde{f}^{-}(\varphi+s,\tau+s)}{\mu(\varphi+s,\tau+s)}ds
$$ 
is analytic in $\mathbb{T}_h\times\tilde{D}^{+}_r$, continuous on the closure of $\mathbb{T}_h\times\tilde{D}^{+}_r$ and having the same upper bound \eqref{2:E15}. Putting these upper bounds together we obtain,
\begin{equation*}
\left|u(\varphi,\tau)\right|\leq \frac{4 K_f K_{p-3}}{r}\frac{1}{\left|\tau^{p-3}e^{i\epsilon(\tau-\varphi)}\right|}
\end{equation*}
and the proof is complete.
\end{proof}

\subsection{Linear operator $\mathcal{L}$}\label{1:linearoperators}
Let $B\subset\mathbb{C}$ be an open set which does not intersect a neighborhood of the origin. Both sets $D_r^{\pm}$ and their intersection satisfy this condition for $r$ sufficiently large. Let $p\in\mathbb{Z}$ and denote by $\mathfrak{X}_p\left(\mathbb{T}_h\times B\right)$ the space of analytic functions $f=(f_1,\ldots,f_4) : \mathbb{T}_h\times B \rightarrow \mathbb{C}^4$ which have continuous extension to the closure of its domain and have finite norm,
\begin{equation*}
\begin{split}
\left\|f\right\|_p = \sup_{(\varphi,\tau) \in \mathbb{T}_h\times B}\left(\left|\tau^{p+1}f_1(\varphi,\tau)\right|\right.&+\left|\tau^{p+1}f_2(\varphi,\tau)\right|\\
&\left.+\left|\tau^{p}f_3(\varphi,\tau)\right|+\left|\tau^{p}f_4(\varphi,\tau)\right|\right) < \infty.
\end{split}
\end{equation*}
The space $\mathfrak{X}_p\left(\mathbb{T}_h\times B\right)$ endowed with the norm $\left\|\cdot\right\|_p$ as defined above is a complex Banach space. When $f\in\mathfrak{X}_p\left(\mathbb{T}_h\times B\right)$ we occasionally write
$$
f(\varphi,\tau)=(\tau^{-p-1}f_1(\varphi,\tau),\tau^{-p-1}f_2(\varphi,\tau),\tau^{-p}f_3(\varphi,\tau),\tau^{-p}f_4(\varphi,\tau)),
$$ 
where the norm of $f$ is now $\displaystyle \left\|f\right\|_p=\sup_{(\varphi,\tau)}\sum_{i=1}^4\left|f_i(\varphi,\tau)\right|$.

For $\mu>0$ let $\mathfrak{Y}_\mu(\mathbb{T}_h\times B)$ be the space of analytic functions $\xi=(\xi_1,\ldots,\xi_4) : \mathbb{T}_h\times B \rightarrow \mathbb{C}^4$ which have continuous extension to the closure of its domain and have finite norm,
\begin{equation*}
\left\|\xi\right\|_\mu = \sup_{(\varphi,\tau) \in \mathbb{T}_h\times B}\sum_{i=1}^4\left|e^{\mu i(\tau-\varphi)}\xi_i(\varphi,\tau)\right| < \infty.
\end{equation*}

Given two Banach spaces $(\mathfrak{X},\left\|\cdot\right\|_{\mathfrak{X}})$ and $(\mathfrak{Y},\left\|\cdot\right\|_{\mathfrak{Y}})$ we define the usual norm on the space of linear operators $\mathcal{L}:\mathfrak{X}\rightarrow\mathfrak{Y}$ as follows,
\begin{equation*}
\left\|\mathcal{L}\right\|_{\mathfrak{Y},\mathfrak{X}}=\sup_{\xi\in\mathfrak{X}\backslash\left\{0\right\}}\frac{\left\|\mathcal{L}(\xi)\right\|_{\mathfrak{Y}}}{\left\|\xi\right\|_{\mathfrak{X}}}.
\end{equation*}

To simplify the notation we will not write, when it is clear from the context, the dependence of the Banach spaces from the domains where the functions are defined. Moreover, we will write the norm of a linear operator $\mathcal{L}:\mathfrak{X}_p\rightarrow\mathfrak{X}_q$ as $\left\|\mathcal{L}\right\|_{q,p}$ and the norm of a linear operator $\mathcal{L}:\mathfrak{Y}_\mu\rightarrow\mathfrak{Y}_{\mu'}$ as $\left\|\mathcal{L}\right\|_{\mu',\mu}$. 

Let $A:\mathbb{T}_h\times B\rightarrow\mathbb{C}^{4\times 4}$ be an analytic matrix-valued function and define $\mathcal{L}:\mathfrak{X}_p\rightarrow\mathfrak{X}_p$ according to,
\begin{equation}\label{2:linearoperatorLA}
\mathcal{L}(\xi)(\varphi,\tau)=\mathcal{D}\xi(\varphi,\tau)-A(\varphi,\tau)\xi(\varphi,\tau),
\end{equation}
where $\mathcal{D}=\partial_\varphi+\partial_\tau$ is the same differential operator defined in the previous section. 
We say that a 4-by-4 matrix-valued function $\mathbf{U}:\mathbb{T}_h\times B\rightarrow\mathbb{C}^{4\times 4}$ is a \textit{fundamental matrix} of $\mathcal{L}$ if $\mathcal{L}(\mathbf{U})=0$, $\det(\mathbf{U})=1$ and the columns $(\mathbf{u}_i)_i$ of $\mathbf{U}$ satisfy $\mathbf{u}_1\in\mathfrak{X}_{-2}$, $\mathbf{u}_2\in\mathfrak{X}_{-3}$, $\mathbf{u}_3\in\mathfrak{X}_{1}$ and $\mathbf{u}_4\in\mathfrak{X}_{2}$.
We also define,
\begin{equation}\label{2:KU}
K_{\mathbf{U}}:=\max\left\{\left\|\mathbf{u}_1\right\|_{-2},\left\|\mathbf{u}_2\right\|_{-3},\left\|\mathbf{u}_3\right\|_{1},\left\|\mathbf{u}_4\right\|_{2}\right\}.
\end{equation}
In the following we will be concerned with the problem of solving equation $\mathcal{L}(\xi)=f$ for a given analytic function $f:\mathbb{T}_h\times B\rightarrow\mathbb{C}^4$ with some prescribed behavior. In other words, we want to invert the linear operator $\mathcal{L}$ in the Banach spaces defined above. To that end, knowing a fundamental matrix $\mathbf{U}$ for $\mathcal{L}$ we can use the method of variation of constants as follows: let $\xi=\mathbf{U}\mathbf{c}$ where $\mathbf{c}:\mathbb{T}_h\times B\rightarrow\mathbb{C}^4$ is analytic. Substituting into $\mathcal{L}(\xi)$ we get,
\begin{equation*}
\begin{split}
\mathcal{L}(\xi)&=\mathcal{D}\left(\mathbf{U}\mathbf{c}\right)-A\mathbf{U}\mathbf{c}\\
&=(\mathcal{D}\mathbf{U})\mathbf{c}+\mathbf{U}\mathcal{D}\mathbf{c}-A\mathbf{U}\mathbf{c}\\
&=\left(\mathcal{D}\mathbf{U}-A\mathbf{U}\right)\mathbf{c} + \mathbf{U}\mathcal{D}\mathbf{c}\\
&=\mathbf{U}\mathcal{D}\mathbf{c}\,.
\end{split}
\end{equation*}
Note that $\mathbf{U}$ has determinant equal one, hence invertible. Thus $\xi=\mathbf{U}\mathbf{c}$ is a solution of equation $\mathcal{L}(\xi)=f$ provided $\mathbf{c}$ satisfies the equation,
\begin{equation}\label{2:equationDcUf}
\mathcal{D}\mathbf{c}=\mathbf{U}^{-1}f.
\end{equation}
A simple computation shows that we can write,
\begin{equation}\label{2:inverseofUbreve}
\mathbf{U}^{-1}=
\begin{pmatrix}
\tau^{-1} u_{1,1}&\tau^{-1} u_{1,2}&\tau^{-2} u_{1,3}&\tau^{-2} u_{1,4}\\
\tau^{-2} u_{2,1}&\tau^{-2} u_{2,2}&\tau^{-3} u_{2,3}&\tau^{-3} u_{2,4}\\
\tau^{2} u_{3,1}&\tau^{2} u_{3,2}&\tau u_{3,3}&\tau u_{3,4}\\
\tau^{3} u_{4,1}&\tau^{3} u_{4,2}&\tau^{2} u_{4,3}&\tau^{2} u_{4,4}
\end{pmatrix}
\end{equation}
for some functions $u_{i,j}:\mathbb{T}_h\times B\rightarrow\mathbb{C}$, analytic with continuous extension to the closure of $\mathbb{T}_h\times B$. Moreover,
\begin{equation}\label{2:KUminus}
K_{\mathbf{U}^{-1}}:=\max_{i,j}\left\{\sup_{(\varphi,\tau)\in \mathbb{T}_h\times B}\left|u_{i,j}(\varphi,\tau)\right|\right\}<\infty\,.
\end{equation}
Depending on the sets where $\mathbf{U}$ and $f$ are analytic we can use Propositions \ref{2:propositionDxf} and \ref{2:propositionDxf1} to obtain a solution of \eqref{2:equationDcUf}, thus constructing a right inverse for $\mathcal{L}$. Before stating and proving a couple of theorems that make the previous discussion precise, let us present an example that motivates the definition of $\mathcal{L}$ and its fundamental matrix.
\subsubsection{An example: $\mathcal{L}_0$}\label{1:subsectionL0}
Here we define a linear operator $\mathcal{L}_0$ in the form of \eqref{2:linearoperatorLA}. This linear operator plays an important role in the perturbation theory developed in the subsequent sections. Let us consider the following PDE,
\begin{equation}\label{2:eqDXH0}
\mathcal{D}\mathbf{x}=X_{H_0}(\mathbf{x}),
\end{equation}
where $H_0$ denotes the leading order of $H$ which we recall for convenience $$H_0=-I_1 + I_2+\eta I_3^2\,.$$ A direct computation shows that,
\begin{equation}\label{2:Gamma0}
\mathbf{\Gamma}_0(\varphi,\tau)=\left(\kappa\tau^{-2}\cos\varphi,\kappa\tau^{-2}\sin\varphi,\kappa\tau^{-1}\cos\varphi,\kappa\tau^{-1}\sin\varphi\right)^{T},
\end{equation}
solves equation \eqref{2:eqDXH0} where $\kappa^2=-\frac{2}{\eta}$. Indeed, using the polar coordinates,
\begin{equation*}
q_1=R \cos\theta,\quad p_1=r \cos\theta,\quad q_2=R \sin\theta,\quad p_2=r \sin\theta.
\end{equation*}
we see that equation \eqref{2:eqDXH0} reduces to the following equations,
\begin{equation*}
\mathcal{D}\theta=1,\quad \mathcal{D}r=-R,\quad \mathcal{D}R=\eta r^3.
\end{equation*}
The last two equations define a second order differential equation $\mathcal{D}^2r=-\eta r^3$ which has a solution $r(\varphi,\tau)=\frac{\kappa}{\tau}$. Thus $R(\varphi,\tau)=\frac{\kappa}{\tau^2}$. Now using $\theta(\varphi,\tau)=\varphi$ as a solution of the first equation we get the desired solution $\mathbf{\Gamma}_0$. The linearized Hamiltonian vector field $A_0:=DX_{H_0}(\mathbf{\Gamma}_0)$ evaluated at $\mathbf{\Gamma}_0$ reads,
\begin{equation}\label{2:A0}
A_0(\varphi,\tau)=\begin{pmatrix}
0&-1&-\frac{1+2\cos^2\varphi}{\tau^{2}}&-\frac{\sin(2\varphi)}{\tau^2}\\
1&0&-\frac{\sin(2\varphi)}{\tau^2}&-\frac{1+2\sin^2\varphi}{\tau^{2}}\\
-1&0&0&-1\\
0&-1&1&0
\end{pmatrix}.
\end{equation}
Note that $A_0$ does not depend on the choice of $\kappa$. Moreover $A_0\colon \mathbb{T}_h \times \mathbb{C}^*\to\mathbb{C}^{4\times 4}$ is analytic. Define $\mathcal{L}_0:\mathfrak{X}_{1}\rightarrow\mathfrak{X}_{1}$ by
\begin{equation}\label{2:LA0}
\mathcal{L}_0(\xi)(\varphi,\tau)=\mathcal{D}\xi(\varphi,\tau)-A_0(\varphi,\tau)\xi(\varphi,\tau)\,.
\end{equation}
It can be checked directly (or using the polar coordinates as before) that,
\begin{equation}\label{2:U0}
\mathbf{U}_0(\varphi,\tau)=\begin{pmatrix}-\frac{2\tau\sin\varphi}{3\kappa}&-\frac{3\tau^2\cos\varphi}{5\kappa}&-\frac{\kappa\sin\varphi}{\tau^2}&-\frac{2\kappa\cos\varphi}{\tau^3}\\
\frac{2\tau\cos\varphi}{3\kappa}&-\frac{3\tau^2\sin\varphi}{5\kappa}&\frac{\kappa\cos\varphi}{\tau^2}&-\frac{2\kappa\sin\varphi}{\tau^3}\\
\frac{\tau^2\sin\varphi}{3\kappa}&\frac{\tau^3\cos\varphi}{5\kappa}&-\frac{\kappa\sin\varphi}{\tau}&-\frac{\kappa\cos\varphi}{\tau^2}\\
-\frac{\tau^2\cos\varphi}{3\kappa}&\frac{\tau^3\sin\varphi}{5\kappa}&\frac{\kappa\cos\varphi}{\tau}&-\frac{\kappa\sin\varphi}{\tau^2}
\end{pmatrix},
\end{equation}
is a fundamental matrix for the linear operator $\mathcal{L}_0$. Moreover, $\mathbf{U}_0(\varphi,\tau)$ is symplectic for all $(\varphi,\tau)\in\mathbb{T}_h\times\mathbb{C}^*$. In particular, $\det(\mathbf{U}_0)=1$.

\subsubsection{Inverse theorems for the linear operator $\mathcal{L}$}

\begin{theorem}\label{2:TheoremLminus}
Let $p\geq3$, $r>1$ and suppose that the linear operator $\mathcal{L}:\mathfrak{X}_p(\mathbb{T}_h\times D_r^{-})\rightarrow\mathfrak{X}_p(\mathbb{T}_h\times D_r^{-})$ has a fundamental matrix $\mathbf{U}$. Then $\mathcal{L}$ has trivial kernel. Moreover there exists a unique bounded linear operator $\mathcal{L}^{-1}:\mathfrak{X}_{p+1}(\mathbb{T}_h\times D_r^{-})\rightarrow\mathfrak{X}_p(\mathbb{T}_h\times D_r^{-})$ such that $\mathcal{L}\mathcal{L}^{-1}=\mathrm{Id}$.
\end{theorem}
\begin{proof}
Let us prove the first assertion of the theorem: kernel of $\mathcal{L}$ is trivial. To that end, let $\xi\in\mathfrak{X}_p(\mathbb{T}_h\times D_r^{-})$ such that $\mathcal{L}(\xi)=0$. Then, according to \eqref{2:equationDcUf} we have that $\mathcal{D}\mathbf{c}=0$ where $\mathbf{c}=\mathbf{U}^{-1}\xi$. Applying Lemma \ref{2:LemmaDx0} to each component of $\mathbf{c}$ we conclude that $\mathbf{c}(\varphi,\tau)=\mathbf{c}_0(\tau-\varphi)$ where $\mathbf{c}_0:\mathbb{C}\rightarrow\mathbb{C}^4$ is a $2\pi$-periodic entire function. Moreover, since $\mathbf{c}_0=\mathbf{U}^{-1}\xi$ we can bound $\mathbf{c}_0$ as follows. Let $\xi=(\tau^{-p-1}\xi_1,\tau^{-p-1}\xi_2,\tau^{-p}\xi_3,\tau^{-p}\xi_4)^{T}$. Then \eqref{2:inverseofUbreve} implies that,
\begin{equation}\label{2:cexpressionD0}
\mathbf{c}_0=\left(\tau^{-p-2}\sum_{i=1}^{4}u_{1,i}\xi_i,\tau^{-p-3}\sum_{i=1}^{4}u_{2,i}\xi_i,\tau^{-p+1}\sum_{i=1}^{4}u_{3,i}\xi_i,\tau^{-p+2}\sum_{i=1}^{4}u_{4,i}\xi_i\right)^{T}.
\end{equation}
It follows from \eqref{2:KUminus} that the functions $u_{i,j}$ are bounded. Thus, $\mathbf{c}_0$ is bounded for $p\geq3$. An entire bounded function must be constant by Liouville's theorem. Moreover, since $\mathbf{c}_0(s)\rightarrow 0$ as $\mathrm{Im}\,s\rightarrow \pm\infty$ we conclude that $\mathbf{c}_0=0$, thus proving that the kernel of $\mathcal{L}$ is trivial.

Now let us construct an inverse of $\mathcal{L}$, i.e., solve equation $\mathcal{L}(\xi)=f$, where $f\in\mathfrak{X}_{p+1}(\mathbb{T}_h\times D_r^{-})$. Let $\xi=\mathbf{U}\mathbf{c}$. Then $\mathbf{c}$ must satisfy,
\begin{equation}\label{2:Lxif2}
\mathcal{D}\mathbf{c}=\mathbf{U}^{-1}f.
\end{equation}
Let $f=(\tau^{-p-2}f_1,\tau^{-p-2}f_2,\tau^{-p-1}f_3,\tau^{-p-1}f_4)^{T}$ and $g=\mathbf{U}^{-1}f$. Taking into account \eqref{2:inverseofUbreve} we can write
\begin{equation*}
g=\left(\tau^{-p-3}\sum_{i=1}^{4}u_{1,i}f_i,\tau^{-p-4}\sum_{i=1}^{4}u_{2,i}f_i,\tau^{-p}\sum_{i=1}^{4}u_{3,i}f_i,\tau^{-p+1}\sum_{i=1}^{4}u_{4,i}f_i\right)^{T}.
\end{equation*}
Bearing in mind that $\left\|f\right\|_{p+1}<\infty$ and \eqref{2:KUminus} we can bound the components of $g$ as follows,
\begin{equation*}\begin{aligned}
\left|g_1(\varphi,\tau)\right|&\leq\frac{K_{\mathbf{U}^{-1}}\left\|f\right\|_{p+1}}{\left|\tau\right|^{p+3}},&\quad \left|g_2(\varphi,\tau)\right|&\leq\frac{K_{\mathbf{U}^{-1}}\left\|f\right\|_{p+1}}{\left|\tau\right|^{p+4}},\\
\left|g_3(\varphi,\tau)\right|&\leq\frac{K_{\mathbf{U}^{-1}}\left\|f\right\|_{p+1}}{\left|\tau\right|^p},&\quad \left|g_4(\varphi,\tau)\right|&\leq\frac{K_{\mathbf{U}^{-1}}\left\|f\right\|_{p+1}}{\left|\tau\right|^{p-1}}.
\end{aligned}
\end{equation*}
For $p\geq3$ we can apply Proposition \ref{2:propositionDxf} to each component of equation \eqref{2:Lxif2} and conclude that there exists an analytic vector-valued function $\mathbf{c}=(c_1,c_2,c_3,c_4):\mathbb{T}_h\times D_r^{-}\rightarrow\mathbb{C}^4$, continuous in the closure of $\mathbb{T}_h\times D_r^{-}$ such that,
\begin{equation*}\begin{aligned}
\left|c_1(\varphi,\tau)\right|&\leq\frac{K_{p+2}K_{\mathbf{U}^{-1}}\left\|f\right\|_{p+1}}{\left|\tau\right|^{p+2}},&\quad \left|c_2(\varphi,\tau)\right|&\leq\frac{K_{p+3}K_{\mathbf{U}^{-1}}\left\|f\right\|_{p+1}}{\left|\tau\right|^{p+3}},\\
\left|c_3(\varphi,\tau)\right|&\leq\frac{K_{p-1}K_{\mathbf{U}^{-1}}\left\|f\right\|_{p+1}}{\left|\tau\right|^{p-1}},&\quad \left|c_4(\varphi,\tau)\right|&\leq\frac{K_{p-2}K_{\mathbf{U}^{-1}}\left\|f\right\|_{p+1}}{\left|\tau\right|^{p-2}}.
\end{aligned}
\end{equation*}
Finally, define the linear operator $\mathcal{L}^{-1}$ as $\mathcal{L}^{-1}(f)=\xi$ where $\xi=\mathbf{U}\mathbf{c}$. Using the previous estimates we obtain the following upper bounds for the components of $\xi$:
\begin{equation*}\begin{aligned}
\left|\xi_1(\varphi,\tau)\right|&\leq\frac{\bar{K}}{\left|\tau\right|^{p+1}}\left\|f\right\|_{p+1},&\quad \left|\xi_2(\varphi,\tau)\right|&\leq\frac{\bar{K}}{\left|\tau\right|^{p+1}}\left\|f\right\|_{p+1},\\
\left|\xi_3(\varphi,\tau)\right|&\leq\frac{\bar{K}}{\left|\tau\right|^{p}}\left\|f\right\|_{p+1},&\quad \left|\xi_4(\varphi,\tau)\right|&\leq\frac{\bar{K}}{\left|\tau\right|^{p}}\left\|f\right\|_{p+1},
\end{aligned}
\end{equation*}
where $\bar{K}=(K_{p-1}+K_{p+3}+K_{p+2}+K_{p-2})K_{\mathbf{U}}K_{\mathbf{U}^{-1}}$. Consequently $\left\|\xi\right\|_p\leq\bar{K}\left\|f\right\|_{p+1}$ yielding $\left\|\mathcal{L}^{-1}\right\|_{p,p+1}\leq\bar{K}$. Thus $$\mathcal{L}^{-1}:\mathfrak{X}_{p+1}(\mathbb{T}_h\times D_r^{-})\rightarrow\mathfrak{X}_p(\mathbb{T}_h\times D_r^{-})$$ is a bounded right inverse for $\mathcal{L}$. The uniqueness follows from the kernel of $\mathcal{L}$ being trivial.
\end{proof}

\begin{theorem}\label{2:InverseLD1p}
Let $p\geq3$, $r>\max\left\{2,\frac{2\tan\theta}{1-\tan\theta}\right\}$ and suppose that the linear operator $\mathcal{L}:\mathfrak{X}_p(\mathbb{T}_h\times D_r^{1})\rightarrow\mathfrak{X}_p(\mathbb{T}_h\times D_r^{1})$ has a fundamental matrix $\mathbf{U}$. Then the kernel of $\mathcal{L}$ consists of functions of the form
\begin{equation*}
\mathbf{U}(\varphi,\tau)\mathbf{c}(\tau-\varphi)
\end{equation*} 
where $\mathbf{c}:\left\{s \in \mathbb{C} : \mathrm{Im}\,s<h-r\right\}\rightarrow\mathbb{C}^4$ is analytic, $2\pi$-periodic, continuous in the closure of its domain and $\mathbf{c}(s)\rightarrow 0$ as $\mathrm{Im}\,s\rightarrow -\infty$.
Moreover,
\begin{enumerate}
	\item\label{2:InverseLD1pitem1} there exists a bounded linear operator $\mathcal{L}^{-1}:\mathfrak{X}_{p+3}(\mathbb{T}_h\times D_r^{1})\rightarrow\mathfrak{X}_p(\mathbb{T}_h\times D_r^{1})$ such that $\mathcal{L}\mathcal{L}^{-1}=\mathrm{Id}$, 
	\item\label{2:InverseLD1pitem2} for any $0<\mu'<\mu$ there exists a bounded linear operator $\mathcal{L}_{\mu}^{-1}:\mathfrak{Y}_{\mu}(\mathbb{T}_h\times D_r^{1})\rightarrow\mathfrak{Y}_{\mu'}(\mathbb{T}_h\times D_r^{1})$ such that $\mathcal{L}\mathcal{L}^{-1}_{\mu}=\mathrm{Id}$.
\end{enumerate}
\end{theorem}
\begin{proof}
The proof of the first part of this theorem is almost identical to the previous one except that the functions are now defined in $\mathbb{T}_h\times D_r^{1}$. As before, if $\xi\in\mathfrak{X}_p$ such that $\mathcal{L}(\xi)=0$ then by the method of variation of constants $\mathcal{D}\mathbf{c}=0$ where $\mathbf{c}=\mathbf{U}^{-1}\xi$. Applying Lemma \ref{2:LemmaDx0} to each component of the vector function $\mathbf{c}$, we conclude that $\mathbf{c}(\varphi,\tau)=\mathbf{c}_0(\tau-\varphi)$ where $\mathbf{c}_0:\left\{s \in \mathbb{C} : \mathrm{Im}\,s<h-r\right\}\rightarrow\mathbb{C}^4$ is an analytic, $2\pi$-periodic vector-valued function. Moreover, as in the proof of the previous theorem we conclude that $\mathbf{c}_0(s)\rightarrow0$ as $\mathrm{Im}\,s\rightarrow -\infty$, thus proving the first part of the theorem. For the second part, let us first prove item \eqref{2:InverseLD1pitem1}. We shall construct an inverse of $\mathcal{L}$ by solving the equation $\mathcal{L}(\xi)=f$ where $f\in\mathfrak{X}_{p+3}(\mathbb{T}_h\times D_r^{1})$. Again, we look for a solution using the method of variation of constants. Let $\xi=\mathbf{U}\mathbf{c}$. As before, $\mathbf{c}$ must satisfy
\begin{equation}\label{2:Lxif2b}
\mathcal{D}\mathbf{c}=\mathbf{U}^{-1}f.
\end{equation}
Let $f=(\tau^{-p-4}f_1,\tau^{-p-4}f_2,\tau^{-p-3}f_3,\tau^{-p-3}f_4)$ and $g=\mathbf{U}^{-1}f$. Taking into account \eqref{2:inverseofUbreve} we can write $g$ as follows,
\begin{equation*}
g=\left(\tau^{-p-5}\sum_{i=1}^{4}u_{1,i}f_i,\tau^{-p-6}\sum_{i=1}^{4}u_{2,i}f_i,\tau^{-p-2}\sum_{i=1}^{4}u_{3,i}f_i,\tau^{-p-1}\sum_{i=1}^{4}u_{4,i}f_i\right)^{T}.
\end{equation*}
Bearing in mind that $\left\|f\right\|_{p+3}<\infty$ and \eqref{2:KUminus} we can bound the components of $g$ as follows,
\begin{equation*}\begin{aligned}
\left|g_1(\varphi,\tau)\right|&\leq\frac{K_{\mathbf{U}^{-1}}\left\|f\right\|_{p+3}}{\left|\tau\right|^{p+5}},&\quad \left|g_2(\varphi,\tau)\right|&\leq\frac{K_{\mathbf{U}^{-1}}\left\|f\right\|_{p+3}}{\left|\tau\right|^{p+6}},\\
\left|g_3(\varphi,\tau)\right|&\leq\frac{K_{\mathbf{U}^{-1}}\left\|f\right\|_{p+3}}{\left|\tau\right|^{p+2}},&\quad \left|g_4(\varphi,\tau)\right|&\leq\frac{K_{\mathbf{U}^{-1}}\left\|f\right\|_{p+3}}{\left|\tau\right|^{p+1}}.
\end{aligned}
\end{equation*}
Since $r>\max\left\{2,\frac{2\tan\theta}{1-\tan\theta}\right\}$ we can apply Proposition \ref{2:propositionDxf1} with $\epsilon=0$ and $p\geq3$ to each component of equation \eqref{2:Lxif2b} and conclude that there exists a vector-valued function $\mathbf{c}=(c_1,c_2,c_3,c_4):\mathbb{T}_h\times D_r^{1}\rightarrow\mathbf{C}^4$ such that each $c_i$ is an analytic function in $\mathbb{T}_h\times D_r^{1}$, continuous in the closure of its domain and satisfying,
\begin{equation*}\begin{aligned}
\left|c_1(\varphi,\tau)\right|&\leq\frac{4K_{p+2}K_{\mathbf{U}^{-1}}\left\|f\right\|_{p+3}}{r\left|\tau\right|^{p+2}},&\quad \left|c_2(\varphi,\tau)\right|&\leq\frac{4K_{p+3}K_{\mathbf{U}^{-1}}\left\|f\right\|_{p+3}}{r\left|\tau\right|^{p+3}},\\
\left|c_3(\varphi,\tau)\right|&\leq\frac{4K_{p-1}K_{\mathbf{U}^{-1}}\left\|f\right\|_{p+3}}{r\left|\tau\right|^{p-1}},&\quad \left|c_4(\varphi,\tau)\right|&\leq\frac{4K_{p-2}K_{\mathbf{U}^{-1}}\left\|f\right\|_{p+3}}{r\left|\tau\right|^{p-2}}.
\end{aligned}
\end{equation*}
Finally, as in the proof of the previous theorem, we define the linear operator $\mathcal{L}^{-1}$ as $\mathcal{L}^{-1}(f)=\xi$ where $\xi=\mathbf{U}\mathbf{c}$. If $\xi_i$ denote the components of $\xi$ then $\xi_i$ can be bounded in $\mathbb{T}_h\times D_r^{1}$ in the following way,
\begin{equation*}\begin{aligned}
\left|\xi_1(\varphi,\tau)\right|&\leq\frac{\bar{K}}{\left|\tau\right|^{p}}\left\|f\right\|_{p+3},&\quad \left|\xi_2(\varphi,\tau)\right|&\leq\frac{\bar{K}}{\left|\tau\right|^{p+1}}\left\|f\right\|_{p+3},\\
\left|\xi_3(\varphi,\tau)\right|&\leq\frac{\bar{K}}{\left|\tau\right|^{p+1}}\left\|f\right\|_{p+3},&\quad \left|\xi_4(\varphi,\tau)\right|&\leq\frac{\bar{K}}{\left|\tau\right|^{p}}\left\|f\right\|_{p+3},
\end{aligned}
\end{equation*}
where $\bar{K}=\frac{4}{r}(K_{p-1}+K_{p+3}+K_{p+2}+K_{p-2})K_{\mathbf{U}}K_{\mathbf{U}^{-1}}$. Consequently $\left\|\xi\right\|_p\leq\bar{K}\left\|f\right\|_{p+3}$ yielding $\left\|\mathcal{L}^{-1}\right\|_{p,p+3}\leq\bar{K}$. Thus $\mathcal{L}^{-1}:\mathfrak{X}_{p+3}(\mathbb{T}_h\times D_r^{1})\rightarrow\mathfrak{X}_p(\mathbb{T}_h\times D_r^{1})$ is a bounded right inverse of $\mathcal{L}$. 

In order to prove item \eqref{2:InverseLD1pitem2} let $0<\mu'<\mu$ and consider the problem of solving equation $\mathcal{L}(\xi)=f$ but now with $f\in\mathfrak{Y}_{\mu}(\mathbb{T}_h\times D_r^{1})\subset\mathfrak{X}_p(\mathbb{T}_h\times D_r^{1})$ where the inclusion clearly holds for any $p\in\mathbb{Z}$. Again, we look for a solution using the method of variation of constants. Thus we have to solve equation \eqref{2:Lxif2b} where $f$ can now be written as $f(\varphi,\tau)=e^{-\mu i(\tau-\varphi)}\tilde{f}(\varphi,\tau)$ where $\tilde{f}$ is a bounded analytic function in $\mathbb{T}_h\times D_r^{1}$. Taking into account \eqref{2:inverseofUbreve} we can bound the components of $g:=\mathbf{U}^{-1}f$ in $\mathbb{T}_h\times D_r^1$  as follows,
\begin{equation*}
\left|g_i(\varphi,\tau)\right|\leq\sup_{(\varphi,\tau)\in \mathbb{T}_h\times D_r^1}\left|\tau^9e^{-(\mu-\mu')i(\tau-\varphi)}\right|\frac{\left\|f\right\|_{\mu}K_{\mathbf{U}^{-1}}}{\left|\tau^6e^{\mu'i(\tau-\varphi)}\right|},\quad i=1,\ldots,4\,.
\end{equation*}
Note that the supremum in the previous estimate is finite since $\mu-\mu'>0$. So we can again apply Proposition \ref{2:propositionDxf1} with $\epsilon=\mu'$ and $p=6$ to each component of equation \eqref{2:Lxif2b} and conclude that there exists an analytic vector-valued function $\mathbf{c}=(c_1,c_2,c_3,c_4):\mathbb{T}_h\times D_r^{1}\rightarrow\mathbb{C}^4$, continuous in the closure of its domain such that,
\begin{equation}\label{2:estimateforci}
\left|c_i(\varphi,\tau)\right|\leq \frac{K_{\mathbf{c}}}{\left|\tau^3e^{\mu'i(\tau-\varphi)}\right|}\left\|f\right\|_{\mu},\quad i=1,\ldots,4,
\end{equation}
where,
\begin{equation*}
K_{\mathbf{c}}=\frac{4\sup_{(\varphi,\tau)\in \mathbb{T}_h\times D_r^1}\left|\tau^9e^{-(\mu-\mu')i(\tau-\varphi)}\right|K_{\mathbf{U}^{-1}}K_3}{r}.
\end{equation*}
As before, we define the linear operator $\mathcal{L}_{\mu'}^{-1}$ as $\mathcal{L}_{\mu}^{-1}(f)=\xi$ where $\xi=\mathbf{U}\mathbf{c}$. Moreover, taking into account the estimate \eqref{2:estimateforci} the $\xi_i$'s  can be bounded in $\mathbb{T}_h\times D_r^{1}$ as follows,
\begin{equation*}
\left|\xi_i(\varphi,\tau)\right|\leq\frac{4K_{\mathbf{U}}K_{\mathbf{c}}}{\left|e^{\mu'i(\tau-\varphi)}\right|}\left\|f\right\|_{\mu},\quad i=1,\ldots,4\,.
\end{equation*}
Consequently $\left\|\xi\right\|_{\mu'}\leq16K_{\mathbf{U}}K_{\mathbf{c}}\left\|f\right\|_{\mu}$ yielding $\left\|\mathcal{L}_{\mu}^{-1}\right\|_{\mu,\mu'}\leq16K_{\mathbf{U}}K_{\mathbf{c}}$. Thus $\mathcal{L}_{\mu}^{-1}:\mathfrak{Y}_{\mu}\rightarrow\mathfrak{Y}_{\mu'}$ is the desired bounded right inverse of $\mathcal{L}$. 
\end{proof}

\section{Solutions of a variational equation}\label{2:sectionVariationalEq}
Let $n\geq3$, $\xi\in\mathfrak{X}_{n+4}(\mathbb{T}_h\times D_r^{-})$ and consider the following linear PDE,
\begin{equation}\label{2:variationalequationGammanxi}
\mathcal{D}\mathbf{u}=DX_{H}(\mathbf{\Gamma}_{n+3}+\xi)\mathbf{u}.
\end{equation}
where $\mathbf{\Gamma}_{n+3}$ is a partial sum of the formal separatrix as defined in Remark \ref{2:remarkGamman}. In the following lemma we prove the existence of a fundamental solution of equation \eqref{2:variationalequationGammanxi} that is close to a partial sum of a formal fundamental solution $\hat{\mathbf{U}}$ of the formal variational equation \eqref{2:formalVariationaleqGammaminus}. We shall use this result to prove Proposition \ref{2:normalizedFundamentalMatrix} at the end of the present section.
\begin{lemma}\label{2:FundamentalMatrix}
Let $n\geq 3$ and $\mathbf{U}_n$ be a partial sum of a formal fundamental solution $\hat{\mathbf{U}}$ as defined in Remark \ref{2:remarkUn}. Then there exists $r_0>0$ sufficiently large such that for every $r> r_0$ the equation \eqref{2:variationalequationGammanxi} has a unique analytic fundamental solution $\mathbf{U}:\mathbb{T}_h\times D_{r}^{-}\rightarrow\mathbb{C}^{4\times 4}$ having continuous extension to the closure of its domain, $\mathbf{U}^T J \mathbf{U}=J$ (symplectic) and $\mathbf{U}-\mathbf{U}_n \in \mathfrak{X}_{n+1}^4(\mathbb{T}_h\times D_{r}^{-})$.
\end{lemma} 
\begin{proof}
We look for a solution of equation \eqref{2:variationalequationGammanxi} in the form,
\begin{equation}\label{2:eqUnV}
\mathbf{U}=\mathbf{U}_n+\mathbf{V},
\end{equation}
where $\mathbf{V}:\mathbb{T}_h\times D_{r}^{-}\rightarrow\mathbb{C}^{4\times 4}$ is a 4-by-4 matrix-valued function such that each column of $\mathbf{V}$ belongs to the space $\mathfrak{X}_{n}\left(\mathbb{T}_h\times D_{r}^{-}\right)$ for some $r>0$ (to be chosen later in the proof). Substituting \eqref{2:eqUnV} into the equation \eqref{2:variationalequationGammanxi} we obtain,
\begin{equation*}
\mathcal{D}\mathbf{V}=DX_{H}(\mathbf{\Gamma}_{n+3}+\xi)\mathbf{V}+DX_{H}(\mathbf{\Gamma}_{n+3}+\xi)\mathbf{U}_n-\mathcal{D}\mathbf{U}_n.
\end{equation*}
This last equation can be rewritten in the following form,
\begin{equation}\label{2:E24}
\mathcal{L}_0(\mathbf{V})=\mathbf{B}\mathbf{V}+\mathbf{R}_n,
\end{equation}
where $\mathcal{L}_0$ is defined by formula \eqref{2:LA0}. Moreover
\begin{equation*}
\mathbf{B}=DX_{H}(\mathbf{\Gamma}_{n+3}+\xi)-A_0\quad\mathrm{and}\quad\mathbf{R}_n=DX_{H}(\mathbf{\Gamma}_{n+3}+\xi)\mathbf{U}_n-\mathcal{D}\mathbf{U}_n.
\end{equation*}
Taking into account the definition of $A_0$ (see \eqref{2:A0}) we can write the entries of the matrix $\mathbf{B}$ as follows,
\begin{equation}\label{2:matrixB}
\mathbf{B}=\begin{pmatrix}
\tau^{-2}b_{1,1}&\tau^{-2}b_{1,2}&\tau^{-3}b_{1,3}&\tau^{-3}b_{1,4}\\
\tau^{-2}b_{2,1}&\tau^{-2}b_{2,2}&\tau^{-3}b_{2,3}&\tau^{-3}b_{2,4}\\
\tau^{-1}b_{3,1}&\tau^{-1}b_{3,2}&\tau^{-2}b_{3,3}&\tau^{-2}b_{3,4}\\
\tau^{-1}b_{4,1}&\tau^{-1}b_{4,2}&\tau^{-2}b_{4,3}&\tau^{-2}b_{4,4}
\end{pmatrix},
\end{equation}
where each function $b_{i,j}:\mathbb{T}_h\times D_r^{-}\rightarrow\mathbb{C}$ is analytic and bounded in $\mathbb{T}_h\times D_r^{-}$. Thus, each column of $\mathbf{B}\mathbf{V}$ belongs to $\mathfrak{X}_{n+1}$. On the other hand, Remark \ref{2:remarkUn} implies that each column of $\mathbf{R}_n$ also belongs to $\mathfrak{X}_{n+1}$. Thus, $\mathbf{B}\mathbf{V}+\mathbf{R}_n\in\mathfrak{X}^4_{n+1}$. Since $\mathcal{L}_0$ has a fundamental matrix $\mathbf{U}_0$ given by \eqref{2:U0} we can apply Theorem \ref{2:TheoremLminus} which guarantees the existence of an unique bounded right inverse $\mathcal{L}_0^{-1}:\mathfrak{X}_{n+1}\rightarrow\mathfrak{X}_{n}$ of $\mathcal{L}_0$ for $r>1$. Thus, in order to solve \eqref{2:E24} for $\mathbf{V}$, it is sufficient to find a fixed point of the following operator,
\begin{equation}\label{2:fixedpointeqV}
\mathbf{V}\mapsto \mathcal{L}_0^{-1}\left(\mathbf{B}\mathbf{V}\right)+\mathcal{L}_0^{-1}\left(\mathbf{R}_n\right),
\end{equation}
defined in $\mathfrak{X}^4_{n+1}\left(\mathbb{T}_h\times D_{r}^{-}\right)$ with $r>1$. Note that $\mathbf{B}$ induces a linear operator $\mathcal{B}:\mathfrak{X}_{n}\rightarrow\mathfrak{X}_{n+1}$ naturally defined by $\mathcal{B}(\mathbf{v})=\mathbf{B}\mathbf{v}$. Thus, in order to prove the existence of a fixed point for \eqref{2:fixedpointeqV} it is enough to show that,
\begin{equation}\label{2E22}
\left\|\mathcal{L}_0^{-1}\circ\mathcal{B}\right\|_{n,n}\leq\frac{1}{2},
\end{equation}
for $r>1$ sufficiently large. Indeed, using the previous upper bound one can show that the linear operator defined by \eqref{2:fixedpointeqV} is contracting and an application of the contraction mapping theorem yields the existence and uniqueness of a fixed point $\mathbf{V}\in\mathfrak{X}_{n}^4\left(\mathbb{T}_h\times D_r^{-}\right)$. 

\begin{figure}[t]
  \begin{center}
    \includegraphics[width=2in]{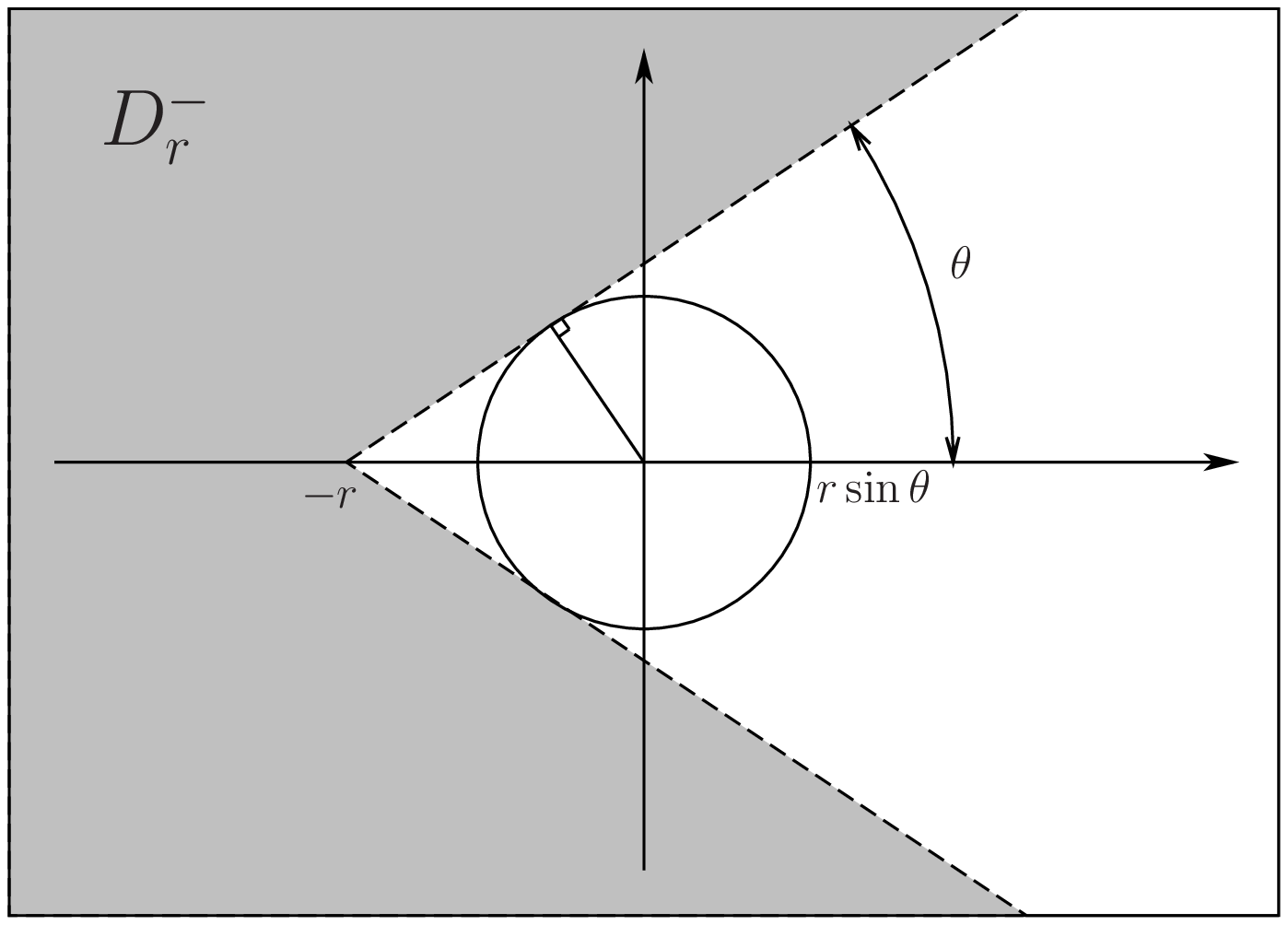}
  \end{center}
  \caption{Domain $D_{r}^{-}$.} 
  \label{2:Drcircle}
\end{figure}

Let us now prove inequality \eqref{2E22}. Given $\mathbf{v}\in\mathfrak{X}_{n}$ we want to bound $\left\|\mathbf{B}\mathbf{v}\right\|_{n+1}$ from above using $\left\|\mathbf{v}\right\|_{n}$. According to \eqref{2:matrixB} we have that,
\begin{equation}\label{2E23}
\mathbf{B}\mathbf{v}=\left(
\tau^{-n-3}\sum_{i=1}^4b_{1,i}v_i,
\tau^{-n-3}\sum_{i=1}^4b_{2,i}v_i,
\tau^{-n-2}\sum_{i=1}^4b_{3,i}v_i,
\tau^{-n-2}\sum_{i=1}^4b_{4,i}v_i\right),
\end{equation}
where $\mathbf{v}=(\tau^{-n-1}v_1,\tau^{-n-1}v_2,\tau^{-n}v_3,\tau^{-n}v_4)$. 
Note that given $r_0>\frac{1}{\sin\theta}$ for every $r>r_0$ we have that $\left|\tau\right|^{-k}\leq\left|\tau\right|^{-1}\leq\frac{1}{r_0\sin\theta}$ in $D_r^{-}$ for $k\in\mathbb{N}$ (see Figure \ref{2:Drcircle}). This observation together with \eqref{2E23} yields
\begin{equation*}
\left\|\mathbf{B}\mathbf{v}\right\|_{n+1}\leq\frac{K_\mathbf{B}}{r_0\sin\theta} \left\|\mathbf{v}\right\|_n,
\end{equation*}
where
\begin{equation*}
K_\mathbf{B}:=\max_{i,j=1,\ldots,4}\left\{\sup_{(\varphi,\tau)\in \mathbb{T}_h\times D_{r}^{-}}\left|b_{i,j}(\varphi,\tau)\right|\right\}<\infty.
\end{equation*}
This proves that the linear operator $\mathcal{B}$ is bounded, $\left\|\mathcal{B}\right\|_{n+1,n}\leq\frac{K_\mathbf{B}}{r_0\sin\theta}$. Now taking into account that $\mathcal{L}_0^{-1}$ is also bounded by Theorem \ref{2:TheoremLminus} we get,
\begin{equation*}
\left\|\mathcal{L}_0^{-1}\circ\mathcal{B}\right\|_{n,n}\leq\left\|\mathcal{L}_0^{-1}\right\|_{n,n+1}\left\|\mathcal{B}\right\|_{n+1,n}\leq\frac{K_\mathbf{B}\left\|\mathcal{L}_0^{-1}\right\|_{n,n+1}}{r_0\sin\theta}.
\end{equation*}
Therefore if $$
r_0>\max\left\{\frac{1}{\sin\theta},\frac{2K_\mathbf{B}\left\|\mathcal{L}_0^{-1}\right\|_{n,n+1}}{\sin\theta_0}\right\},
$$ 
then for every $r>r_0$ the inequality \eqref{2E22} holds. 
Finally, note that we can repeat the previous arguments with $n+1$ instead of $n$ and obtain a unique $\tilde{\mathbf{V}}\in\mathfrak{X}_{n+1}^4\left(\mathbb{T}_h\times D_{\tilde{r}}^{-}\right)$ for $\tilde{r}$ sufficiently large such that $\tilde{\mathbf{U}}=\mathbf{U}_{n+1}+\tilde{\mathbf{V}}$ solves equation \eqref{2:variationalequationGammanxi}. It follows that $\tilde{\mathbf{U}}-\mathbf{U}_{n}\in\mathfrak{X}_{n+1}^4\left(\mathbb{T}_h\times D_{\tilde{r}}^{-}\right)$ and due to the uniqueness of the fixed point we conclude that $\tilde{\mathbf{U}}-\mathbf{U}_{n}=\mathbf{V}$. Thus $\mathbf{V}\in\mathfrak{X}_{n+1}^4\left(\mathbb{T}_h\times D_{r}^{-}\right)$ for every $r$ sufficiently large. In order to conclude the proof of the theorem we just need to show that $\mathbf{U}$ is in fact symplectic. This is not difficult, as it follows from Proposition \ref{2:LemmaDx0}, $\hat{\mathbf{U}}^T J \hat{\mathbf{U}}=J$ and the fact that if $\mathbf{u}$ and $\mathbf{v}$ are columns of $\mathbf{U}$ then $\mathcal{D}\Omega(\mathbf{u},\mathbf{v})=0$.
\end{proof}
Now using the previous lemma we can proof Proposition \ref{2:normalizedFundamentalMatrix}.
\begin{proof}[Proof of Proposition \ref{2:normalizedFundamentalMatrix}]
According to Lemma \ref{2:FundamentalMatrix} we know that for every $n\geq3$ there exists $r_0>0$ such that for every $r>r_0$ there exists a unique fundamental solution $\mathbf{U}$ such that $\mathbf{U}-\mathbf{U}_n\in\mathfrak{X}_{n+1}(\mathbb{T}_h\times D_r^{-})$ and $\mathbf{U}^T J \mathbf{U}= J$. The uniqueness of the solution implies that the third and fourth columns of $\mathbf{U}$ are $\partial_\varphi\mathbf{\Gamma}^-$ and $\partial_\tau\mathbf{\Gamma}^-$ respectively. Thus $\mathbf{U}$ is a normalized fundamental solution. To complete the proof it remains to show that $\mathbf{U}$ is in fact independent of $n$. Indeed for every $n\geq3$, we can trace the proof of Lemma \ref{2:FundamentalMatrix} and see that, by increasing $r$ if necessary, we can make $\left\|\mathbf{U}-\mathbf{U}_3\right\|_3$ as small as we want in order to apply the contraction mapping theorem. Thus, the uniqueness of the fixed point implies that $\mathbf{U}$ is in fact independence of $n$. 
\end{proof}
\section{Proof of Theorem \ref{2:theoremunstableseparatrix}}\label{2:sectionunstableseparatrix}
Let $n\geq6$ and $r>0$ (to be chosen later in the proof). We look for a solution of equation \eqref{2:DH} of the form,
\begin{equation}\label{2:Gammaxi}
\mathbf{\Gamma}^{-}=\mathbf{\Gamma}_n+\xi,
\end{equation}
where $\xi\in\mathfrak{X}_{n}\left(\mathbb{T}_h\times D_r^{-}\right)$ and $\mathbf{\Gamma}_n$ is a partial sum of the formal separatrix as defined in Remark \ref{2:remarkGamman}. Substituting \eqref{2:Gammaxi} into equation \eqref{2:DH} we obtain,
\begin{equation*}
\mathcal{D}\xi=X_{H}(\mathbf{\Gamma}_n+\xi)-\mathcal{D}\mathbf{\Gamma}_n.
\end{equation*}
Now we rewrite the previous equation as follows,
\begin{equation}\label{2:eqLxi}
\mathcal{L}(\xi)=\mathbf{Q}(\xi)+\mathbf{R}_n,
\end{equation}
where $\mathcal{L}:\mathfrak{X}_{n}\rightarrow\mathfrak{X}_{n}$ is a linear operator acting by $\mathcal{L}(\xi)=\mathcal{D}\xi-DX_{H}(\mathbf{\Gamma}_n)\xi$ and
\begin{equation*}
\mathbf{Q}(\xi)=X_H(\mathbf{\Gamma}_n+\xi)-X_H(\mathbf{\Gamma}_n)-DX_H(\mathbf{\Gamma}_n)\xi,\qquad\mathbf{R}_n=X_H(\mathbf{\Gamma}_n)-\mathcal{D}\mathbf{\Gamma}_n.
\end{equation*}
Our goal is to solve equation \eqref{2:eqLxi} with respect to $\xi$. To that end we will invert the linear operator $\mathcal{L}$ and obtain a new equation from which we can apply a fixed point argument to get the desired solution. According to Theorem \ref{2:TheoremLminus} we can invert $\mathcal{L}$ as long as it has a fundamental matrix $\mathbf{U}$. Since $n\geq6$, the existence of a fundamental matrix follows from Theorem \ref{2:FundamentalMatrix}. Thus, there exists an $r_0>1$ such that for every $r>r_0$ the linear operator $\mathcal{L}$ has a fundamental matrix $\mathbf{U}$ such that $\mathbf{U}-\mathbf{U}_{n-3}\in\mathfrak{X}_{n-2}^4$. Hence, we can apply Theorem \ref{2:TheoremLminus} to get a unique bounded linear operator $\mathcal{L}^{-1}:\mathfrak{X}_{n+1}\rightarrow\mathfrak{X}_{n}$ such that $\mathcal{L}\mathcal{L}^{-1}=\mathrm{Id}$. 

Now let us prove that given $\xi\in\mathfrak{X}_{n}\left(\mathbb{T}_h\times D_r^{-}\right)$ the function $\mathbf{Q}(\xi)+\mathbf{R}_n$ belongs to $\mathfrak{X}_{n+1}\left(\mathbb{T}_h\times D_r^{-}\right)$ for $r$ sufficiently large.
First note that Remark \ref{2:remarkGamman} implies that $\mathbf{R}_n \in \mathfrak{X}_{n+1}\left(\mathbb{T}_h\times D_r^{-}\right)$ for any $r>0$. So it remains to show that $\mathbf{Q}(\xi) \in \mathfrak{X}_{n+1}\left(\mathbb{T}_h\times D_r^{-}\right)$ for $r>0$ sufficiently large. Denote the components of the vector field $X_H$ by $v_i$ and consider the following auxiliary functions,
\begin{equation*}
\gamma_i(t)=v_i(\mathbf{\Gamma}_n+t\xi)-v_i(\mathbf{\Gamma}_n)-t\nabla v_i(\mathbf{\Gamma}_n)\xi,\quad i=1,\ldots,4.
\end{equation*}
Note that $\gamma_i(0)=0$ for $i=1,\ldots,4$ and $\mathbf{Q}(\xi)=(\gamma_1(1),\gamma_2(1),\gamma_3(1),\gamma_4(1))^{T}$. We can integrate by parts each function $\gamma_i$ to obtain,
\begin{equation*}
\gamma_i(1)=\int_{0}^{1}(1-s)\gamma_i''(s)ds,\quad i=1,\ldots,4.
\end{equation*}
By the intermediate value theorem there exist $t_i \in [0,1]$ for $i=1,\ldots,4$ such that $\gamma_i(1)=(1-t_i)\gamma_i''(t_i)$ where the second derivative of $\gamma_i$ can be easily computed
\begin{equation}\label{2:gamma2dotestimate}
\gamma_i''(s)=\xi^{T}\left.\mathrm{Hess}\left(v_i\right)\right|_{\mathbf{\Gamma}_n+s\xi}\xi.
\end{equation}
Taking into account that $\xi\in\mathfrak{X}_{n}$ and the fact that $X_H$ is analytic we obtain the following estimate,
\begin{equation*}
\left|\gamma_i(1)\right|\leq 8\left\|H\right\|_{C^3}\left|\tau\right|^{-2n}\left\|\xi\right\|_n^2,
\end{equation*}
for $r>1$ where $\left\|\cdot\right\|_{C^3}$ is the standard $C^3$-norm. 
Using the previous upper bound and the fact that given $r_1>\max\left\{r_0,\frac{1}{\sin\theta}\right\}$ and every $r>r_1$ we have that $\left|\tau\right|^{-2}\leq\left|\tau\right|^{-1}$ for $\tau\in D_r^{-}$, then we can estimate $\left\|\mathbf{Q}(\xi)\right\|_{n+1}$ in the following way,
\begin{equation}\label{2:estimateQxi}
\left\|\mathbf{Q}(\xi)\right\|_{n+1}\leq2^5\left\|H\right\|_{C^3}\left\|\xi\right\|_n^2\sup_{\tau\in D^{-}_{r}}\left|\tau\right|^{-n+2}\leq \frac{2^5\left\|H\right\|_{C^3}\left\|\xi\right\|_n^2}{(r_1\sin\theta)^{n-2}},
\end{equation}
where this last estimate holds since $n\geq6$. Thus $\mathbf{Q}(\xi) \in \mathcal{X}_{n+1}\left(\mathbb{T}_h\times D_r^{-}\right)$ as we wanted to show. 
Now in order to solve equation \eqref{2:eqLxi}, it is sufficient to find a fixed point in $\mathfrak{X}_{n}\left(\mathbb{T}_h\times D_r^{-}\right)$ of the following non-linear operator,
\begin{equation*}
\xi\mapsto\mathcal{L}^{-1}(\mathbf{Q}(\xi))+\mathcal{L}^{-1}(\mathbf{R}_n).
\end{equation*}
Let us denote this operator by $\mathcal{G}$. So in order to apply the contraction mapping theorem we have to check that $\mathcal{G}$ is contracting in some invariant ball
$$
\mathfrak{B}_\rho=\left\{\xi \in \mathfrak{X}_{n}\:\colon\: \left\|\xi\right\|_n\leq\rho\right\},
$$
where $\rho>0$. First we prove that $\mathcal{G}(\mathfrak{B}_\rho)\subseteq\mathfrak{B}_\rho$ for some $\rho>0$. Let $\rho=2\left\|\mathcal{L}^{-1}\right\|_{n,n+1}\left\|\mathbf{R}_n\right\|_{n+1}$ and $\xi \in \mathfrak{B}_\rho$, then \eqref{2:estimateQxi} implies that,
\begin{equation*}\begin{split}
\left\|\mathcal{L}^{-1}(\mathbf{Q}(\xi))-\mathcal{L}^{-1}(\mathbf{R}_n)\right\|_n\leq& \left\|\mathcal{L}^{-1}\right\|_{n,n+1}\left(\frac{2^5\left\|H\right\|_{C^3}\left\|\xi\right\|_n^2}{(r_1\sin\theta)^{n-2}}+\left\|\mathbf{R}_n\right\|_{n+1}\right)\leq \rho,
\end{split}\end{equation*}
provided, 
\begin{equation}\label{2:choiceofr0Gamma}
r_1\geq\frac{(2^6\left\|H\right\|_{C^3}\left\|\mathcal{L}^{-1}\right\|_{n,n+1}\rho)^{\frac{1}{n-2}}}{\sin\theta}.
\end{equation}
Thus $\mathcal{G}$ leaves invariant a closed ball $\mathfrak{B}_\rho$. To check that $\mathcal{G}$ is contracting in $\mathfrak{B}_\rho$ we let $\xi_1,\xi_2 \in \mathfrak{B}_\rho$ and consider a segment connecting both points, i.e., $\gamma_t=(1-t)\xi_1+t\xi_2$. Clearly $\gamma_t \in \mathfrak{B}_\rho$ for all $t\in [0,1]$. Similar as before we define the following auxiliary functions,
\begin{equation*}
\psi_i(t)=v_i(\mathbf{\Gamma}_n+\gamma_t)-v_i(\mathbf{\Gamma}_n)-\nabla v_i(\mathbf{\Gamma}_n)\gamma_t,\quad i=1,\ldots,4.
\end{equation*}
Note that,
\begin{equation*}\begin{split}
\mathbf{Q}(\xi_1)&=(\psi_1(0),\psi_2(0),\psi_3(0),\psi_4(0))^{T},\\
\mathbf{Q}(\xi_2)&=(\psi_1(1),\psi_2(1),\psi_3(1),\psi_4(1))^{T}.
\end{split}\end{equation*}
By the mean value theorem there exist $t_i \in [0,1]$ for $i=1,\ldots,4$ such that $\psi_i(1)-\psi_i(0)=\psi_i'(t_i)$. Differentiating the functions $\psi_i$ we obtain,
\begin{equation}\label{2:differencexis}
\psi_i(1)-\psi_i(0)=\left(\nabla v_i\left(\mathbf{\Gamma}_n+\gamma_{t_i}\right)-\nabla v_i\left(\mathbf{\Gamma}_n\right)\right)\cdot\left(\xi_2-\xi_1\right).
\end{equation}
Thus, we can bound the differences \eqref{2:differencexis} as follows,
\begin{equation*}
\left|\psi_i(1)-\psi_i(0)\right|\leq 8\left\|H\right\|_{C^3}\rho \left|\tau\right|^{-2n}\left\|\xi_2-\xi_1\right\|_n,
\end{equation*}
which implies that,
\begin{equation*}
\left\|\mathbf{Q}(\xi_2)-\mathbf{Q}(\xi_1)\right\|_{n+1}\leq \frac{2^5 \rho \left\|H\right\|_{C^3}}{(r_1\sin \theta)^{n-2}}\left\|\xi_2-\xi_1\right\|_n.
\end{equation*}
Applying the linear operator $\mathcal{L}^{-1}$ and taking into account \eqref{2:choiceofr0Gamma} we get,
\begin{equation*}\begin{split}
\left\|\mathcal{L}^{-1}(\mathbf{Q}(\xi_2)-\mathbf{Q}(\xi_1))\right\|_{n}&\leq\left\|\mathcal{L}^{-1}\right\|_{n,n+1} \frac{2^5 \rho \left\|H\right\|_{C^3}}{(r_1\sin\theta)^{n-2}}\left\|\xi_2-\xi_1\right\|_n\\
&\leq\frac{1}{2}\left\|\xi_2-\xi_1\right\|_n,
\end{split}\end{equation*}
which proves that $\left\|\mathcal{G}(\xi_2)-\mathcal{G}(\xi_1)\right\|_n\leq\frac{1}{2}\left\|\xi_2-\xi_1\right\|_n$ in $\mathfrak{B}_\rho$. Thus $\mathcal{G}$ is contracting in the ball $\mathfrak{B}_\rho$ provided $r>r_1$ where,
\begin{equation*}
r_1> \max\left\{r_0,\frac{1}{\sin\theta},\frac{(2^6\left\|H\right\|_{C^3}\left\|\mathcal{L}^{-1}\right\|_{n,n+1}\rho)^{\frac{1}{n-2}}}{\sin\theta}\right\}.
\end{equation*} 
To conclude the proof of the theorem let us check that the unique function $\mathbf{\Gamma}^{-}$ obtained with $n\geq6$ is in fact independent of $n$. Increasing $r>0$ the distance $\left\|\mathbf{\Gamma}^{-}-\mathbf{\Gamma}_6\right\|_6$ can be made as small as we want in order to apply the contraction mapping theorem for $n=6$. Due to the uniqueness of the fixed point we conclude that the function $\mathbf{\Gamma}^{-}$ is in fact independent of $n$. Finally for every $n\geq0$ there exists $r>0$ sufficiently large such that, 
\begin{equation*}
\mathbf{\Gamma}^{-}-\mathbf{\Gamma}_n=\mathbf{\Gamma}^{-}-\mathbf{\Gamma}_{n+1}+\mathbf{\Gamma}_{n+1}-\mathbf{\Gamma}_n\in\mathfrak{X}_{n+1}(\mathbb{T}_h\times D_r^{-}).
\end{equation*} 
Consequently $\mathbf{\Gamma}^{-}\sim\hat{\mathbf{\Gamma}}$ and the proof is complete.
\qed

\section{Proof of Theorem \ref{2:theoremdeltaasymptoticformula}}\label{2:theoremasympdelta}

Let $\xi_{*}=\mathbf{\Gamma}^{+}-\mathbf{\Gamma}^{-}$. Note that since both $\mathbf{\Gamma}^{\pm}$ have the same asymptotic expansion $\hat{\mathbf{\Gamma}}$ then $\xi_{*}\in\mathfrak{X}_{n}(\mathbb{T}_h\times D_{r}^1)$ for every $n\in\mathbb{N}$ where,
$$
D_r^1=D_r^{+}\cap D_r^{-}\cap \left\{\tau\in\mathbb{C}\:|\: \mathrm{Im}\,\tau<-r\right\}.
$$ 
Let us outline the main steps of the proof. In the first step we write an integral equation for $\xi_*$ and derive, using a fixed point argument, a sequence of functions $\left\{\xi_k\right\}_{k\geq0}$ converging to $\xi_*$. In a second step we prove that the sequence $\left\{\xi_k\right\}_{k\geq0}$ is uniformly bounded (with respect to $k$) by a function that is exponentially small as $\tau\rightarrow\infty$ in $D_r^1$. This is proved by exploiting a recursive equation that is used to define the sequence of functions. In the third and final step of the proof we derive the constant $\Theta^{-}$  and obtain the desired asymptotic formula for $\mathbf{\Gamma}^{+}-\mathbf{\Gamma}^{-}$, thus completing the proof of the theorem. So let us start with,

\bigskip

\textbf{Step 1.} For definiteness let us henceforth suppose that $n=5$. We want to prove the following:

\begin{quotation}
	\textit{For $r>0$ sufficiently large there exists a sequence $\left\{\xi_k\right\}_{k\geq0}$ in $\mathfrak{X}_{5}(\mathbb{T}_h\times D_{r}^1)$ such that $\xi_k\rightarrow\xi_*$ as $k\rightarrow +\infty$.}
\end{quotation}

To prove this we write a fixed point equation for $\xi_*$ and use the contraction mapping theorem. Using the fact that both $\mathbf{\Gamma}^{-}$ and $\mathbf{\Gamma}^{+}$ are solutions of equation \eqref{2:DH} we can write, 
\begin{equation*}
\mathcal{D}\xi_{*}-DX_{H}(\mathbf{\Gamma}^{-})\xi_{*}=X_{H}(\mathbf{\Gamma}^{-}+\xi_{*})-X_H(\mathbf{\Gamma}^{-})-DX_H(\mathbf{\Gamma}^{-})\xi_{*}.
\end{equation*}
Or equivalently,
\begin{equation}\label{2:NonlinearoperatorLxistar}
\mathcal{L}(\xi_{*})=\mathbf{Q}(\xi_{*}),
\end{equation}
where $\mathcal{L}:\mathfrak{X}_{5}(\mathbb{T}_h\times D_r^1)\rightarrow\mathfrak{X}_{5}(\mathbb{T}_h\times D_r^1)$ is the linear operator defined by $\mathcal{L}(\xi)=\mathcal{D}\xi-DX_{H}(\mathbf{\Gamma}^{-})\xi$ and
\begin{equation*} \mathbf{Q}(\xi_{*})=X_{H}(\mathbf{\Gamma}^{-}+\xi_{*})-X_H(\mathbf{\Gamma}^{-})-DX_H(\mathbf{\Gamma}^{-})\xi_{*}.
\end{equation*}
Now we construct a right inverse of $\mathcal{L}$. According to Proposition \ref{2:normalizedFundamentalMatrix} there exists $r_1>0$ and a unique normalized fundamental solution $\mathbf{U}:\mathbb{T}_h\times D_{r_1}^1\rightarrow\mathbb{C}^{4\times4}$ such that $\mathbf{U}\sim\hat{\mathbf{U}}$.  Thus $\mathbf{U}$ is a fundamental matrix for $\mathcal{L}$ provided $r>r_1$. For $r>\max\left\{2,\frac{2\tan\theta}{1-\tan\theta},r_1\right\}$ we can apply Theorem \ref{2:InverseLD1p} which guarantees the existence of a bounded right inverse $\mathcal{L}^{-1}:\mathfrak{X}_{8}(\mathbb{T}_h\times D_r^1)\rightarrow\mathfrak{X}_{5}(\mathbb{T}_h\times D_r^1)$ of $\mathcal{L}$, i.e., $\mathcal{L}\mathcal{L}^{-1}=\mathrm{Id}$.
Moreover, similar estimates as in the proof of Theorem \ref{2:theoremunstableseparatrix} (see \eqref{2:estimateQxi}) show that for $r>0$ sufficiently large we have,
\begin{equation}\label{2:estimateQxin3}
\left\|\mathbf{Q}(\xi_{*})\right\|_{8}\leq \frac{2^5\left\|H\right\|_{C^3}\left\|\xi_{*}\right\|_5^2}{r}.
\end{equation}
Thus $\mathbf{Q}(\xi_{*}) \in \mathfrak{X}_{8}$.
Consequently,
\begin{equation}\label{2:xi0}
\xi_0:=\xi_{*}-\mathcal{L}^{-1}(\mathbf{Q}(\xi_{*})),
\end{equation}
belongs to the kernel of $\mathcal{L}$. According to Theorem \ref{2:InverseLD1p} there exists a $2\pi$-periodic analytic function $\mathbf{c}_0:\mathbb{H}_{r-h}\rightarrow \mathbb{C}^4$, continuous in the closure of its domain, such that $\xi_0(\varphi,\tau)=\mathbf{U}(\varphi,\tau)\mathbf{c}_0(\tau-\varphi)$. The domain of $\mathbf{c}_0$ is a half plane,
\begin{equation*}
\mathbb{H}_{r-h}=\left\{s\in\mathbb{C}\:\colon\:\mathrm{Im}(s)<-r+h\right\}. 
\end{equation*}
Thus \eqref{2:xi0} implies that,
\begin{equation*}
\xi_{*}=\mathcal{L}^{-1}(\mathbf{Q}(\xi_{*}))+\mathbf{U}\mathbf{c}_0,
\end{equation*}
and the function $\xi_{*}$ is a fixed point of the nonlinear operator,
\begin{equation}\label{2:Nonlinearoperatorxi}
\xi\mapsto \mathcal{L}^{-1}(\mathbf{Q}(\xi))+\mathbf{U}\mathbf{c}_0,
\end{equation}
which is defined in $\mathfrak{X}_{5}(\mathbb{T}_h\times D_{r}^1)$. Let $\rho:=2\left\|\mathbf{U}\mathbf{c}_0\right\|_5$. Similar estimates as in the proof of Theorem \ref{2:theoremunstableseparatrix} show that the nonlinear operator defined in \eqref{2:Nonlinearoperatorxi} is contracting in the ball $\mathfrak{B}_\rho=\left\{\xi\in\mathfrak{X}_{5}\:|\:\left\|\xi\right\|_5\leq\rho\right\}$ provided,
\begin{equation*}
r>2^6\left\|\mathcal{L}^{-1}\right\|_{5,8}\left\|H\right\|_{C^3}\rho.
\end{equation*}
Thus, by the contraction mapping theorem, the sequence $\left\{\xi_k\right\}$ defined by,
\begin{equation}\label{2:seqxik}
\xi_{k+1}=\mathcal{L}^{-1}(\mathbf{Q}(\xi_k))+\mathbf{U}\mathbf{c}_0,\qquad k\geq 0,
\end{equation}
converges to $\xi_{*}$, i.e., $\left\|\xi_k-\xi_{*}\right\|_5\rightarrow 0$ as $k\rightarrow\infty$. 

\bigskip

\textbf{Step 2.} It is convenient to estimate the functions $\xi_k$ using the following sup-norm: given a bounded analytic function $g=(g_1,\ldots,g_4):\mathbb{T}_h\times D_r^1\rightarrow\mathbb{C}^4$ let,
\begin{equation}\label{2:defnormg}
\left\|g\right\|=\sup_{(\varphi,\tau)\in \mathbb{T}_h\times D_r^1}\sum_{i=1}^4\left|g_i(\varphi,\tau)\right|.
\end{equation}
In the following we want to prove:

\begin{quotation}
	\textit{There exist $C_*>0$ and $r>0$ sufficiently large such that for every $k\geq0$ we have $\left\|e^{i(\tau-\varphi)}\mathbf{U}^{-1}\xi_k\right\|\leq C_*$.}
\end{quotation}

In order to prove this uniform estimate we define a new sequence of functions:
\begin{equation}\label{2:definitiontildexik}
\zeta_k(\varphi,\tau)=e^{i(\tau-\varphi)}\mathbf{U}^{-1}(\varphi,\tau)\xi_k(\varphi,\tau),\quad\forall k\geq0.
\end{equation}
Let $C_k:=\left\|\zeta_k\right\|$. We want to prove that there exists $C_{*}>0$ and $r>0$ sufficiently large such that $C_k\leq C_*$ for all $k\geq0$. 

To that end, we construct another right inverse of $\mathcal{L}$. Fix arbitrary small positive real numbers $\epsilon,\epsilon'\in\mathbb{R}^+$ such that $\epsilon<\epsilon'$ and define $\mu:=2-\epsilon$ and $\mu':=2-\epsilon'$. Since $0<\mu'<\mu$ we can apply Theorem \ref{2:InverseLD1p} which guarantees the existence of a bounded right inverse $\mathcal{L}_{\mu}^{-1}:\mathfrak{Y}_{\mu}(\mathbb{T}_h\times D_r^1)\rightarrow\mathfrak{Y}_{\mu'}(\mathbb{T}_h\times D_r^1)$ of $\mathcal{L}$. Using \eqref{2:definitiontildexik} and similar estimates as in the proof of the Theorem \ref{2:theoremunstableseparatrix} (see \eqref{2:gamma2dotestimate}) show that the components of $\mathbf{Q}(\xi_k)$ can be bounded by,
\begin{equation*}
2^7\left\|H\right\|_{C^3}K_{\mathbf{U}}^2\left|e^{-2i(\tau-\varphi)}\tau^6\right|C_k^2,
\end{equation*}
in $\mathbb{T}_h\times D_r^1$. Thus,
\begin{equation}\label{2:estimateQxikmu}
\left\|\mathbf{Q}(\xi_k)\right\|_{\mu}=\left\|e^{i(2-\epsilon)(\tau-\varphi)}\mathbf{Q}(\xi_k)\right\|\leq 2^9\left\|H\right\|_{C^3}K_{\mathbf{U}}^2 r^6e^{(h-r)\epsilon}C_k^2,
\end{equation}
for values of $r=O(\epsilon^{-1})$.
Hence $\mathcal{L}^{-1}(\mathbf{Q}(\xi_k))-\mathcal{L}_{\mu'}^{-1}(\mathbf{Q}(\xi_k))$ belongs to the kernel of $\mathcal{L}$ and by Theorem \ref{2:InverseLD1p} we known that there exists a $2\pi$-periodic analytic function $\mathbf{c}_k:\mathbb{H}_{r-h}\rightarrow \mathbb{C}^4$, continuous in the closure of its domain such that,
\begin{equation}\label{2:UckLL}
\mathbf{U}\mathbf{c}_k=\mathcal{L}^{-1}(\mathbf{Q}(\xi_k))-\mathcal{L}_{\mu}^{-1}(\mathbf{Q}(\xi_k)).
\end{equation}
Taking into account \eqref{2:definitiontildexik} and \eqref{2:UckLL} we can rewrite the recursive formula \eqref{2:seqxik} as follows,
\begin{equation}\label{2:seqxikmub}
\zeta_{k+1}=e^{i(\tau-\varphi)}\mathbf{U}^{-1}\mathcal{L}_{\mu}^{-1}(\mathbf{Q}(\xi_k))+e^{i(\tau-\varphi)}\mathbf{c}_k+e^{i(\tau-\varphi)}\mathbf{c}_0.
\end{equation}
In the following we estimate the norm of the functions in the right-hand-side of \eqref{2:seqxikmub}. We will also need the norm induced by \eqref{2:defnormg} on the space of 4-by-4 matrix-valued functions $G=(G_{i,j}):\mathbb{T}_h\times D_r^1\rightarrow\mathbb{C}^{4\times 4}$,
\begin{equation*}
\left\|G\right\|=\max_{j=1,\ldots,4}\sup_{(\varphi,\tau)\in \mathbb{T}_h\times D_r^1}\sum_{i=1}^4\left|G_{i,j}(\varphi,\tau)\right|.
\end{equation*}
Note that given an analytic function $\gamma:D_r^1\rightarrow\mathbb{C}$ such that $\gamma(\tau)=O(\tau^{-3})$ we have,
\begin{equation}\label{2:estimategammaU}
\left\|\gamma\mathbf{U}^{-1}\right\|\leq4K_{\mathbf{U}^{-1}}\sup_{\tau\in D_r^1}\left|\tau^{3}\gamma(\tau)\right|.
\end{equation}
Let us start estimating the norm of the first term in \eqref{2:seqxikmub}. 
Taking into account \eqref{2:estimategammaU} we obtain,
\begin{equation*}\begin{split}
\left\|e^{i(\tau-\varphi)}\mathbf{U}^{-1}\right.&\left.\mathcal{L}_{\mu}^{-1}(\mathbf{Q}(\xi_k))\right\|\leq\left\|e^{-(\mu'-1)i(\tau-\varphi)}\mathbf{U}^{-1}\right\|\left\|e^{\mu' i(\tau-\varphi)}\mathcal{L}_{\mu}^{-1}(\mathbf{Q}(\xi_k))\right\|\\
&\leq 4K_{\mathbf{U}^{-1}}\sup_{(\varphi,\tau)\in \mathbb{T}_h\times D_r^1}\left|\tau^{3}e^{-(\mu'-1)i(\tau-\varphi)}\right|\left\|\mathcal{L}_{\mu}^{-1}(\mathbf{Q}(\xi_k))\right\|_{\mu'}.
\end{split}
\end{equation*}
Thus, \eqref{2:estimateQxikmu} implies that
\begin{equation}\label{2:estimateeLmuprime}
\left\|e^{i(\tau-\varphi)}\mathbf{U}^{-1}\mathcal{L}_{\mu}^{-1}(\mathbf{Q}(\xi_k))\right\|\leq M_1(r) e^{-\frac{1}{2}(r-h)} C_k^2,
\end{equation}
where
\begin{equation}\label{2:constantKbar}
M_1(r)=2^{11}K_{\mathbf{U}^{-1}}K_{\mathbf{U}}^2\left\|\mathcal{L}_{\mu}^{-1}\right\|_{\mu',\mu}\left\|H\right\|_{C^3}r^9e^{-(\frac{1}{2}-(\epsilon'-\epsilon))(r-h)}.
\end{equation}
Clearly $M_1(r)=O(1)$ since $\epsilon'-\epsilon>0$ is arbitrarily small. Now we deal with the second term in equation \eqref{2:seqxikmub}. Taking into account \eqref{2:UckLL} we write,
\begin{equation}\label{2:equationck}
\mathbf{c}_k=\mathbf{U}^{-1}\mathcal{L}^{-1}(\mathbf{Q}(\xi_k))-\mathbf{U}^{-1}\mathcal{L}_{\mu}^{-1}(\mathbf{Q}(\xi_k)).
\end{equation}
Let us estimate each term of \eqref{2:equationck} separately. Using \eqref{2:estimategammaU} we have,
\begin{equation*}\begin{split}
\left\|\mathbf{U}^{-1}\mathcal{L}^{-1}(\mathbf{Q}(\xi_k))\right\|&\leq\left\|\tau^{-5}\mathbf{U}^{-1}\right\|\left\|\tau^5\mathcal{L}^{-1}(\mathbf{Q}(\xi_k))\right\|\\
&\leq4K_{\mathbf{U}^{-1}}\sup_{\tau\in D_r^1}\left|\tau^{-2}\right|\left\|\mathcal{L}^{-1}(\mathbf{Q}(\xi_k))\right\|_5\\
&\leq 4K_{\mathbf{U}^{-1}}\left\|\mathcal{L}^{-1}\right\|_{5,8}\left\|\mathbf{Q}(\xi_k)\right\|_8.
\end{split}
\end{equation*}
Moreover, by \eqref{2:definitiontildexik} we have that $\left\|\xi_k\right\|_5\leq 4K_{\mathbf{U}} r^9e^{-(r-h)}C_k$, which together with \eqref{2:estimateQxin3} imply that,
\begin{equation}\label{2:estimateQxik}
\left\|\mathbf{Q}(\xi_{k})\right\|_{8}\leq 2^9\left\|H\right\|_{C^3}K_{\mathbf{U}}^2 r^{17}e^{2h}e^{-2r}C_k^2.
\end{equation}
Thus,
\begin{equation*}
\left\|\mathbf{U}^{-1}\mathcal{L}^{-1}(\mathbf{Q}(\xi_k))\right\|\leq 2^{11}\left\|H\right\|_{C^3}K_{\mathbf{U}}^2K_{\mathbf{U}^{-1}}\left\|\mathcal{L}^{-1}\right\|_{5,8} r^{17}e^{-2(r-h)}C_k^2.
\end{equation*}
On the other hand, the second term of \eqref{2:equationck} can be estimate as follows,
\begin{equation*}\begin{split}
\left\|\mathbf{U}^{-1}\mathcal{L}_{\mu}^{-1}(\mathbf{Q}(\xi_k))\right\|&\leq\left\|e^{-\mu'i(\tau-\varphi)}\mathbf{U}^{-1}\right\|\left\|e^{\mu'i(\tau-\varphi)}\mathcal{L}^{-1}_{\mu}(\mathbf{Q}(\xi_k))\right\|\\
&\leq 4K_{\mathbf{U}^{-1}}\sup_{(\varphi,\tau)\in \mathbb{T}_h\times D_r^1}\left|\tau^3e^{-(2-\epsilon')i(\tau-\varphi)}\right|\left\|\mathcal{L}_{\mu}^{-1}(\mathbf{Q}(\xi_k))\right\|_{\mu'}\\
&\leq 4K_{\mathbf{U}^{-1}}r^3e^{-(2-\epsilon')(r-h)}\left\|\mathcal{L}_{\mu}^{-1}\right\|_{\mu',\mu}\left\|\mathbf{Q}(\xi_k)\right\|_{\mu}.
\end{split}
\end{equation*}
Taking into account \eqref{2:estimateQxikmu} we get,
\begin{equation*}
\left\|\mathbf{U}^{-1}\mathcal{L}_{\mu}^{-1}(\mathbf{Q}(\xi_k))\right\|\leq 2^{11}\left\|H\right\|_{C^3}K_{\mathbf{U}}^2K_{\mathbf{U}^{-1}}\left\|\mathcal{L}_{\mu}^{-1}\right\|_{\mu',\mu} r^9 e^{-(2-(\epsilon'-\epsilon))(r-h)} C_k^2.
\end{equation*}
Finally, putting all these estimates together we obtain,
\begin{equation*}
\left\|\mathbf{c}_k\right\|\leq M_2(r)e^{-\frac{3}{2}(r-h)}C_k^2,
\end{equation*}
where,
\begin{equation*}\begin{split}
M_2(r)= 2^{11}\left\|H\right\|_{C^3}K_{\mathbf{U}}^2K_{\mathbf{U}^{-1}}\left(\left\|\mathcal{L}_{\mu}^{-1}\right\|_{\mu',\mu}\right.&\left. r^9 e^{-(\frac{1}{2}-(\epsilon'-\epsilon))(r-h)}\right.\\
&\left.+\left\|\mathcal{L}^{-1}\right\|_{5,8} r^{17}e^{-\frac{1}{2}(r-h)}\right).
\end{split}
\end{equation*}
Similar to $M_1$ we conclude that $M_2(r)=O(1)$. In order to conclude the proof of the assertion of this step we need the following simple result.
\begin{lemma}\label{2:lemmaestimatecs}
Let $\sigma>0$ and $f:\mathbb{H}_\sigma\rightarrow\mathbb{C}$ an analytic function, $2\pi$-periodic, continuous in the closure of $\mathbb{H}_\sigma$ and $f(z)\rightarrow 0$ as $\mathrm{Im}\,z\rightarrow-\infty$. Then,
\begin{equation*}
\left|f(z)\right|\leq\sup_{\mathrm{Im}\,z=-\sigma}\left|f(z)\right|e^{\mathrm{Im}\,z+\sigma}.
\end{equation*}
\end{lemma}
\begin{proof}
The proof is a simple application of the maximum modulus principle for analytic functions.
\end{proof}
Applying the previous result to each component of $\mathbf{c}_k$ we get,
\begin{equation*}
\left|c_{k,i}(z)\right|\leq\sup_{\mathrm{Im}\,z=-r+h}\left|c_{k,i}(z)\right|e^{\mathrm{Im}\,z+r-h},\quad i=1,\ldots,4.
\end{equation*}
Thus,
\begin{equation}\label{2:estimateck2}
\left\|e^{i(\tau-\varphi)}\mathbf{c}_k\right\|\leq \left\|\mathbf{c}_k\right\|e^{r-h}\leq M_2(r)e^{-\frac{1}{2}(r-h)}C_k^2
\end{equation}
Regarding the last term in the right-hand-side of equation \eqref{2:seqxikmub} we know that by definition $C_0=\left\|e^{i(\tau-\varphi)}\mathbf{c}_0\right\|$. Applying Lemma \ref{2:lemmaestimatecs} we get $C_0<\infty$. Thus, taking norms in both sides of equation \eqref{2:seqxikmub} and using the estimates \eqref{2:estimateeLmuprime} and \eqref{2:estimateck2} we obtain,
\begin{equation}\label{2:recursionCkplus1}
C_{k+1}\leq\left(M_1(r)+M_2(r)\right)e^{-\frac{1}{2}(r-h)}C_k^2+C_0.
\end{equation}
Since both $M_1$ and $M_2$ are bounded with respect to $r$ we can choose $r>0$ sufficiently large such that,
\begin{equation*}
\left(M_1+M_2\right)C_0 e^{-\frac{1}{2}(r-h)}\leq\frac{1}{4},
\end{equation*}
which implies that $C_k\leq C_*$ for all $k\geq0$ where $C_*:=2C_0$.

\bigskip

\textbf{Step 3.} 
In order to finish the proof of the theorem note that the uniform estimate obtained in the previous step implies that $\left\|e^{i(\tau-\varphi)}\mathbf{U}^{-1}\xi_{*}\right\|\leq C_*$. Thus, the estimate \eqref{2:estimateQxikmu} applied to $\xi_{*}$ gives that $\mathbf{Q}(\xi_*)\in\mathfrak{Y}_{\mu}(\mathbb{T}_h\times D_r^1)$. Moreover, as $\xi_{*}-\mathcal{L}_{\mu}^{-1}\left(\mathbf{Q}(\xi_{*})\right)\in \mathrm{Ker}(\mathcal{L})$ there exists an analytic $2\pi$-periodic vector-valued function $\mathbf{c}_*:\mathbb{H}_{r-h}\rightarrow\mathbb{C}^4$ such that $\xi_{*}=\mathbf{U}\mathbf{c}_{*}+\mathcal{L}_{\mu}^{-1}\left(\mathbf{Q}(\xi_{*})\right)$. Since $\mathbf{c}_*(z)\rightarrow 0$ as $\mathrm{Im}\,z\rightarrow -\infty$, we can write its Fourier series as follows,
\begin{equation*}
\mathbf{c}_{*}(z)=\sum_{m=1}^{\infty}\mathbf{c}_{*,m} e^{-im z},
\end{equation*}
where $\mathbf{c}_{*,m}\in\mathbb{C}^4$. Moreover, as $\mathcal{L}_{\mu}^{-1}\left(\mathbf{Q}(\xi_{*})\right)\in\mathfrak{Y}_{\mu'}(\mathbb{T}_h\times D_r^1)$ then,
\begin{equation*}
\xi_{*}(\varphi,\tau)=e^{-i(\tau-\varphi)}\mathbf{U}(\varphi,\tau)\Theta^-+O\left(e^{-(2-\epsilon')i(\tau-\varphi)}\right),
\end{equation*}
where $\Theta^-:=\mathbf{c}_{*,1}$. This completes the proof of the theorem.
\qed

\bigskip

\nocite{*}
\bibliographystyle{cdraifplain}

\end{document}